\documentclass{amsart}
\usepackage{amssymb,stmaryrd}

\newtheorem*{thma}{Theorem~A}
\newtheorem*{thmb}{Theorem~B}
\newtheorem*{thmc}{Theorem~C}
\newtheorem*{thmd}{Theorem~D}
\newtheorem{thm}{Theorem}[section]
\newtheorem{cor}[thm]{Corollary}
\newtheorem{prop}[thm]{Proposition}
\newtheorem{fact}[thm]{Fact}
\newtheorem{lemma}[thm]{Lemma}
\newtheorem{claim}{Claim}[thm]
\newtheorem{subclaim}{Subclaim}[claim]

\theoremstyle{definition}
\newtheorem{defn}[thm]{Definition}
\newtheorem{q}[thm]{Question}

\theoremstyle{remark}
\newtheorem{remark}[thm]{Remark}

\newenvironment{cproof}{\paragraph{\emph{Proof.\,}}}{\hfill{$\boxtimes$}\par\vspace{2mm}}
\newenvironment{scproof}{\paragraph{\emph{Proof.\,}}}{\hfill{$\boxminus$}\par\vspace{2mm}}

\renewcommand\restriction{\mathbin\upharpoonright}
\renewcommand\mid{\mathrel{|}\allowbreak}

\newcommand*\axiomfont[1]{\textsf{\textup{#1}}}
\newcommand\ch{\axiomfont{CH}}
\newcommand\MA{\axiomfont{MA}_{\omega_1}}
\newcommand\MM{\axiomfont{MM}}
\newcommand\PFA{\axiomfont{PFA}}
\newcommand\nsp{\axiomfont{NSP}}
\newcommand\mrp{\axiomfont{MRP}}
\newcommand\pid{\axiomfont{PID}}
\newcommand\zfc{\axiomfont{ZFC}}
\newcommand\gch{\axiomfont{GCH}}
\newcommand\sch{\axiomfont{SCH}}
\newcommand\s{\subseteq}
\newcommand\sq{\sqsubseteq}
\newcommand\forces{\Vdash}
\newcommand\br{\blacktriangleright}
\newcommand\diagonal{\bigtriangleup}

\newcommand\sd{\framebox[2.7mm][l]{$\diamondsuit$}\hspace{0.4mm}{}}

\DeclareMathOperator{\ap}{AP}
\DeclareMathOperator{\sap}{SAP}
\DeclareMathOperator{\cp}{CP}
\DeclareMathOperator{\cspec}{Cspec}
\DeclareMathOperator{\subadditive}{subadditive}

\DeclareMathOperator{\tcf}{tcf}

\DeclareMathOperator{\ind}{ind}
\DeclareMathOperator{\reg}{Reg}
\DeclareMathOperator{\cf}{cf}
\DeclareMathOperator{\cl}{cl}
\DeclareMathOperator{\Tr}{Tr}
\DeclareMathOperator{\tr}{tr}
\DeclareMathOperator{\im}{Im}
\DeclareMathOperator{\otp}{otp}
\DeclareMathOperator{\dom}{dom}
\DeclareMathOperator{\add}{Add}
\DeclareMathOperator{\acc}{acc}
\DeclareMathOperator{\nacc}{nacc}
\DeclareMathOperator{\refl}{Refl}
\DeclareMathOperator{\U}{U}
\DeclareMathOperator{\p}{P}
\DeclareMathOperator{\pr}{Pr}
\DeclareMathOperator{\ssup}{ssup}
\newcommand\inds{\boxminus^{\ind}}
\newcommand\US{\U^{\subadditive}}

\subjclass[2010]{Primary 03E35; Secondary 03E02, 03E05, 06E10}

\author{Chris Lambie-Hanson}
\address{Department of Mathematics and Applied Mathematics, Virginia Commonwealth University,
Richmond, VA 23284, USA}
\urladdr{http://people.vcu.edu/~cblambiehanso}

\author{Assaf Rinot}
\address{Department of Mathematics, Bar-Ilan University, Ramat-Gan 5290002, Israel.}
\urladdr{http://www.assafrinot.com}

\begin{document}
\title[Knaster and friends III]{Knaster and friends III: Subadditive colorings}
\begin{abstract}
  We continue our study of strongly unbounded colorings, this time
  focusing on subadditive maps. In Part I of this
  series, we showed that, for many pairs of infinite cardinals $\theta < \kappa$,
  the existence of a strongly unbounded coloring $c:[\kappa]^2 \rightarrow \theta$
  is a theorem of $\zfc$. Adding the requirement of subadditivity to a strongly
  unbounded coloring is a significant strengthening, though, and here we see that
  in many cases the existence of a \emph{subadditive} strongly unbounded coloring
  $c:[\kappa]^2 \rightarrow \theta$ is independent of $\zfc$. We connect the
  existence of subadditive strongly unbounded colorings with a number of
  other infinitary combinatorial principles, including the narrow system property,
  the existence of $\kappa$-Aronszajn trees with ascent paths, and square principles.
  In particular, we show that the existence of a closed, subadditive, strongly
  unbounded coloring $c:[\kappa]^2 \rightarrow \theta$ is equivalent to a certain
  weak indexed square principle $\inds(\kappa, \theta)$. We conclude the paper with
  an application to the failure of the infinite productivity of $\kappa$-stationarily
  layered posets, answering a question of Cox.
\end{abstract}

\maketitle
\section{Introduction}

For infinite regular cardinals $\theta < \kappa$, the positive partition relation
$\kappa \rightarrow (\kappa)^2_\theta$, which asserts that every coloring
$c:[\kappa]^2 \rightarrow \theta$ has a homogeneous set of cardinality $\kappa$,
is equivalent to $\kappa$ being weakly compact. For non-weakly-compact cardinals
$\kappa$, though, one can seek to measure the incompactness of $\kappa$ by
asking whether certain strengthenings of the negative relation $\kappa \nrightarrow
(\kappa)^2_\theta$ hold. One natural such strenthening is to require that there
exist colorings $c:[\kappa]^2 \rightarrow \theta$ witnessing certain strong
unboundedness properties. In \cite{paper34}, which forms Part~I of this series of papers,
the authors introduce the following
coloring principle, which asserts the existence of such strongly unbounded colorings, and
use it to answer questions about the infinite productivity of the $\kappa$-Knaster
condition for uncountable $\kappa$.

\begin{defn}
  $\U(\kappa, \mu, \theta, \chi)$ asserts the existence of a coloring $c:[\kappa]^2
  \rightarrow \theta$ such that for every $\sigma < \chi$, every pairwise disjoint subfamily
  $\mathcal{A} \s [\kappa]^{\sigma}$ of size $\kappa$,
  and every $i < \theta$, there exists $\mathcal{B} \in [\mathcal{A}]^\mu$
  such that $\min(c[a \times b]) > i$ for all $(a, b) \in [\mathcal{B}]^2$.
\end{defn}

Much of \cite{paper34} is devoted to analyzing situations in which
$\U(\ldots)$ necessarily holds and, moreover, is witnessed
by \emph{closed} or \emph{somewhere-closed} colorings (see Definition~\ref{def21} below).
In Part~II of this series \cite{paper35}, we studied $\cspec(\kappa)$,
the \emph{$C$-sequence spectrum of $\kappa$} (see Definition~\ref{cspecdef} below),
which is another measure
of the incompactness of $\kappa$, and found some unexpected connections between $\cspec(\kappa)$
and the validity of instances of $\U(\kappa,\ldots)$.

In this paper, which can be read largely independently
of \cite{paper34,paper35}, we investigate \emph{subadditive} witnesses to $\U(\ldots)$.
\begin{defn}\label{def12}
A coloring $c:[\kappa]^2\rightarrow \theta$ is
  \emph{subadditive} if, for all $\alpha < \beta < \gamma < \kappa$, the following inequalities
  hold:
  \begin{enumerate}
    \item $c(\alpha, \gamma) \leq \max\{c(\alpha, \beta), c(\beta, \gamma)\}$;
    \item $c(\alpha, \beta) \leq \max\{c(\alpha, \gamma), c(\beta, \gamma)\}$.
  \end{enumerate}
\end{defn}
Adding the requirement of subadditivity significantly strengthens the coloring
principle, and we prove that the existence of closed, subadditive witnesses to
$\U(\ldots)$ is equivalent to a certain
indexed square principle.
Our first main result improves Clause~(1) of \cite[Theorem~A]{paper34}.

\begin{thma} Let $\theta<\kappa$ be a pair of infinite regular cardinals.
The following are equivalent:
  \begin{enumerate}
    \item $\inds(\kappa,\theta)$ holds;
    \item There is a closed, subadditive witness to $\U(\kappa,2,\theta,2)$;
    \item There is a closed, subadditive witness to $\U(\kappa,\kappa,\theta,\sup(\reg(\kappa))$.
  \end{enumerate}

In addition, $\square(\kappa, \sq_{\theta})$ implies (1)--(3).
\end{thma}
We also prove that a version of square with built-in
diamond for a singular cardinal $\lambda$ gives rise to somewhere-closed subadditive witnesses
to $\U(\lambda^+, \ldots)$, which in turn imply that the $C$-sequence spectrum of $\lambda^+$ is rich:

\begin{thmb}
  Suppose that $\lambda$ is a singular cardinal,
  $\vec f$ is a scale for $\lambda$ in some product $\prod\vec\lambda$,
  and $\sd(\vec{\lambda})$ holds.
  Let $\Sigma$ denote the set of good points for $\vec f$.

  Then, for every $\theta\in\reg(\lambda)\setminus(\cf(\lambda)+1)$,
  there exists a $\Sigma$-closed, subadditive witness to $\U(\lambda^+, \lambda^+,
  \theta, \lambda)$.
  In particular, $\reg(\lambda)\s \cspec(\lambda^+)$.
\end{thmb}

For a pair of infinite regular cardinals $\theta<\kappa$ and a coloring
$c:[\kappa]^2\rightarrow\theta$, an interesting facet
of the study of the unboundedness
properties of $c$ is the set $\partial(c)$ of its levels of divergence (see Definition~\ref{levels} below).
Any coloring $c$ for which $\partial(c)$ is stationary is automatically a
somewhere-closed witness to $\U(\kappa,\kappa,\theta,\theta)$.
We prove that the existence of a (fully) closed witness $c$ to $\U(\kappa,\kappa,\theta,\theta)$ for which $\partial(c)$ is stationary is equivalent to the existence of a nonreflecting stationary subset of $E^\kappa_\theta$,
and that the existence of a nonreflecting stationary subset of $E^\kappa_\theta$
does not suffice to yield a subadditive witness to $\U(\kappa,2,\theta,2)$.
We have three main consistency results concerning the characteristic $\partial(c)$:
\begin{thmc}
\begin{enumerate}
\item For any pair of infinite regular cardinals $\theta<\kappa$,
there is a $\kappa$-strategically closed, $\theta^+$-directed closed
forcing notion that adds a subadditive witness $c$ to $\U(\kappa, \kappa,\theta, \theta)$ for which $\partial(c)$ is stationary;
\item For any pair of infinite regular cardinals $\theta<\kappa$,
there is a $\kappa$-strategically closed, $\theta$-directed closed
forcing notion that adds a closed subadditive witness $c$ to $\U(\kappa, \kappa,\theta, \theta)$ for which $\partial(c)$ is stationary;
\item For regular uncountable cardinals $\theta<\lambda<\kappa$ such that $\lambda$ is supercompact and $\kappa$ is weakly compact,
there is a forcing extension in which $\square(\kappa,\theta)$ fails,
yet, there is a closed, subadditive witness $c$ to $\U(\kappa, \kappa, \theta, \theta)$ for which
  $\partial(c)$ is stationary.
\end{enumerate}
\end{thmc}

On the Ramsey-theoretic side, we prove that in the presence of large cardinals,
for many pairs of infinite regular cardinals $\theta<\kappa$,
$\kappa\rightarrow[\kappa]^2_{\theta,\text{finite}}$ holds restricted to the class of subadditive colorings (in particular,
refuting subadditive instances of $\U(\kappa,2,\theta,2)$),
and that similar results hold at small cardinals in forcing extensions or in the presence of forcing axioms.

On the anti-Ramsey-theoretic side,
we have a result reminiscent of the motivating result of \cite{paper34}
concerning the infinite
productivity of strong forms of the $\kappa$-chain condition.
When combined with Theorem~A, the next theorem shows
that $\square(\kappa)$ yields a gallery of counterexamples to productivity of $\kappa$-stationarily layered posets, answering a question of Cox \cite{cox}.

\begin{thmd}
  Suppose that $\theta \leq \chi < \kappa$ are infinite, regular cardinals,
  $\kappa$ is $({<}\chi)$-inaccessible, and there is a closed and subadditive witness $c$ to $\U(\kappa,2,\theta,2)$.
Then there is a sequence of posets $\langle\mathbb{P}_i \mid  i < \theta\rangle$ such that:
  \begin{enumerate}
    \item for all $i < \theta$, $\mathbb{P}_i$ is well-met and $\chi$-directed
      closed with greatest lower bounds;
    \item for all $j < \theta$, $\prod_{i < j} \mathbb{P}_i$ is $\kappa$-stationarily layered;
    \item $\prod_{i < \theta} \mathbb{P}_i$ is not $\kappa$-cc.
  \end{enumerate}

If, in addition, $\partial(c)\cap E^\kappa_\chi$ is stationary, then the sequence $\langle \mathbb P_i\mid i<\theta\rangle$ can be made constant.
\end{thmd}

As a corollary, we get that Magidor's forcing for changing the cofinality of a measurable cardinal $\lambda$ to a regular cardinal $\theta<\lambda$
adds a poset $\mathbb P$ whose $\theta^{th}$ power is not $\lambda^+$-cc,
but all of whose lower powers are $\lambda^+$-stationarily layered.

\subsection{Organization of this paper}

In Section~\ref{prelim_sec}, we present some useful definitions and facts about
$\U(\kappa, \mu, \theta, \chi)$, largely derived from Part I of
this series. We also present a pseudo-inverse to the fact, observed in Part I,
that Shelah's principle $\pr_1(\kappa,\kappa,\theta,\chi)$ implies $\U(\kappa,2,\theta,\chi)$.

In Section~\ref{subadditive_sec}, we review
the notion of \emph{subadditivity} and some of its variations and prove
that any subadditive witness to $\U(\kappa, 2, \theta, 2)$ is in fact a witness
to $\U(\kappa, \mu, \theta, \chi)$ for all $\mu < \kappa$ and
all $\chi \leq \cf(\theta)$ (and, under certain closure
assumptions, even stronger principles). Subsection~\ref{nsp_sec} contains results
connecting subadditive strongly unbounded colorings to narrow systems and trees
with ascent paths. In Subsection~\ref{locally_small_sec}, we discuss \emph{locally
small} colorings of the form $c:[\lambda^+]^2 \rightarrow \cf(\lambda)$, focusing
in particular on the case in which $\lambda$ is a singular cardinal. Locally
small colorings are necessarily witnesses to $\U(\lambda^+, 2, \cf(\lambda), \cf(\lambda))$,
and retain this property in any outer model with the same cardinals. In
Subsection~\ref{levels_sec}, we introduce a subset $\partial(c) \subseteq \kappa$
associated with a coloring $c:[\kappa]^2 \rightarrow \theta$ that is useful in the
analysis of $\U(\kappa, \mu, \theta, \chi)$, particularly in the context of subadditive
colorings. We then introduce a forcing notion that establishes Clause~(1) of
Theorem~C. Subsection~\ref{consistency_sec} contains a number of results indicating
the extent to which various compactness principles place limits on the existence of
certain subadditive witnesses to $\U(\kappa, \mu, \theta, \chi)$. In particular,
it is shown that simultaneous stationary reflection, the existence of highly
complete or indecomposable ultrafilters, and the P-ideal dichotomy all have
such an effect.

In Section~\ref{indexedsquaresection}, we introduce an indexed square principle
$\inds(\kappa, \theta)$ and prove that it is equivalent to the existence of a
closed, subadditive witness to $\U(\kappa, 2, \theta, 2)$,
thereby establishing the first part of Theorem~A. We also prove a consistency result indicating
that $\inds(\kappa, \theta)$ is a proper weakening of $\square^{\ind}(\kappa, \theta)$
and does not even imply $\square(\kappa, \theta)$, in the process proving Clause~(3)
of Theorem~C. Section~\ref{indexedsquaresection} also contains the proof of
the second part of Theorem~A and the proof of Clause~(2) of Theorem~C.

Section~\ref{singular_sec} is concerned with successors of
singular cardinals. We begin by proving Theorem~B,
showing that a certain square with built-in diamond
sequence on a singular cardinal $\lambda$ entails the existence of a subadditive
witness to $\U(\lambda^+, \lambda^+, \theta, \lambda)$ for all
$\theta \in \reg(\lambda)\setminus (\cf(\lambda)+1)$. We then present an improvement
upon a result from Part I of this series proving the existence of closed witnesses
to $\U(\lambda^+, \lambda^+, \theta, \cf(\lambda))$ for all singular $\lambda$
whose cofinality is not greatly Mahlo and all $\theta \leq \cf(\lambda)$.

Section~\ref{stat_layered_sec} deals with the infinity productivity of
$\kappa$-stationarily layered posets and contains our proof of Theorem~D.

\subsection{Notation and conventions}

Throughout the paper, $\kappa$ denotes a regular uncountable cardinal, and $\chi,\theta,$ and $\mu$
denote cardinals $\le\kappa$.
$\lambda$ will always denote an infinite cardinal.
We say that $\kappa$ is \emph{$\chi$-inaccessible} iff, for all $\nu<\kappa$,
$\nu^\chi<\kappa$, and say that $\kappa$ is \emph{$({<}\chi)$-inaccessible} iff, for all $\nu<\kappa$ and $\mu<\chi$,
$\nu^{\mu}<\kappa$.
We denote by  $H_\Upsilon$
the collection of all sets of hereditary cardinality less $\Upsilon$,
where $\Upsilon$ is a regular cardinal
sufficiently large to ensure that all objects of interest are in $H_\Upsilon$.

$\reg$ denotes the class of infinite regular
cardinals, and $\reg(\kappa)$ denotes $\reg \cap \kappa$. $E^\kappa_\chi$
denotes the set $\{\alpha < \kappa \mid \cf(\alpha) = \chi\}$, and
$E^\kappa_{\geq \chi}$, $E^\kappa_{<\chi}$, $E^\kappa_{\neq\chi}$, etc.\ are defined analogously.

For a set of ordinals $a$, we write $\ssup(a) := \sup\{\alpha + 1 \mid
\alpha \in a\}$, $\acc^+(a) := \{\alpha < \ssup(a) \mid \sup(a \cap \alpha) = \alpha > 0\}$,
$\acc(a) := a \cap \acc^+(a)$, $\nacc(a) := a \setminus \acc(a)$,
and $\cl(a):= a\cup\acc^+(a)$.
For sets of ordinals $a$ and $b$, we write $a < b$ if, for all $\alpha \in a$
and all $\beta \in b$, we have $\alpha < \beta$.
For a set of ordinals $a$ and an ordinal $\beta$, we write
$a < \beta$ instead of $a < \{\beta\}$ and $\beta < a$ instead of $\{\beta\} < a$.

For any set $\mathcal A$, we write
$[\mathcal A]^\chi:=\{ \mathcal B\s\mathcal A\mid |\mathcal B|=\chi\}$ and
$[\mathcal A]^{<\chi}:=\{\mathcal B\s\mathcal A\mid |\mathcal B|<\chi\}$.
In particular, $[\mathcal{A}]^2$ consists of all unordered pairs from $\mathcal{A}$.
In some scenarios, we will also be interested in ordered pairs from $\mathcal{A}$.
In particular, if $\mathcal{A}$ is either a set of ordinals or a collection of sets
of ordinals, then we will abuse notation and write $(a,b) \in [\mathcal{A}]^2$
to mean $\{a,b\} \in [\mathcal{A}]^2$ and $a < b$.

\section{Preliminaries} \label{prelim_sec}

In this brief section, we recall a key definition and present a few useful facts
about $\U(\ldots)$.
We start by recalling the following definition
from \cite{paper34} concerning \emph{closed} colorings.

\begin{defn}\label{def21}
  Suppose that $c:[\kappa]^2 \rightarrow \theta$ is a coloring.
  \begin{enumerate}
    \item For all $\beta < \kappa$ and $i \leq \theta$, we let $D^c_{\leq i}(\beta)$
    denote the set $\{\alpha < \beta \mid c(\alpha, \beta) \leq i\}$.
    \item For all $\Sigma \subseteq \kappa$, $c$ is \emph{$\Sigma$-closed}
    if, for all $\beta < \kappa$ and $i \leq \theta$,
    \[
      \acc^+(D^c_{\leq i}(\beta)) \cap \Sigma \subseteq D^c_{\leq i}(\beta).
    \]
    \item $c$ is \emph{somewhere-closed} if it is $\Sigma$-closed for some stationary $\Sigma\s\kappa$.
    \item $c$ is \emph{closed} if it is $\kappa$-closed.
  \end{enumerate}
\end{defn}

The following fact is a useful tool for proving that certain
colorings satisfy strong instances of $\U(\ldots)$.

\begin{fact}[\cite{paper34}]\label{pumpclosed}
  Suppose that $c:[\kappa]^2\rightarrow\theta$ is a coloring and $\omega\le\chi<\kappa$.
  Then $(1)\implies(2)\implies(3)$:
  \begin{enumerate}
    \item For some stationary $\Sigma\s E^\kappa_{\ge\chi}$, $c$ is a $\Sigma$-closed
      witness to $\U(\kappa,2,\theta,\chi)$.
    \item For every family $\mathcal A\s[\kappa]^{<\chi}$ consisting of $\kappa$-many
      pairwise disjoint sets, for every club $D\s\kappa$, and for every $i<\theta$,
      there exist $\gamma\in D$, $a\in\mathcal A$, and $\epsilon < \gamma$ such that:
      \begin{itemize}
        \item $\gamma < a$;
        \item for all $\alpha \in (\epsilon, \gamma)$ and all $\beta\in a$,
        we have $c(\alpha,\beta)>i$.
      \end{itemize}
      \item $c$ witnesses $\U(\kappa,\kappa,\theta,\chi)$.
  \end{enumerate}
\end{fact}

It is sometimes useful to consider the following unbalanced form of $\U(\cdots)$.

\begin{defn}\label{unbalU}
$\U(\kappa, \mu\circledast\nu, \theta, \chi)$ asserts the existence of a coloring $c:[\kappa]^2
  \rightarrow \theta$ such that for every $\sigma < \chi$, every pairwise disjoint subfamily
  $\mathcal{A} \s [\kappa]^{\sigma}$ of size $\kappa$,
  and every $i < \theta$, there exist $\mathcal A'\in[\mathcal A]^\mu$ and $\mathcal B'\in[\mathcal A]^\nu$
  such that, for every $(a,b)\in\mathcal A'\times\mathcal B'$, $a<b$ and
  $\min(c[a \times b]) > i$.
\end{defn}

    \begin{lemma}\label{lemma24} Suppose that $c:[\kappa]^2\rightarrow\theta$ witnesses $\U(\kappa,2,\theta,2)$,
    with $\theta<\kappa$.
    \begin{enumerate}
    \item For every cofinal $A\s\kappa$,
      there exists  $\epsilon<\kappa$ such that
      $$\sup\{\beta\in A\setminus\epsilon\mid \sup\{c(\alpha, \beta) \mid
      \alpha \in A \cap \epsilon\}=\theta\}=\kappa;$$
    \item For every cofinal $A\s\kappa$,
      there exists  $\beta\in A$ such that $\{c(\alpha, \beta) \mid
      \alpha \in A \cap \beta\}$ is unbounded in $\theta$.
      In particular, $c$ witnesses $\U(\kappa,\cf(\theta)\circledast1,\theta,2)$;
      \item For every stationary $S\s\kappa$, there exists $\epsilon<\kappa$ such that, for every $i<\theta$, $\{\beta\in S\mid \epsilon<\beta, ~ c(\epsilon,\beta)>i\}$ is stationary;
      \item For every stationary $S\s\kappa$ and a family of functions $\mathcal H\in[{}^S\theta]^{<\kappa}$,
      there exists $\epsilon\in S$ such that, for every $h\in\mathcal H$,  the following set is stationary:
      $$\{\beta\in S\mid \epsilon<\beta, ~  c(\epsilon,\beta)>\max\{h(\epsilon),h(\beta)\}\}.$$
      \end{enumerate}
    \end{lemma}

    \begin{proof} Clause~(2) follows from Clause~(1),
    and Clause~(3) follows from Clause~(4).

    (1) Towards a contradiction, suppose that $A$ is a counterexample.
      Then, for every $\epsilon<\kappa$, there exists a large enough $\eta_\epsilon\in[\epsilon,\kappa)$
      such that, for every $\beta\in A\setminus\eta_\epsilon$,
      for some $i < \theta$, $A \cap \epsilon \subseteq D^c_{\leq i}(\beta)$.
      Fix a sparse enough cofinal subset $A'\s A$ such that,
      for every $\epsilon\in A'$, $\min(A'\setminus(\epsilon+1))\ge\eta_{\epsilon+1}$.
      It follows that for every $\beta\in A'$, for some $i_\beta<\theta$,
      $A \cap \beta \subseteq D^c_{\leq i_\beta}(\beta)$.
      Fix $i^* < \theta$ and an unbounded $A^* \subseteq A'$ such that
      $i_\beta = i^*$ for all $\beta \in A^*$. Since $c$ witnesses
      $\U(\kappa, 2, \theta, 2)$, we can find $(\alpha, \beta) \in [A^*]^2$
      such that $c(\alpha, \beta) > i^*$, contradicting the fact that
      $\alpha\in A \cap \beta \subseteq D^c_{\leq i^*}(\beta)$.

    (4) Towards a contradiction, suppose that $S$ and $\mathcal H$ form a counterexample.
For every $\epsilon\in S$, fix a function $h_\epsilon\in\mathcal H$
and a club $D_\epsilon\s \kappa$ disjoint from $\{\beta\in S\mid \epsilon<\beta, ~ c(\epsilon,\beta)>\max\{h_\epsilon(\epsilon),h_\epsilon(\beta)\}\}$.
Let $D:=\diagonal_{\epsilon<\kappa}D_\epsilon$.
As $|\mathcal H|<\kappa$, find $h\in\mathcal H$ for which $A:=\{\epsilon\in D\cap S\mid h_\epsilon=h\}$ is cofinal in $\kappa$.
As $\theta<\kappa$, find $i<\theta$ for which $B:=\{\beta\in A\mid h(\beta)=i\}$ is cofinal in $\kappa$.
Now, as $c$ witnesses $\U(\kappa,2,\theta,2)$, we may pick $(\epsilon,\beta)\in[B]^2$ such that $c(\epsilon,\beta)>i$.
But $i=\max\{h_\epsilon(\epsilon),h_\epsilon(\beta)\}$, contradicting the fact that $\beta\in D_\epsilon\cap S$.
    \end{proof}

The following is a corollary to a result from \cite{paper34} which we never took the time to derive.
\begin{prop} If $\kappa$ is a Mahlo cardinal admitting a nonreflecting stationary subset of $\reg(\kappa)$,
then $\U(\kappa,\kappa,\theta,\kappa)$ holds for every $\theta \leq \kappa$.
\end{prop}
\begin{proof}  By \cite[Theorem~4.11]{paper34}.
\end{proof}

By \cite[Lemma~2.3]{paper34}, Shelah's principle
$\pr_1(\kappa,\kappa,\theta,\chi)$ implies $\U(\kappa,2,\theta,\chi)$. Here we deal with a pseudo-inverse:

\begin{prop}  Suppose that $\U(\kappa,2,\theta,\chi)$ holds with $\chi\le\cf(\theta)=\theta\le\theta^{<\theta}<\kappa$.
Then $V^{\add(\theta,1)}\models \pr_1(\kappa,\kappa,\theta,\chi)$.
\end{prop}
\begin{proof} In $V$, let $c:[\kappa]^2\rightarrow\theta$ be a witness to
$\U(\kappa,2,\theta,\chi)$. Let $G$ be $\add(\theta, 1)$-generic over $V$, and work
for now in $V^{\add(\theta,1)}$. Let $g:\theta\rightarrow\theta$
be the Cohen-generic function, and set $d:=g\circ c$. To see that $d$ witnesses $\pr_1(\kappa,\kappa,\theta,\chi)$,
let $\mathcal A$ be a $\kappa$-sized pairwise disjoint subfamily of $[\kappa]^{<\chi}$,
and let $\tau<\theta$; we need to find $(a,b)\in[\mathcal A]^2$ such that
$d[a\times b]=\{\tau\}$. Since $\chi \leq \theta$, every element of $\mathcal{A}$ is
in $V$. Also, letting $\dot{\mathcal{A}} \in V$ be an $\add(\theta, 1)$-name for
$\mathcal{A}$, we can fix for each $a \in \mathcal{A}$ a condition $p_a \in G$ such that
$p_a \Vdash ``a \in \dot{\mathcal{A}}$.
As $\chi\le\theta^{<\theta}<\kappa$, and hence $|\add(\theta, 1)| < \kappa$, by passing to a subfamily if necessary, we may assume that there is a fixed $p^* \in G$
such that $p_a = p^*$ for all $a \in \mathcal{A}$, and therefore $\mathcal A =
\{a \in [\kappa]^{<\chi} \mid p \Vdash ``a \in \dot{\mathcal{A}}"\}$ is in $V$.
We therefore move back to $V$ and run a density argument. Let $p:i\rightarrow\theta$ be an arbitrary condition in $\add(\theta,1)$ below $p^*$. By the hypothesis on $c$,
pick $(a,b)\in[\mathcal A]^2$ such that $x:=c[a\times b]$ is disjont from $i$.
Let $j:=\ssup(x)$, and let $q:j\rightarrow\theta$ be some condition extending $p$ and satisfying $q(\xi)=\tau$ for all $\xi\in x$.
Clearly, $q(c(\alpha,\beta))=\tau$ for all $(\alpha,\beta)\in a\times b$,
so $q$ forces that $d[a\times b]=\{\tau\}$, as sought.
\end{proof}

We conclude this short section with another simple fact worth recording.
\begin{prop} If $\kappa$ is weakly compact, then, in $V^{\add(\kappa,1)}$:
\begin{enumerate}
\item $\U(\kappa,\kappa,\omega,2)$ holds;
\item $\U(\kappa,2,\omega,\omega)$ holds;
\item $\U(\kappa,2,\theta,2)$ fails for every regular uncountable $\theta<\kappa$.
\end{enumerate}
\end{prop}
\begin{proof} (1) By \cite[Lemma~2.7]{paper34}.

(2) It is not hard to see that, in $V^{\add(\kappa,1)}$, $\pr_1(\kappa,\omega_1,\omega,\omega)$ holds.
In particular, by \cite[Lemma~2.3]{paper34}, $\U(\kappa,2,\omega,\omega)$ holds.

(3) Evidently, for every $ccc$ forcing $\mathbb P$ and every infinite cardinal $\theta$,
$\kappa\rightarrow[\kappa]^2_\theta$ implies $V^{\mathbb P}\models \kappa\rightarrow[\kappa]^2_{\theta,\omega}$.
In particular, in $V^{\add(\kappa,1)}$, $\U(\kappa,2,\theta,2)$ fails for every regular uncountable $\theta<\kappa$.
\end{proof}

\section{Subadditive colorings} \label{subadditive_sec}

We now turn to \emph{subadditive} colorings, which form the primary topic of this paper.
We begin by recalling the definition of subadditivity, splitting the definition
into its two constituent parts.

\begin{defn}
A coloring $c:[\kappa]^2\rightarrow \theta$ is
  \emph{subadditive} iff the following two statements hold:
  \begin{enumerate}
  \item $c$ is \emph{subadditive of the first kind},
  that is, for all $\alpha < \beta < \gamma < \kappa$,
  $$c(\alpha, \gamma) \leq \max\{c(\alpha, \beta), c(\beta, \gamma)\};$$
  \item $c$ is \emph{subadditive  of the second kind},
  that is, for all $\alpha < \beta < \gamma < \kappa$,
  $$c(\alpha, \beta) \leq \max\{c(\alpha, \gamma), c(\beta, \gamma)\}.$$
  \end{enumerate}
\end{defn}

We shall write $\US(\kappa, \mu, \theta, \chi)$ to assert that $\U(\kappa, \mu, \theta, \chi)$
holds and that it moreover admits a subadditive witness. Note that the function $c:[\kappa]^2
\rightarrow \kappa$ defined by letting
$c(\alpha, \beta) := \alpha$ for all $(\alpha, \beta) \in [\kappa]^2$ is a closed, subadditive
witness to $\U(\kappa, \kappa, \kappa, \kappa)$, so, in all situations
of interest, we will have $\theta < \kappa$.

Subadditivity allows us to show that witnesses to certain instances of
$\U(\ldots)$ in fact satisfy stronger instances, as in the following lemma.

\begin{lemma}\label{uPLUSsub}
  Suppose that $c:[\kappa]^2\rightarrow\theta$ is a witness
  to
  $\US(\kappa,2,\theta,2)$, with $\theta<\kappa$. Then the following statements all hold.
  \begin{enumerate}
    \item For every $\mu < \kappa$, $c$ witnesses $\U(\kappa,\mu,\theta,2)$.
    If there exist no $\kappa$-Souslin trees, then $c$ moreover witnesses $\U(\kappa,\kappa,\theta,2)$.
    \item For every $\mu\le\kappa$, if $c$ witnesses $\U(\kappa,\mu,\theta,2)$,
      then $c$ moreover witnesses $\U(\kappa,\mu,\theta,\cf(\theta))$.
    \item For every $\chi\in\reg(\kappa)$, if $c$ is $\Sigma$-closed for
      some stationary $\Sigma \s E^\kappa_{\geq \chi}$, then $c$ witnesses
      $\U(\kappa,\kappa,\theta,\chi)$.
  \end{enumerate}
\end{lemma}

\begin{proof}
  For each $i<\theta$, we define an ordering $<_i$ on $\kappa$ by letting $\alpha<_i\beta$
  iff $\alpha<\beta$ and $c(\alpha,\beta)\le i$. The fact that $<_i$ is transitive follows
  from the fact that $c$ is subadditive of the first kind.
  \begin{claim} For every $i<\theta$, $(\kappa, <_i)$ is a tree with no branches of size $\kappa$.
  \end{claim}
  \begin{cproof} Let $i<\theta$ and $\gamma<\kappa$. To see that set $\{\alpha<\kappa\mid \alpha<_i\gamma\}$
    is well-ordered by $<_i$, fix $\alpha < \beta$ such that $\alpha, \beta <_i \gamma$. As $c$ is subadditive
    of the second kind,
    $c(\alpha,\beta)\le\max\{c(\alpha,\gamma),c(\beta,\gamma)\}\le i$, so $\alpha<_i\beta$.

  In addition, by the hypothesis on $c$, for every $A\in[\kappa]^\kappa$,
  there is a pair $(\alpha,\beta)\in[A]^2$ such that $c(\alpha,\beta)>i$,
  so $\neg(\alpha<_i\beta)$. That is, $(\kappa,<_i)$ has no chains of size $\kappa$.
  \end{cproof}

  (1) Fix $\mu \le \kappa$, and suppose that we are given some $A\in[\kappa]^\kappa$ and $i<\theta$.
  We would like to find $B\in[A]^\mu$ such that $c(\alpha,\beta)>i$ for all $(\alpha,\beta)\in[B]^2$.
  In particular, if there exists $B\in[\kappa]^\mu$ which
  is an antichain with respect to $<_i$, then we are done.
  Hereafter, suppose this is not the case.

  $\br$ If $\mu<\kappa$, then $(A,<_i)$ is a tree of size $\kappa$ all of whose levels have size $<\mu$.
  Since $\mu<\kappa$ and $\kappa$ is regular, a result of Kurepa \cite{kurepa} implies
  that $(\kappa,<_i)$ has a branch of size $\kappa$, contradicting the preceding claim.

  $\br$ If $\mu=\kappa$,  then $(A,<_i)$ forms a $\kappa$-Souslin tree.

\medskip

  (2) Suppose that $\mu \leq \kappa$ and $c$ witnesses $\U(\kappa, \mu, \theta, 2)$.
  Suppose also that $\mathcal A\s[\kappa]^{<\cf(\theta)}$ consists of $\kappa$-many
  pairwise disjoint sets and $i<\theta$ is a prescribed color. We will find
  $\mathcal{B} \in [\mathcal{A}]^\mu$ such that $\min(c[a \times b]) > i$ for all
  $(a,b) \in [\mathcal{B}]^2$.
  For each $\gamma<\kappa$, pick $a_\gamma\in\mathcal A$ with $\gamma < a_\gamma$.
  Define $f:\kappa\rightarrow\theta$ and $g:\kappa\rightarrow\kappa$ by setting,
  for all $\gamma < \kappa$,
  \begin{itemize}
    \item $f(\gamma):=\sup\{c(\gamma,\alpha)\mid \alpha\in a_\gamma\}$;
    \item $g(\gamma):=\sup(a_\gamma)$.
  \end{itemize}
  Fix $\varepsilon<\theta$ for which $A:=\{ \gamma<\kappa\mid f(\gamma)=\varepsilon\ \&\
  g[\gamma]\s\gamma\}$ is stationary. Since $c$ witnesses $\U(\kappa,\mu,\theta,2)$,
  we can pick $B\in[A]^\mu$ such that $c(\gamma,\delta)>\max\{\varepsilon,i\}$ for all $(\gamma, \delta) \in [B]^2$.
  We claim that $\mathcal B:=\{ a_\gamma\mid \gamma\in B\}$ is a subfamily of $\mathcal A$ as sought.
  To see this, pick a pair $(\gamma,\delta)\in [B]^2$ along with $(\alpha,\beta)\in a_\gamma\times a_\delta$.
  Clearly, $\gamma<\alpha<\delta<\beta$. By the subadditivity of $c$ and the choice of $\varepsilon$, we have:
  $$
    c(\gamma,\delta)\le\max\{c(\gamma,\alpha),c(\alpha,\delta)\}\le
    \max\{c(\gamma,\alpha),c(\alpha,\beta),c(\delta,\beta)\}\le\max\{\varepsilon,c(\alpha,\beta)\}.
  $$
  Recalling that $\max\{\varepsilon,i\}<c(\gamma,\delta)$, we infer that
  $i<c(\gamma,\delta)\le c(\alpha,\beta)$, as sought.

  (3) Suppose that $\chi\in\reg(\kappa)$, $\Sigma \s E^\kappa_{\geq \chi}$ is
  stationary, and $c$ is $\Sigma$-closed.

  \begin{claim}
    Suppose that $\mathcal A\s[\kappa]^{<\chi}$ is a family consisting of $\kappa$-many
    pairwise disjoint sets, $D$ is a club in $\kappa$, and $i<\theta$ is a prescribed color.
    Then there exist $\gamma\in D \cap \Sigma$ and $a\in\mathcal A$ such that
    \begin{enumerate}
      \item[(a)] $\gamma < a$;
      \item[(b)] for all $\beta\in a$, we have $c(\gamma,\beta)>i$.
    \end{enumerate}
  \end{claim}

  \begin{cproof}
    Suppose not. Let $T := D \cap \Sigma$, and note that, for all $\gamma\in T$ and
    all $a\in\mathcal A$ with $a>\gamma$, for some $\beta\in a$, we have
    $c(\gamma,\beta)\le i$. Consider the tree $(T,<_i)$.
    We claim that it has no antichains of size $\chi$. To see this, fix an arbitrary
    $X\s T$ of order type $\chi$, and let $\delta:=\sup(X)$. Fix an arbitrary $a\in\mathcal A$
    with $a > \delta$. For all $\gamma\in X$, since $a>\gamma$, we may
    find some $\beta\in a$ with $c(\gamma,\beta)\le i$. Since $|X|=\chi>|a|$, we may
    then pick $(\gamma, \gamma') \in [X]^2$ and $\beta\in a$ such that
    $c(\gamma,\beta)\le i$ and $c(\gamma',\beta)\le i$. It follows that
    $c(\gamma,\gamma')\le\max\{c(\gamma,\beta),c(\gamma',\beta)\}\le i$. In particular,
    $X$ is not an antichain. Now, since $\chi<\kappa$ are regular cardinals, the
    aforementioned result of Kurepa \cite{kurepa} entails
    the existence of $B\in[T]^\kappa$ which is a chain with respect to $<_i$.
    By the hypothesis on $c$, we can pick $(\alpha, \beta) \in [B]^2$ such that $c(\alpha,\beta)>i$.
    Then $\neg(\alpha<_i\beta)$ which is a contradiction to the fact that $B$ is a chain in $(T,<_i)$.
  \end{cproof}

  As $c$ is $\Sigma$-closed and $\Sigma \s E^\kappa_{\geq \chi}$,
  Clause~(b) of the preceding is equivalent to:
  \begin{enumerate}
    \item[(b')] there is $\epsilon < \gamma$ such that, for all $\alpha \in (\epsilon,
    \gamma)$ and all $\beta\in a$, we have $c(\alpha,\beta)>i$.
  \end{enumerate}
  Thus, by the implication $(2)\implies(3)$ of Fact~\ref{pumpclosed}, we can conclude
  that $c$ witnesses $\U(\kappa, \kappa, \theta, \chi)$, as desired.
\end{proof}

\subsection{Narrow systems and trees with ascent paths} \label{nsp_sec}
Given a binary relation $R$ on a set $X$,
for $a,b \in X$, we say that $a$ and $b$ are \emph{$R$-comparable} iff $a = b$, $a \mathrel{R} b$, or $b \mathrel{R} a$.
$R$ is \emph{tree-like} iff, for all $a,b,c \in X$, if $a \mathrel{R} c$ and $b \mathrel{R} c$, then $a$ and $b$ are $R$-comparable.

\begin{defn}[Magidor-Shelah, \cite{MgSh:324}]\label{system_def}
$S = \langle \bigcup_{\alpha \in I}\{\alpha\} \times \theta_\alpha, \mathcal{R} \rangle$ is a \emph{$\kappa$-system} if all of the following hold:
	\begin{enumerate}
		\item $I \subseteq \kappa$ is unbounded and, for all $\alpha \in I$, $\theta_\alpha$ is a cardinal such that $0 < \theta_\alpha < \kappa$;
		\item $\mathcal{R}$ is a set of binary, transitive, tree-like relations on $\bigcup_{\alpha \in I}\{\alpha\} \times \theta_\alpha$ and $0 < |\mathcal{R}| < \kappa$;
		\item for all $R \in \mathcal{R}$, $\alpha_0, \alpha_1 \in I$, $\beta_0 < \theta_{\alpha_0}$, and $\beta_1 < \theta_{\alpha_1}$, if $(\alpha_0, \beta_0) \mathrel{R} (\alpha_1,\beta_1)$, then $\alpha_0 < \alpha_1$;
		\item for every $(\alpha_0,\alpha_1)\in[I]^2$. there are $(\beta_0,\beta_1)\in \theta_{\alpha_0}\times \theta_{\alpha_1}$ and $R \in \mathcal{R}$ such that $(\alpha_0, \beta_0) \mathrel{R} (\alpha_1, \beta_1)$.
	\end{enumerate}

Define $\mathrm{width}(S) := \sup\{|\mathcal{R}|,\theta_\alpha \mid \alpha \in I\}$.
A $\kappa$-system $S$ is \emph{narrow}  if $\mathrm{width}(S)^+ < \kappa$.
For $R \in \mathcal{R}$, a \emph{branch of $S$ through $R$} is a set $B\s \bigcup_{\alpha \in I}\{\alpha\} \times \theta_\alpha$ such that for all $a,b \in B$, $a$ and $b$ are $R$-comparable.
A branch $B$ is \emph{cofinal} iff $\sup\{ \alpha\in I\mid \exists \tau<\theta_\alpha~(\alpha,\tau)\in B\}=\kappa$.
\end{defn}

\begin{defn}[\cite{narrow_systems}]
 The \emph{$(\theta, \kappa)$-narrow system property}, abbreviated $\nsp(\theta, \kappa)$, asserts that every narrow $\kappa$-system of width $<\theta$ has a cofinal branch.
\end{defn}

\begin{lemma}\label{lmma35} Suppose that $\theta<\theta^+<\kappa$ and $c:[\kappa]^2\rightarrow\theta$ is a subadditive coloring.
If $\nsp(\theta^+,\kappa)$ holds,
then $c$ fails to witness $\U(\kappa,2,\theta,2)$.
\end{lemma}
\begin{proof} Define an binary relation $ R$ on $\kappa\times\theta$ by letting
$(\alpha,i)\mathrel{R}(\beta,j)$ iff $\alpha<\beta$, $i=j$, and $c(\alpha,\beta)\le i$.
\begin{claim} Let $\alpha<\beta<\gamma<\kappa$. Then:
\begin{enumerate}
\item there exists $i<\theta$ such that $(\alpha,i)\mathrel{R}(\beta,i)$;
\item for all $i<\theta$, if $(\alpha,i)\mathrel{R}(\beta,i)$  and $(\beta,i)\mathrel{R}(\gamma,i)$,
then $(\alpha,i)\mathrel{R}(\gamma,i)$;
\item for all $i<\theta$, if $(\alpha,i)\mathrel{R}(\gamma,i)$  and $(\beta,i)\mathrel{R}(\gamma,i)$,
then $(\alpha,i)\mathrel{R}(\beta,i)$.
\end{enumerate}
\end{claim}
\begin{cproof} (1) Just take $i:=c(\alpha,\beta)$.

(2) By subadditivity of the first kind.

(3) By subadditivity of the second kind.
\end{cproof}

It thus follows that $S:=\langle \kappa\times\theta,\{R\}\rangle$
is a narrow $\kappa$-system.
So, assuming that $\nsp(\theta^+,\kappa)$ holds,
we may fix $B\s\kappa\times\theta$ that forms a cofinal branch of $S$ through $R$.
Pick $A\in[\kappa]^\kappa$ and $i<\theta$ such that $B=A\times\{i\}$.
Now, if $c$ were to witness $\U(\kappa,2,\theta,2)$, then we could fix $(\alpha,\beta)\in[A]^2$ such that $c(\alpha,\beta)>i$.
But then $(\alpha,i)\mathrel{R}(\beta,i)$ would fail to hold, contradicting the fact that $\{(\alpha,i),(\beta,i)\}\s B$.
\end{proof}

We thus arrive at the following Ramsey-theoretic result.

\begin{cor}\label{cor35} Suppose that $\theta<\theta^+<\kappa$ and $\nsp(\theta^+,\kappa)$ holds.
For every subadditive coloring $c:[\kappa]^2\rightarrow\theta$,
there exists $A\in[\kappa]^\kappa$ such that $c``[A]^2$ is finite.
\end{cor}
\begin{proof} Suppose that $c:[\kappa]^2\rightarrow\theta$ is subadditive coloring.
By the preceding lemma, we may fix $A_0\in[\kappa]^\kappa$
such that $\theta_0:=|c``[A_0]|$ is smaller than $\theta$.
If $\theta_0$ is finite, then we are done. Otherwise, we may identity $c\restriction[A_0]^2$
with a subadditive map from $[\kappa]^2$ to $\theta_0$,
and then utilize $\nsp((\theta_0)^+,\kappa)$ to find $A_1\in[A_0]^\kappa$
for which $\theta_1:=|c``[A_1]|$ is smaller than $\theta_0$.
Continuing in this fashion, we produce a
decreasing chain of sets $A_0\supseteq A_1\supseteq\ldots$
and a strictly decreasing sequence of cardinals $\theta_0>\theta_1>\cdots$.
As the cardinals are well-founded, after finitely many steps we will arrive at a set $A_n\in[\kappa]^\kappa$
for which $\theta_n:=|c``[A_n]|$ is finite.
\end{proof}

\begin{cor}\label{cor36} $\PFA$ implies that for every regular cardinal $\kappa \geq \aleph_2$,
for every subadditive coloring $c:[\kappa]^2\rightarrow\omega$,
there exists $A\in[\kappa]^\kappa$ such that $c``[A]^2$ is finite.
\end{cor}
\begin{proof} By \cite[\S10]{narrow_systems}, $\PFA$ implies $\nsp(\omega_1, \kappa)$ for all regular $\kappa\ge\omega_2$.
Now, appeal to Lemma~\ref{lmma35}.
\end{proof}

Next, we prove a pair of lemmas linking $\US(\ldots)$ to the existence of trees with
ascent paths but no cofinal branches. We first recall the following definition.

\begin{defn}
  Suppose that $\mathcal{T} = (T, <_T)$ is a tree of height $\kappa$.
  \begin{enumerate}
    \item For all $\alpha < \kappa$, $T_\alpha$ denotes level $\alpha$ of $\mathcal{T}$.
    \item A \emph{$\theta$-ascent path} through $\mathcal{T}$ is a sequence of functions
    $\langle f_\alpha \mid \alpha < \kappa \rangle$ such that
    \begin{enumerate}
      \item for all $\alpha < \kappa$, $f_\alpha:\theta \rightarrow T_\alpha$;
      \item for all $\alpha < \beta < \kappa$, there is $i < \theta$ such that,
      for all $j \in [i, \theta)$, we have $f_\alpha(j) <_T f_\beta(j)$.
    \end{enumerate}
  \end{enumerate}
\end{defn}

\begin{lemma}\label{lemma66}
  Suppose that $\theta\in\reg(\kappa)$, and $\mathcal T$ is
  a tree of height $\kappa$ admitting a $\theta$-ascent path but no branch of size $\kappa$.
  Then $\US(\kappa,2,\theta,\theta)$ holds.
\end{lemma}

\begin{proof} Write $\mathcal T$ as $(T,<_T)$.
  Fix a $\theta$-ascent path $\langle f_\alpha\mid\alpha<\kappa\rangle$ through $\mathcal T$,
  and derive a coloring $c:[\kappa]^2\rightarrow\theta$ via
  \[
    c(\alpha,\beta):=\min\{i<\theta\mid \forall j\in[i,\theta)~f_\alpha(j)<_T f_\beta(j)\}.
  \]

  We shall show that $c$ witnesses $\US(\kappa,2,\theta,\theta)$.

  \begin{claim}
    $c$ is subadditive.
  \end{claim}
  \begin{cproof}
    Suppose that $\alpha<\beta<\gamma<\kappa$.

    $\br$ Let $i:=\max\{c(\alpha,\beta),c(\beta,\gamma)\}$.
    Then, for all $j\in[i,\theta)$, we have
    \[
      f_\alpha(j)<_T f_\beta(j)<_T f_\gamma(j),
    \]
    and hence $f_\alpha(j)<_T f_\gamma(j)$. Consequently, $c(\alpha,\gamma)\le i$.

    $\br$ Let $i:=\max\{c(\alpha,\gamma),c(\beta,\gamma)\}$.
    Then, for all $j\in[i,\theta)$, we have $f_\alpha(j)<_T f_\gamma(j)$ and $f_\beta(j)<_T f_\gamma(j)$.
    But $\mathcal T$ is a tree, and hence $f_\alpha(j)<_T f_\beta(j)$. Consequently, $c(\alpha,\beta)\le i$.
  \end{cproof}

  By Lemma~\ref{uPLUSsub}, it remains to verify that $c$ witnesses $\U(\kappa,2,\theta,2)$.
  Suppose this is not the case. Then for some $A\in[\kappa]^\kappa$, $i:=\sup(c``[A]^2)$ is $<\theta$.
  But then the $<_T$-downward closure of $\{f_\alpha(i) \mid \alpha \in A\}$ is a branch
  of size $\kappa$ through $\mathcal T$, contradicting our assumptions.
\end{proof}

The preceding admits a partial converse. Before stating it, we recall the notion of a
\emph{$C$-sequence}, which will be used in its proof.

\begin{defn}
  A \emph{$C$-sequence} over $\kappa$ is a sequence $\langle C_\beta \mid \beta < \kappa
  \rangle$ such that, for all $\beta < \kappa$, $C_\beta$ is a closed subset of
  $\beta$ with $\sup(C_\beta) = \sup(\beta)$.
\end{defn}

\begin{lemma}\label{lemma35}
  Suppose that $\theta\in\reg(\kappa)$.
  If there exists a somewhere-closed witness to $\US(\kappa,2,\theta,\theta)$,
  then there exists a tree $\mathcal T$ of height $\kappa$ admitting a $\theta$-ascent path but no branch of size $\kappa$.
\end{lemma}
\begin{proof} As there is an $\omega_1$-Aronszajn tree, and as any $\omega_1$-Aronszajn tree admits an $\omega$-ascent path,
we may assume that $\kappa\ge\aleph_2$.
Fix a stationary subset $\Sigma\s\kappa$ and a $\Sigma$-closed coloring $c:[\kappa]^2\rightarrow\theta$ witnessing $\US(\kappa,2,\theta,\theta)$.
For every $\beta\in\acc(\kappa)$, let $i(\beta)$ denote the least $i\le\theta$ such that $\sup(D^c_{\le i}(\beta))=\beta$.
Note that if $\cf(\beta)\neq\theta$, then $i(\beta)<\theta$.
Now, for every $i<\theta$, define a $C$-sequence $\vec C^i:=\langle C_\beta^i\mid\beta<\kappa\rangle$ as follows.

$\br$ Let $C^i_0:=\emptyset$.

$\br$ For all $\beta<\kappa$, let $C_{\beta+1}^i:=\{\beta\}$.

$\br$ For all $\beta\in\acc(\kappa)$ such that $i\ge i(\beta)$, let $C_\beta^i:=\cl(D^c_{\le i}(\beta))$.

$\br$ For all $\beta\in\acc(\kappa)$ such that $i<i(\beta)<\theta$, let $C_\beta^i:=\cl(D^c_{\le i(\beta)}(\beta))$.

$\br$ For any other $\beta$, let $C^i_\beta$ be a club in $\beta$ of order-type $\theta$.

Then, let $\rho_2^i:[\kappa]^2\rightarrow\omega$ be the corresponding \emph{number of steps}
function derived from walking along $\vec C^i$, as in \cite[\S6.3]{todorcevic_book}.
Then, for all $i<\theta$ and $\beta<\kappa$, define a function $\rho_{2\beta}^i:\beta\rightarrow\omega$
via $\rho_{2\beta}^i(\alpha):=\rho_2^i(\alpha,\beta)$.
Now, let $$T:=\{ \rho_{2\beta}^i\restriction\alpha\mid i<\theta, ~ \alpha\le\beta<\kappa\}.$$
It is clear that $\mathcal T:=(T,{\s})$ is a tree of height $\kappa$.
\begin{claim} $\mathcal T$ has no branch of size $\kappa$.
\end{claim}
\begin{cproof}
Otherwise, by a standard argument (e.g., the proof of \cite[Corollary~2.6]{paper18}), there exists $i<\theta$ for which
$\rho_2^i$ admits a homogeneous set of size $\kappa$. Fix such an $i$.
By \cite[Theorem~6.3.2]{todorcevic_book}, then, we may fix a club $C\s\kappa$ such that,
for every $\alpha<\kappa$,
there exists $\beta<\kappa$ such that $C\cap\alpha\s C^i_{\beta}$.
By the definition of $\vec C^i$, for every $\alpha<\kappa$ with $\otp(C\cap\alpha)>\theta$,
it follows that there exist $j<\theta$ and $\beta<\kappa$ such that $C\cap\alpha\s \cl(D^c_{\le j}(\beta))$.
By the pigeonhole principle, we may now fix $j<\kappa$ such that
for every $\alpha<\kappa$, there exists $\beta_\alpha\ge\alpha$ such that
$C\cap\alpha\s \cl(D^c_{\le j}(\beta_\alpha))$.

As $c$ in particular witnesses $\U(\kappa,2,\theta,3)$,
we may now find $(\gamma,\alpha)\in[C\cap\Sigma]^2$ such that
$c(\gamma,\beta_\alpha)>j$. As $\gamma\in C\cap\alpha$, we infer that $\gamma\in\cl(D^c_{\le j}(\beta_\alpha))$. As $\gamma\in\Sigma$ and $c$ is $\Sigma$-closed,
it follows that $\gamma\in D^c_{\le j}(\beta_\alpha)$,
contradicting the fact $c(\gamma,\beta_\alpha)>j$.
\end{cproof}

\begin{claim} $\mathcal T$ admits a $\theta$-ascent path.
\end{claim}
\begin{cproof} As $\kappa\ge\aleph_2$, $S:=\acc(\kappa)\setminus E^\kappa_\theta$ is stationary.
For every $\alpha<\kappa$, define $f_\alpha:\theta\rightarrow T_\alpha$ via $f_\alpha(i):=\rho_{2\min(S\setminus\alpha)}^i\restriction\alpha$.
To see that $\langle f_\alpha\mid \alpha<\kappa\rangle$ forms a $\theta$-ascent path through our tree,
fix arbitrary $\alpha<\beta<\kappa$.
Write $\gamma:=\min(S\setminus\alpha)$ and $\delta:=\min(S\setminus\beta)$.
To avoid trivialities, suppose that $\gamma\neq\delta$, so that $\alpha\le\gamma<\beta\le\delta$.
As $(\gamma,\delta)\in[S]^2$,  $i:=\max\{i(\gamma),i(\delta),c(\gamma,\delta)\}$ is $<\theta$,
and, for all $j\in[i,\theta)$, $C_\gamma^j:=\cl(D^c_{\le j}(\gamma))$
and $C_\delta^j:=\cl(D^c_{\le j}(\delta))$.
By subadditivity, for every $\epsilon<\gamma$, $c(\epsilon,\gamma)\le\max\{c(\epsilon,\delta),c(\gamma,\delta)\}$
and $c(\epsilon,\delta)\le\max\{c(\epsilon,\gamma),c(\gamma,\delta)\}$.
So, for every $j\in[i,\theta)$, from $c(\gamma,\delta)\le i\le j$, we infer that $C^j_\gamma=C^j_\delta\cap\gamma$, $\rho_{2\gamma}^j=\rho_{2\delta}^j$ and $f_\alpha(j)\s f_\beta(j)$,
as sought.
\end{cproof}
This completes the proof.
\end{proof}

\subsection{Locally small colorings} \label{locally_small_sec}
\begin{defn} A coloring $c:[\lambda^+]^2\rightarrow\cf(\lambda)$ is \emph{locally small} iff $|D^c_{\le i}(\beta)|<\lambda$ for all $i<\cf(\lambda)$ and $\beta<\lambda^+$.
\end{defn}

By Corollary~\ref{lemma314} below, if $\lambda$ is regular,
then any locally small coloring is a witness to $\U(\lambda^+,\lambda^+,\lambda,\lambda)$.
In the general case, a locally small coloring witnesses an unbalanced strengthening of $\U(\lambda^+,2,\cf(\lambda),\cf(\lambda))$, as follows.

\begin{lemma}\label{locally} Suppose that $c:[\lambda^+]^2\rightarrow\cf(\lambda)$ is a
locally small coloring.
\begin{enumerate}
\item For every $\nu<\cf(\lambda)$, $c$ witnesses $\U(\lambda^+,\lambda\circledast\nu,\cf(\lambda),\cf(\lambda))$;
\item If $c$ is subadditive, then $c$ moreover witnesses $\U(\lambda^+,\lambda^+,\cf(\lambda),\cf(\lambda))$.
\end{enumerate}
\end{lemma}
\begin{proof}Suppose that we are given $\sigma < \cf(\lambda)$, a pairwise disjoint subfamily
  $\mathcal{A} \s [\lambda^+]^{\sigma}$ of size $\lambda^+$,
  and some $i < \cf(\lambda)$. Find a large enough $\epsilon<\lambda^+$ such that $\mathcal A\cap\mathcal P(\epsilon)$ has size $\lambda$.
  Now, given $\nu<\cf(\lambda)$, fix any $\mathcal B'\in[\mathcal A]^\nu$ with $\min(\bigcup\mathcal B')>\epsilon$.
  Since $c$ is locally small and $|\bigcup\mathcal B'|<\cf(\lambda)$, $X:=\bigcup_{b\in\mathcal B'}\bigcup_{\beta\in b} D^c_{\le i}(\beta)$ has size $<\lambda$.
  As $\mathcal A\cap\mathcal P(\epsilon)$ consists of $\lambda$-many pairwise disjoint sets, it follows that
  $\mathcal A':=\{ a\in\mathcal A\cap\mathcal P(\epsilon)\mid a\cap X=\emptyset\}$ has size $\lambda$.
  Evidently, for every $(a,b)\in\mathcal A'\times\mathcal B'$, we have $a<b$ and
  $\min(c[a \times b]) > i$.

  (2) Suppose that $c$ is subadditive. By Lemma~\ref{uPLUSsub}(2), it suffices to prove that $c$ witnesses $\U(\lambda^+,\lambda^+,\cf(\lambda),2)$.
  So, let $A\in[\lambda^+]^{\lambda^+}$ and $i<\cf(\lambda)$ be given; we need to find $B\in[A]^{\lambda^+}$
  such that $c(\alpha,\beta)>i$ for all $(\alpha,\beta)\in[B]^{\lambda^+}$.

  As in the proof of Lemma~\ref{uPLUSsub},
  we define a tree ordering $<_i$ on $\lambda^+$ by letting $\alpha<_i\beta$
  iff $\alpha<\beta$ and $c(\alpha,\beta)\le i$.
  By Clause~(1), $c$ in particular witnesses $\U(\lambda^+,2,\lambda^+,2)$,
  and hence $(\lambda^+,<_i)$ admits no chains of size $\lambda^+$.
  Assuming that the sought-after $B\in[A]^{\lambda^+}$ does not exist, it follows that
  $(A,<_i)$ has no antichains of size $\lambda^+$, so $(A,<_i)$ is a $\lambda^+$-Souslin tree.
  Forcing with this tree, we arrive at a $\lambda^+$-distributive forcing extension $V[G]$
  in which $(A,<_i)$ does admit a chain of size $\lambda^+$.
  But $V[G]$ is a $\lambda^+$-distributive forcing extension of $V$,
  and hence $c$ remains locally small in $V[G]$, and in particular, it witnesses $\U(\lambda^+,2,\cf(\lambda),2)$.
  This is a contradiction.
\end{proof}

\begin{lemma}\label{locallysmall} For every infinite cardinal $\lambda$, there exists a locally small witness to $\U(\lambda^+,\lambda^+,\cf(\lambda),\lambda)$
which is subadditive of the first kind.
\end{lemma}
\begin{proof} We focus on the case in which $\lambda$ is singular,
since, if $\lambda$ is regular, a better result is given by Lemma~\ref{fact219}(1) below.
Let $\langle \lambda_i\mid i<\cf(\lambda)\rangle$ be a strictly increasing sequence of regular uncountable cardinals converging to $\lambda$.
By \cite[Lemma~4.1]{Sh:351}, there exists a coloring $c:[\lambda^+]^2\rightarrow\cf(\lambda)$ which is subadditive of the first kind
and locally small in the following strong sense:
for all $i<\cf(\lambda)$ and $\beta<\lambda^+$,
$|D^c_{\le i}(\beta)|\le\lambda_i$.

\begin{claim}
For every pairwise disjoint subfamily $\mathcal A\s[\lambda^+]^{<\lambda}$ of size $\lambda^+$,
for every club $D\s\lambda^+$, and for every $i<\cf(\lambda)$,
      there exist $\gamma\in D$, $a\in\mathcal A$, and $\epsilon < \gamma$ such that:
      \begin{itemize}
        \item $\gamma < a$;
        \item for all $\alpha \in (\epsilon, \gamma)$ and $\beta\in a$,
        we have $c(\alpha,\beta)>i$.
      \end{itemize}
      \end{claim}
\begin{cproof} Given $\mathcal A$ and $i$ as above, fix a large enough $j<\cf(\lambda)$ such that $j\ge i$
and $\mathcal A_j:=\{ a\in\mathcal A\mid |a|\le\lambda_j\}$ has size $\lambda^+$.
Now, given a club $D\s\lambda^+$, fix $\gamma\in D\cap E^{\lambda^+}_{>\lambda_j}$,
and pick any $a\in\mathcal A_j$ with $\gamma<a$.

For all $\beta\in a$, we have $|D^c_{\le j}(\beta)|\le\lambda_j<\cf(\gamma)$, so $\epsilon:=\sup(\bigcup\{D^c_{\le j}(\beta)\cap\gamma\mid \beta\in a\})$ is $<\gamma$.
Then for all $\alpha\in(\epsilon,\gamma)$ and $\beta\in a$, we have $c(\alpha,\beta)>j\ge i$.
\end{cproof}
It now follows from Fact~\ref{pumpclosed} that $c$ witnesses $\U(\lambda^+,\lambda^+,\cf(\lambda),\lambda)$.
\end{proof}

\begin{lemma}[Todorcevic, \cite{todorcevic_book}]\label{fact219}
\begin{enumerate}
\item If $\lambda$ is regular, then
  there exists a locally small and subadditive witness to $\U(\lambda^+,\lambda^+,\allowbreak\lambda,\lambda)$.
\item If $\square_\lambda$ holds, then   there exists a locally small and subadditive witness to $\U(\lambda^+,\lambda^+,\cf(\lambda),\lambda)$
which is moreover closed.
\end{enumerate}
\end{lemma}
\begin{proof}
(1)  Consider the function $\rho:[\lambda^+]^2\rightarrow\lambda$ defined in \cite[\S9.1]{todorcevic_book}.
  By \cite[Lemma~9.1.1]{todorcevic_book}, $\rho$ is locally small.
  By \cite[Lemma~9.1.2]{todorcevic_book}, $\rho$ is subadditive.
  By \cite[Theorem~6.2.7]{todorcevic_book}, $\rho$ is a witness to $\U(\lambda^+,\lambda^+,\lambda,\lambda)$.

(2) This follows from Lemmas 7.3.7, 7.3.8, 7.3.11 and 7.3.12 of \cite{todorcevic_book}, together with Lemma~\ref{uPLUSsub}(3).
For $\lambda$ singular, a slightly better result is proved in \cite[Theorem~5.8]{lh_trees_squares_reflection}.
\end{proof}

The following result, due independently to Shani and Lambie-Hanson, shows that
the hypothesis of $\square_\lambda$ cannot be weakened to $\square_{\lambda, 2}$
in Lemma~\ref{fact219}(2). (We note that $\gch$ is not explicitly mentioned
in the quoted results, but it is evident from their proofs that, if $\gch$ holds
in the relevant ground models, then it continues to hold in the forcing extensions
witnessing the conclusion of the result.)

\begin{fact}[Shani, {\cite[Theorem 1]{shani}}, Lambie-Hanson, {\cite[Corollaries 5.13 and 5.14]{lh_trees_squares_reflection}}]\label{fact321}
  Relative to the existence of large cardinals, it is consistent with $\gch$ that
  there is an uncountable cardinal $\lambda$ such that
  $\square_{\lambda, 2}$ holds, and, for every $\theta < \lambda$,
  $\US(\lambda^+, 2, \theta, 2)$ fails. $\lambda$ can be either regular or singular here, though
  attaining the result for singular $\lambda$ requires significantly larger cardinals than
  attaining it for regular $\lambda$.
\end{fact}

The principle $\sap^*_\lambda$ was introduced in \cite[Definition~2.12]{paper07} as a weakening of Jensen's weak square principle $\square^*_\lambda$.
By \cite[Theorem~2.6]{paper07}, assuming $2^\lambda=\lambda^+$, $\sap^*_\lambda$ implies that $\diamondsuit(S)$ holds for every stationary
subset $S\s\lambda^+$ that reflects stationarily often.
By \cite[Corollary~2.16]{paper07}, if $\lambda$ is a singular cardinal such that $2^{<\lambda}<2^\lambda=\lambda^+$
and every stationary subset of $E^{\lambda^+}_{\cf(\lambda)}$ reflects, then $\sap^*_\lambda$ moreover implies $\diamondsuit^*(\lambda^+)$.

\begin{prop}\label{prop317} Suppose that $\lambda$ is a singular cardinal and there exists
a locally small and subadditive coloring $c:[\lambda^+]^2\rightarrow\cf(\lambda)$.
Then $\sap^*_\lambda$ holds and $E^{\lambda^+}_{>\cf(\lambda)}\in I[\lambda^+;\lambda]$.
\end{prop}
\begin{proof} Let $c$ be as above.
By \cite[Definitions 2.4 and 2.12]{paper07},
as $c$ is locally small and subadditive of the first kind,\footnote{The terminology in \cite{paper07} is slightly different; there
\emph{locally small} is dubbed \emph{normal},
and \emph{subadditive of the first kind} is dubbed \emph{subadditive}.}
to verify $\sap^*_\lambda$, it suffices to verify that for every
stationary $S\s E^{\lambda^+}_{\cf(\lambda)}$  and every $\gamma\in\Tr(S)$,
there exists a stationary $S_\gamma\s S\cap \gamma$  such that $\sup(c``[S_\gamma]^2)<\cf(\lambda)$.

To this end, fix arbitrary $\gamma\in E^{\lambda^+}_{>\cf(\lambda)}$ and a stationary $s\s\gamma$.
As $\cf(\gamma)>\cf(\lambda)$, there exists a large enough $i<\cf(\lambda)$ such that $S_\gamma:=D^c_{\le i}(\gamma)\cap s$ is stationary in $\gamma$.
Since $c$ is subadditive of the second kind, for any pair $(\alpha,\beta)\in[S_\gamma]^2$, we have that $c(\alpha,\beta)\le\max\{c(\alpha,\gamma),c(\beta,\gamma)\}\le i$.
Therefore, $\sup(c``[S_\gamma]^2)\le i$.
Recalling \cite[Definitions 2.3]{paper07},
the very same argument shows that $E^{\lambda^+}_{>\cf(\lambda)}\in I[\lambda^+;\lambda]$.
\end{proof}

\begin{cor} If $\lambda$ is a singular strong limit and there exists
a locally small and subadditive coloring $c:[\lambda^+]^2\rightarrow\cf(\lambda)$,
then $\ap_\lambda$ holds.
\end{cor}
\begin{proof} By the preceding proposition, the hypothesis imply that $E^{\lambda^+}_{>\cf(\lambda)}\in I[\lambda^+;\lambda]$.
Now, given that  $\lambda$ is a strong limit, we moreover get that $\lambda^+\in I[\lambda^+;\lambda]=I[\lambda^+]$, meaning that $\ap_\lambda$ holds.
\end{proof}

\subsection{Forms of coherence and levels of divergence} \label{levels_sec}
We will  also be interested in variants of subadditivity, as captured by the next definitions.

\begin{defn}
A coloring $c:[\kappa]^2\rightarrow \theta$ is
  \emph{weakly subadditive} iff the following two statements hold:
  \begin{enumerate}
  \item $c$ is \emph{weakly subadditive of the first kind},
  that is, for all $\beta < \gamma< \kappa$ and $i<\theta$,
  there is $j<\theta$ such that $D^c_{\le i}(\beta)\s D^c_{\le j}(\gamma)$;
  \item $c$ is \emph{weakly subadditive of the second kind},
  that is, for all $\beta < \gamma < \kappa$ and $i<\theta$,
  there is $j<\theta$ such that $D^c_{\le i}(\gamma)\cap\beta\s D^c_{\le j}(\beta)$.
  \end{enumerate}
\end{defn}

\begin{defn}[Forms of coherence] Let $c:[\kappa]^2\rightarrow\theta$ be a coloring.
\begin{enumerate}
\item $c$ is \emph{$\ell_\infty$-coherent} iff for all $\gamma<\delta<\kappa$,
there is $j<\theta$ such that, for all $i<\theta$, $D^c_{\le i}(\gamma)\s D^c_{\le i+j}(\delta)$ and $D^c_{\le i}(\delta)\cap\gamma\s D^c_{\le i+j}(\gamma)$;
\item For a cardinal $\lambda<\kappa$, $c$ is \emph{$\lambda$-coherent} iff for every $(\gamma,\delta)\in[\kappa]^2$,
$$|\{\alpha<\gamma\mid c(\alpha,\gamma)\neq c(\alpha,\delta)\}|<\lambda;$$
\item For $S\s\acc(\kappa)$,\footnote{Strictly speaking, to avoid ambiguity with Clause~(2), we need to assume that $|S|\ge2$, but in all cases of interest $S$ will in fact be stationary in $\kappa$.} $c$ is \emph{$S$-coherent} iff
for all $\beta\le\gamma<\delta<\kappa$ with $\beta\in S$,
$$\sup\{\alpha<\beta\mid c(\alpha,\gamma)\neq c(\alpha,\delta)\}<\beta.$$
\end{enumerate}
\end{defn}
\begin{remark} For every $\lambda\in\reg(\kappa)$, $c$ is $\lambda$-coherent iff it is $E^\kappa_\lambda$-coherent
iff it is $E^\kappa_{\ge\lambda}$-coherent.
\end{remark}

Motivated by the proof of Lemma~\ref{lemma35}, we introduce the following definition.
\begin{defn}[Levels of divergence]\label{levels}
For a coloring $c:[\kappa]^2 \rightarrow \theta$,
let $$\partial(c):=\{\beta\in\acc(\kappa)\mid \forall\gamma<\kappa\forall i<\theta\, \sup(D^c_{\le i}(\gamma)\cap\beta)<\beta\}.$$
\end{defn}

Note that $\partial(c)\s E^\kappa_{\cf(\theta)}$ and that $c$ is vacuously $\partial(c)$-closed.

\begin{lemma}\label{lemma38} Suppose that $c:[\kappa]^2\rightarrow\theta$ is a coloring.
If $c$ is weakly subadditive of the second kind,
then $$\partial(c)=\{\beta\in\acc(\kappa)\mid\forall i<\theta\, \sup(D^c_{\le i}(\beta))<\beta\}.$$
\end{lemma}
\begin{proof}
If $c$ is weakly subadditive of the second kind,
then for all $\beta<\gamma<\kappa$ and $i<\theta$ such that
$\sup(D^c_{\le i}(\gamma)\cap\beta)=\beta$,
there exists $j<\theta$ such that
$D^c_{\le i}(\gamma)\cap\beta \s D^c_{\le j}(\beta)$,
and hence $\sup(D^c_{\le j}(\beta))=\beta$.
\end{proof}

We will be particularly interested in situations in which $\partial(c)$ is stationary
in $\kappa$; one reason for this is the following lemma, indicating that colorings $c$
for which $\partial(c)$ is stationary automatically witness an instance of $\U(\ldots)$.

\begin{lemma}\label{lemma39}
  Suppose that $c:[\kappa]^2 \rightarrow \theta$ is a coloring for which $\partial(c)$ is stationary. Then $c$ witnesses
  $\U(\kappa, \kappa, \theta, \cf(\theta))$.
\end{lemma}

\begin{proof}
  We prove that $c$ satisfies Clause (2) of Fact~\ref{pumpclosed} with $\chi :=
  \cf(\theta)$, which will yield our desired conclusion.
  Fix a family $\mathcal{A} \subseteq [\kappa]^{<\cf(\theta)}$
  consisting of $\kappa$-many pairwise disjoint sets, a club $D \subseteq \kappa$,
  and an $i < \theta$. Since $\partial(c)$ is stationary, we can fix
  $\gamma \in \partial(c) \cap D$. Also fix an $a \in \mathcal{A}$ such that
  $\gamma < a$. As $\gamma\in\partial(c)$, for all $\beta \in a$,
  we have $\sup(D^c_{\leq i}(\beta) \cap \gamma) < \gamma$. Since
  $|a| < \cf(\theta) = \cf(\gamma)$, we can find $\epsilon < \gamma$ such that
  $\sup(D^c_{\leq i}(\beta) \cap \gamma) < \epsilon$ for all $\beta \in a$.
  Now, for all $\alpha \in (\epsilon, \gamma)$ and all $\beta \in a$, we have
  $\alpha \notin D^c_{\leq i}(\beta)$, so $c(\alpha, \beta) > i$, as desired.
\end{proof}

\begin{cor}  Suppose that $\theta\in\reg(\kappa)$.
  If there exists a subadditive coloring $c:[\kappa]^2\rightarrow\theta$
  for which $\partial(c)$ is stationary,
  then there exists a tree $\mathcal T$ of height $\kappa$ admitting a $\theta$-ascent path but no branch of size $\kappa$.
\end{cor}
\begin{proof} This follows from Lemma~\ref{lemma35}, since any coloring as above is somewhere-closed,
and witnesses $\U(\kappa,\kappa,\theta,\theta)$.
\end{proof}

\begin{lemma}\label{prop323} Let $c:[\kappa]^2\rightarrow\theta$ be a coloring, with $\theta\in\reg(\kappa)$.
\begin{enumerate}
\item If $c$ is $\theta$-coherent, then it is $\ell_\infty$-coherent;
\item If $c$ is subadditive, then it is $\partial(c)$-coherent and $\ell_\infty$-coherent;
\item If $c$ is $\ell_\infty$-coherent, then it is weakly subadditive.
\end{enumerate}
\end{lemma}
\begin{proof}
(1) If $c$ is $\theta$-coherent,
then for all $\gamma<\delta<\kappa$,
$$j:=\sup\{ c(\alpha,\gamma), c(\alpha,\delta)\mid \alpha<\gamma, ~ c(\alpha,\gamma)\neq c(\alpha,\delta)\}$$
is $<\theta$, and, for every $i\in[j,\theta)$,
$D^c_{\le i}(\gamma)=D^c_{\le i}(\delta)\cap\gamma$.

(2) Suppose that $c$ is subadditive.
To see that $c$ is $\ell_\infty$-coherent, let $\gamma<\delta<\kappa$ be arbitrary.
Set $j:=c(\gamma,\delta)$.
For all $i<\theta$ and $\alpha\in D^c_{\le i}(\gamma)$,
$c(\alpha,\delta)\le\max\{c(\alpha,\gamma),c(\gamma,\delta)\}\le\max\{i,j\}$,
and, for all $i<\theta$ and $\alpha\in D^c_{\le i}(\delta)\cap\gamma$,
$c(\alpha,\gamma)\le\max\{c(\alpha,\delta),c(\gamma,\delta)\}\le\max\{i,j\}$.
Altogether, $D^c_{\le i}(\gamma)=D^c_{\le i}(\delta)\cap\gamma$ for every $i\in[j,\theta)$.

Next, we show that $c$ is $\partial(c)$-coherent. To this end, fix $\beta \leq \gamma
< \delta < \kappa$ with $\beta \in \partial(S)$. Set $i := c(\gamma, \delta)$. By
the subadditivity of $c$, it follows that $D^c_{\leq j}(\gamma) = D^c_{\leq j}(\delta)
\cap \gamma$ for all $j \in [i, \theta)$. In particular,
$\{\alpha < \beta \mid c(\alpha, \gamma) \neq c(\alpha, \delta)\}
  \subseteq D^c_{\le i}(\gamma)\cap\beta$.
Since $\beta \in \partial(c)$, $\sup(D^c_{\le i}(\gamma)\cap\beta)<\beta$,
so we are done.

(3) This is clear.
\end{proof}

\begin{cor}\label{lemma314} If $\lambda$ is regular and $c:[\lambda^+]^2\rightarrow\lambda$ is locally small,
then $\partial(c)=E^{\lambda^+}_\lambda$, and hence:
\begin{enumerate}
\item $c$ is a witness to $\U(\lambda^+,\lambda^+,\lambda,\lambda)$;
\item If $c$ is subadditive, then $c$ is $\lambda$-coherent.
\end{enumerate}
\end{cor}
\begin{proof} The fact that $\partial(c) = E^{\lambda^+}_\lambda$ is immediate.
Now, Clause~(1) follows from Lemma~\ref{lemma39},
and Clause~(2) follows from Lemma~\ref{prop323}(2).
\end{proof}

\begin{lemma}\label{lemma325} Let $\theta\in\reg(\kappa)$.
If there exists a $\theta$-coherent witness to $\U(\kappa,2,\theta,2)$,
then $\US(\kappa,2,\theta,2)$ holds.
\end{lemma}
\begin{proof}
Suppose that $c:[\kappa]^2\rightarrow\theta$ is a $\theta$-coherent witness to $\U(\kappa,2,\theta,2)$.
Define a coloring $d:[\kappa]^2\rightarrow\theta$ by letting,
for all $\gamma<\delta<\kappa$, $d(\gamma,\delta)$ be the least $j<\theta$
such that
$D^c_{\le i}(\gamma)=D^c_{\le i}(\delta)\cap\gamma$
for every $i\in[j,\theta)$.
It is clear that $d$ is subadditive. Now, if $d$ fails to witness $\U(\kappa,2,\theta,2)$,
then we may fix $A\in[\kappa]^\kappa$ and $j<\theta$ such that $d(\gamma,\delta)\le j$ for all $(\gamma,\delta)\in[A]^2$. For every $\gamma\in A$, let $\gamma':=\min(A\setminus(\gamma+1))$ and $i_\gamma:=c(\gamma,\gamma')$. Fix $A'\in[A]^\kappa$ and $i<\theta$ such that $i=\max\{j,i_\gamma\}$ for all $\gamma\in A'$.

Then, for every $(\gamma,\delta)\in[A']^2$,
$\gamma<\gamma'\le\delta$ and $\gamma\in D^c_{\le i_\gamma}(\gamma')\s D^c_{\le i}(\gamma')=D^c_{\le i}(\delta)\cap\gamma'$, so $c(\gamma,\delta)\le i$.
It follows that $c``[A']^2$ is bounded in $\theta$, contradicting the fact that $c$ witnesses $\U(\kappa,2,\theta,2)$.
\end{proof}
\begin{cor} If there exists a uniformly coherent $\kappa$-Souslin tree,
then, for every $\theta\in\reg(\kappa)$,
$\US(\kappa,2,\theta,\theta)$ holds.
\end{cor}
\begin{proof} Suppose that there exists a uniformly coherent $\kappa$-Souslin tree. This means that there exists a downward closed subfamily $T\s{}^{<\kappa}2$ such that:
\begin{itemize}
\item[(a)] $(T,{\s})$ is a $\kappa$-Souslin tree;
\item[(b)] for all $s,t\in T$,
$\{ \epsilon\in\dom(s)\cap\dom(t)\mid s(\epsilon)\neq t(\epsilon)\}$ is finite;
\item[(c)] for all $s,t\in T$, if $\dom(s)<\dom(t)$,
then $s*t:=s\cup(t\restriction(\dom(t)\setminus\dom(s)))$ is in $T$.
\end{itemize}

\begin{claim} There exists a downward closed subfamily $\hat T\s{}^{<\kappa}\kappa$ satisfying (a)--(c), in addition to the following:
\begin{itemize}
\item[(d)] For all $t\in\hat T\cap{}^\alpha\kappa$ and $i<\alpha$, $t{}^\smallfrown\langle i\rangle\in\hat T$.
\end{itemize}
\end{claim}
\begin{cproof} The proof is similar to that of \cite[Theorem~3.6]{MR2013395}.
Denote $T_\alpha:=T\cap{}^\alpha\kappa$ for all $\alpha< \kappa$.
By a standard fact, we may fix a club $E\s\kappa$ such that, for every
$(\alpha,\beta)\in[E]^2$, every node in $T_\alpha$ admits at least $|\alpha|$ many extensions in $T_\beta$.
We may also assume that $0\in E$.
Let $\pi:\kappa\leftrightarrow E$ denote the order-preserving bijection,
and denote
$T':=\bigcup_{\alpha<\kappa}T_{\pi(\alpha)}$.
Our next goal is to define a map $\Pi:T'\rightarrow{}^{<\kappa}\kappa$ such that all of the following hold:
\begin{enumerate}
\item for all $\alpha<\kappa$ and $t\in T_{\pi(\alpha)}$, $\Pi(t)\in{}^\alpha\kappa$;
\item for all $s\subsetneq t$ from $T'$, $\Pi(s)\subsetneq\Pi(t)$;
\item for all $s,t\in T'$, if $\dom(s)<\dom(t)$, then $\Pi(s*t)=\Pi(s)*\Pi(t)$.
\end{enumerate}

We shall define $\Pi\restriction T_{\pi(\alpha)}$ by recursion on $\alpha<\kappa$:
\begin{itemize}
\item[$\br$] For $\alpha=0$, we have $T_{\pi(\alpha)}=T_0=\{\emptyset\}$,
so we set $\Pi(\emptyset):=\emptyset$.

\item[$\br$]  For $\alpha=\bar\alpha+1$  such that $\Pi\restriction T_{\pi(\bar\alpha)}$ has been successfully defined,  we first fix $\bar t\in T_{\pi(\bar\alpha)}$.
Find a cardinal $\mu\ge|\bar\alpha|$ and an injective enumeration $\langle t^i\mid i<\mu\rangle$
of all the extensions of
$\bar t$ in $T_{\pi(\alpha)}$.
Finally, for every $t\in T_{\pi(\alpha)}$, find the unique $i<\mu$ such that $\bar t * t=t^i$,
and then let  $\Pi(t):=\Pi(t\restriction\bar\alpha){}^\smallfrown\langle i\rangle$.
It is clear that Properties (1)--(3) are preserved.

\item[$\br$]  For $\alpha\in\acc(\kappa)$ such that $\Pi\restriction T_{\pi(\bar\alpha)}$ has been successfully defined for all $\bar\alpha<\alpha$, we just let $\Pi(t):=\bigcup\{\Pi(t\restriction \pi(\bar\alpha))\mid \bar\alpha<\alpha\}$ for every $t\in T_{\pi(\alpha)}$.
It is clear that Properties (1)--(3) are preserved.
\end{itemize}

Set $\hat T:=\im(\Pi)$.
By Property~(2), $(\hat T,{\s})$ is order-isomorphic to $(T',{\s})$ which is a $\kappa$-sized subtree of $(T,{\s})$, so $(\hat T,{\s})$ is indeed a $\kappa$-Souslin tree.
By Properties (3) and (b), for all $s,t\in\hat T$,
$\{ \epsilon\in\dom(s)\cap\dom(t)\mid s(\epsilon)\neq t(\epsilon)\}$ is finite.
By Properties (3) and (c),
for all $s,t\in\hat T$, if $\dom(s)<\dom(t)$,
then $s*t:=s\cup(t\restriction(\dom(t)\setminus\dom(s)))$ is in $T$.
Finally, by the definition of $\Pi\restriction T_{\pi(\alpha)}$ for successor ordinals $\alpha<\kappa$,
and by Property~(c), it is indeed the case that,
for all $t\in\hat T\cap{}^\alpha\kappa$ and $i<\alpha$, $t{}^\smallfrown\langle i\rangle\in\hat T$.
\end{cproof}

Let $\theta\in\reg(\kappa)$.
By Lemma~\ref{lemma325} and Lemma~\ref{uPLUSsub}(2),
to show that $\US(\kappa,2,\allowbreak\theta,\theta)$ holds,
it suffices to find a $\theta$-coherent witness to $\U(\kappa,2,\theta,2)$.
To this end, fix $\hat T$ as in the preceding claim,
and then fix some sequence $\langle t_\beta\mid\beta<\kappa\rangle$
such that $t_\beta \in\hat T\cap{}^\beta\kappa$ for all $\beta < \kappa$.
Define a coloring $c:[\kappa]^2\rightarrow\theta$ via:
$$c(\alpha,\beta):=\begin{cases}t_\beta(\alpha),&\text{if }t_\beta(\alpha)<\theta;\\
0,&\text{otherwise}.
\end{cases}$$

Evidently, $c$ is $\omega$-coherent.
Now, given $A\in[\kappa]^\kappa$, we claim that $c``[A]^2=\theta$.
To see this, let $i<\theta$,
and note that $S:=\{ t_\alpha{}^\smallfrown\langle i\rangle\mid \alpha\in A\setminus\theta\}$
forms a subset of $\hat T$ of size $\kappa$, and hence it cannot be an antichain. Pick $s,t\in S$ such that $s\subsetneq t$.
Let $(\alpha,\beta)\in[A]^2$ be such that $s=t_\alpha{}^\smallfrown\langle i\rangle$ and $t=t_\beta{}^\smallfrown\langle i\rangle$.
As $t_\alpha{}^\smallfrown\langle i \rangle$ and $t_\beta$ are both initial segments of $t$, we infer that $t_\beta(\alpha)=i$, and hence $c(\alpha,\beta)=i$, as sought.
\end{proof}
\begin{remark}
By \cite[Theorem~C]{paper29}, for every singular cardinal $\lambda$, ${\square(\lambda^+)}+{\gch}$ entails the existence of a uniformly coherent $\lambda^+$-Souslin tree.
By Fact~\ref{fact321} and the preceding Corollary, the same conclusion does not follow from ${\square(\lambda^+,2)}+{\gch}$.
\end{remark}

We now show that the existence of a coloring $c$ for which $\partial(c)$ is stationary
is in fact equivalent to the existence of a nonreflecting stationary subset of
$E^\kappa_\theta$.

\begin{lemma}\label{partial} For a subset $S\s E^\kappa_\theta$, the following are equivalent:
\begin{enumerate}
\item for every $\gamma\in E^\kappa_{>\omega}$, $S\cap\gamma$ is nonstationary in $\gamma$;
\item there exists a coloring $c:[\kappa]^2\rightarrow\theta$
for which $\partial(c)\supseteq S$;
\item there exists an $S$-coherent, closed coloring $c:[\kappa]^2\rightarrow\theta$
for which $\partial(c)\supseteq S$.
\end{enumerate}
\end{lemma}
\begin{proof} $(3)\implies(2)$: This is trivial.

$(2)\implies(1)$: Suppose that $c:[\kappa]^2\rightarrow\theta$ is a coloring
for which $\partial(c)\supseteq S$.
Towards a contradiction, suppose that we are given $\gamma\in E^\kappa_{>\omega}$ such that $S\cap\gamma$ is stationary in $\gamma$.
As $S\s E^\kappa_\theta$, it follows that $\cf(\gamma)>\theta$, and hence $\gamma\notin\partial(c)$.
As $\cf(\gamma)\neq\theta$, we may pick $i<\theta$ such that $\sup(D^c_{\le i}(\gamma))=\gamma$.
Now, pick $\beta\in\acc^+(D^c_{\le i}(\gamma))\cap S$.
It follows that $\sup(D^c_{\le i}(\gamma)\cap\beta)=\beta$,
contradicting the fact that $\beta\in S\s\partial(c)$.

$(1)\implies(3)$: Pick a $C$-sequence $\vec C=\langle C_\alpha\mid\alpha<\kappa\rangle$ such that,
for all $\alpha<\kappa$, $\otp(C_\alpha)=\cf(\alpha)$ and $\acc(C_\alpha)\cap S=\emptyset$.
Let $\tr$ be the function derived from $\vec C$ as in \cite[Definition~4.4]{paper34}.
Define a coloring $c:[\kappa]^2\rightarrow\theta$ via
$$c(\alpha,\gamma):=\sup(\{ \otp(C_{\eta}\cap\alpha)\mid \eta\in\im(\tr(\alpha,\gamma))\}\cap\theta).$$
By \cite[Lemma~4.7]{paper34}, $c$ is closed.
\begin{claim} Let $\beta<\gamma<\kappa$ with $\beta\in S$.
Then there exists $\epsilon<\beta$ such that, for every $\alpha\in(\epsilon,\beta)$,
$c(\alpha,\gamma)=c(\alpha,\beta)\ge \otp(C_\beta\cap\alpha)$.
\end{claim}
\begin{cproof} For every $\alpha < \beta$, since $\beta\in\im(\tr(\alpha,\beta))$ and $\otp(C_\beta\cap\alpha)<\otp(C_\beta)=\theta$, it is immediate that $c(\alpha,\beta)\ge\otp(C_\beta\cap\alpha)$.
Set $j:=c(\beta,\gamma)$,
and then fix a large enough $\epsilon<\beta$ such that $j\le\otp(C_\beta\cap\epsilon)$ and $\lambda(\beta,\gamma)\le\epsilon$.\footnote{The function $\lambda(\cdot,\cdot)$ is defined on \cite[p.~258]{todorcevic_book}.}

Let $\alpha\in(\epsilon,\beta)$. Then $\tr(\alpha,\gamma)=\tr(\beta,\gamma){}^\smallfrown\tr(\alpha,\beta)$,
and hence $$c(\alpha,\gamma)=\sup(\{ \otp(C_{\eta}\cap\alpha)\mid \eta\in\im(\tr(\beta,\gamma))\cup \im(\tr(\alpha,\beta))\}\cap\theta).$$
In particular, $c(\alpha,\gamma)\ge c(\alpha,\beta)\ge\otp(C_\beta\cap\alpha)\ge j$. Now, as
    \begin{align*}
j=c(\beta,\gamma)=&\sup(\{ \otp(C_{\eta}\cap\beta)\mid \eta\in\im(\tr(\beta,\gamma))\}\cap\theta)\\
\ge&\sup(\{ \otp(C_{\eta}\cap\alpha)\mid \eta\in\im(\tr(\beta,\gamma))\}\cap\theta),
    \end{align*}
it follows that
$$c(\alpha,\gamma)=\sup(\{ \otp(C_{\eta}\cap\alpha)\mid \eta\in\im(\tr(\alpha,\beta))\}\cap\theta)=c(\alpha,\beta),$$
as sought.
\end{cproof}
It now immediately follows that $\partial(c)\supseteq S$ and that $c$ is $S$-coherent.
\end{proof}
\begin{remark}
By Fact~\ref{fact321}, the preceding lemma cannot be strengthened to assert that
the existence of a nonreflecting stationary subset of $E^\kappa_\theta$
gives rise to a \emph{subadditive} coloring $c:[\kappa]^2\rightarrow\theta$ for which $\partial(c)$ is stationary.
In fact, a nonreflecting stationary subset of $E^\kappa_\theta$ is not even enough to
imply the existence of a coloring $c:[\kappa]^2\rightarrow\theta$ such that
$\partial(c)$ is stationary and $c$ is weakly subadditive of the first kind. This is
because, by Theorem~\ref{pid_cor} below, $\PFA$ implies that, for example,
any witness to $\U(\omega_3,2,\omega,2)$ is not weakly subadditive of the first kind,
whereas, by a result of Beaudoin  (see the remark at the end of \cite[\S 2]{MR1099782}), $\PFA$ is consistent with
the existence of a nonreflecting stationary subset of $E^{\omega_3}_\omega$.
\end{remark}

By Lemma~\ref{fact219}(1) and Lemma~\ref{lemma314}(2),
for every infinite regular cardinal $\lambda$,
there exists a locally small coloring $c:[\lambda^+]^2\rightarrow\lambda$ that is $\lambda$-coherent.
We shall now prove that for every singular cardinal $\lambda$,
a locally small coloring $c:[\lambda^+]^2\rightarrow\cf(\lambda)$ is never $\cf(\lambda)$-coherent.
Assuming that $c$ is subadditive of the first kind (which is indeed possible, by Lemma~\ref{locallysmall}),
even weaker forms of coherence are not feasible.

\begin{lemma}\label{lemma331} Suppose that $c:[\lambda^+]^2\rightarrow\cf(\lambda)$ is a locally small coloring.
\begin{enumerate}
\item If $\lambda$ is regular
or if $c$ is subadditive of the first kind, then for every cardinal $\theta<\lambda$, $c$ is not $\theta$-coherent.
\item If $\lambda$ is singular, then $c$ is not $\cf(\lambda)$-coherent.
\end{enumerate}
\end{lemma}
\begin{proof} The proof is similar to that of \cite[Theorem~3.7]{MR2013395}.
Suppose for sake of contradiction that $c$ is $\theta$-coherent for some fixed cardinal $\theta<\lambda$ and that either $c$ is subadditive of the first kind or $\theta \leq
\cf(\lambda)$.

\begin{claim}
For every $\gamma<\lambda^+$,
there exists $i<\cf(\lambda)$ such that
$$\otp(D^c_{\le i}(\gamma))+\theta<\otp(D^c_{\le i}(\gamma+\lambda)).$$
\end{claim}
\begin{cproof} Let $\gamma<\lambda^+$. Denote $\delta:=\gamma+\lambda$.
First, since $\otp([\gamma,\delta))=\lambda>\theta$, we may let
$$i_0:=\min\{i<\cf(\lambda)\mid \otp(D^c_{\le i}(\delta)\setminus\gamma)>\theta\}.$$
Second, if $\theta\le\cf(\lambda)$, then, since $c$ is $\theta$-coherent,
$$i_1:=\ssup\{ c(\alpha,\gamma),c(\alpha,\delta)\mid \alpha<\gamma, c(\alpha,\gamma)\neq c(\alpha,\delta)\}$$
is an ordinal $<\cf(\lambda)$. If $\theta>\cf(\lambda)$, then we instead let $i_1:=c(\gamma,\delta)$.

We claim that $i:=\max\{i_0,i_1\}$ is as sought.

$\br$ If $\theta\le\cf(\lambda)$, then for every $\alpha\in D^c_{\le i}(\gamma)$,
either $c(\alpha,\delta)\le i_1\le i$ or $c(\alpha,\delta)=c(\alpha,\gamma)\le i$.
Therefore, $D^c_{\le i}(\gamma)\s D^c_{\le i}(\delta)$.

$\br$ Otherwise, $c$ is subadditive of the first kind and $i_i = c(\gamma, \delta)$.
Then, for every $\alpha\in D^c_{\le i}(\gamma)$, we have
$c(\alpha,\delta)\le\max\{c(\alpha,\gamma),c(\gamma,\delta)\}\le\max\{c(\alpha,\gamma),i_1\}\le i$,
so that $\alpha\in D^c_{\le i}(\delta)$.
Thus, again, $D^c_{\le i}(\gamma)\s D^c_{\le i}(\delta)$.

\smallskip

Now, set $x:=D^c_{\le i}(\delta)\setminus\gamma$,
and notice that $\otp(x)\ge\otp(D^c_{\le i_0}(\delta)\setminus\gamma)\ge\theta+1$.
Altogether  $D^c_{\le i}(\gamma)\uplus x\s D^c_{\le i}(\delta)$ with $D^c_{\le i}(\gamma)\s\min(x)$,
and hence $\otp(D^c_{\le i}(\gamma))+\theta+1\le\otp(D^c_{\le i}(\delta))$.
\end{cproof}

By the claim, for each $\gamma<\lambda^+$, we may fix $i_\gamma<\cf(\lambda)$ such that
$\otp(D^c_{\le i_\gamma}(\gamma))+\theta<\otp(D^c_{\le i_\gamma}(\gamma+\lambda))$.
Fix a sparse enough stationary subset $S\s\lambda^+$ along with $i<\cf(\lambda)$ such that
$i_\gamma=i$ for all $\gamma\in S$, and such that $\beta+\lambda<\gamma$ for all $(\beta,\gamma)\in[S]^2$.
Define a map $f:S\rightarrow\lambda$ via
$$f(\beta):=\otp(D^c_{\le i}(\beta+\lambda)).$$
Let $(\beta,\gamma)\in[S]^2$. As $\beta+\lambda<\gamma$ and $c$ is $\theta$-coherent,
$$\otp(D^c_{\le i}(\beta+\lambda))<\otp(D^c_{\le i}(\gamma))+\theta.$$
Altogether,
$$f(\beta)=\otp(D^c_{\le i}(\beta+\lambda))<\otp(D^c_{\le i}(\gamma))+\theta<\otp(D^c_{\le i}(\gamma+\lambda))=f(\gamma).$$
Therefore, $f$ is an injection from a set of size $\lambda^+$ to $\lambda$, which is a contradiction.
\end{proof}

We conclude this subsection by introducing a notion of forcing that adds
a subadditive coloring $c:[\kappa]^2\rightarrow\theta$ whose $\partial(c)$ is stationary.
This will prove Clause~(1) of Theorem~C.
\begin{thm}\label{thm310}
  Suppose that $\theta \in\reg(\kappa)$. Then:
  \begin{enumerate}
  \item There exists a
  $\theta^+$-directed closed,  $\kappa$-strategically closed forcing notion $\mathbb P$
  that adds a $E^\kappa_{\geq \theta}$-closed, subadditive coloring
  $c:[\kappa]^2\rightarrow\theta$
  for which $\partial(c)$ is stationary.
  In particular, $\forces_{\mathbb P}``\US(\kappa, \kappa,\theta, \theta)"$.
  \item  If $\kappa = \lambda^+$ and $\cf(\lambda) = \theta < \lambda$,
  then there is also a
  $\theta^+$-directed closed,  $(\lambda+1)$-strategically closed (hence $({<}\kappa)$-distributive)
  forcing notion $\mathbb P$
  that adds a coloring as above which is moreover locally small.
  \end{enumerate}
\end{thm}

\begin{proof}
  In Case~(2), let $\langle \lambda_i \mid i < \theta \rangle$ be an increasing sequence of
  regular cardinals converging to $\lambda$, with $\lambda_i > \theta$.
  In Case~(1), simply let $\lambda_i := \kappa$ for all $i < \theta$.

  We will define a forcing notion $\mathbb{P}$ whose generic object will generate
  a coloring $c$ as above. Our poset $\mathbb{P}$ consists of
  all subadditive colorings of the form $p : [\gamma_p + 1]^2 \rightarrow \theta$
  such that
  \begin{itemize}
    \item $\gamma_p < \kappa$;
    \item $p$ is $E^{\gamma_p+1}_{\geq \theta}$-closed;
    \item for all $\beta \leq \gamma_p$ and all $i < \theta$, we have
    $|D^p_{\leq i}(\beta)| < \lambda_i$.
  \end{itemize}

  $\mathbb{P}$ is ordered by reverse inclusion. We also include the
  empty set as the unique maximal condition in $\mathbb{P}$.

  \begin{claim} \label{directed_closed_claim}
    $\mathbb{P}$ is $\theta^+$-directed closed.
  \end{claim}

  \begin{cproof}
    Note that $\mathbb{P}$ is \emph{tree-like}, i.e., if $p,q,r \in \mathbb{P}$
    and $r$ extends both $p$ and $q$, then $p$ and $q$ are $\leq_{\mathbb{P}}$-comparable.
    It therefore suffices to prove that $\mathbb{P}$ is $\theta^+$-closed.
    To this end, suppose that $\xi < \theta^+$ and $\vec{p} = \langle p_\eta \mid
    \eta < \xi \rangle$ is a decreasing sequence of conditions in $\mathbb{P}$.
    We may assume without loss of generality that $\xi$ is an infinite regular
    cardinal and $\vec{p}$ is strictly decreasing, i.e., $\langle \gamma_\eta
    \mid \eta < \xi \rangle$ is strictly increasing, where $\gamma_\eta$ denotes
    $\gamma_{p_\eta}$. For all $\eta < \xi$, by possibly extending $p_\eta$
    to copy some information from $p_{\eta+1}$,
    we may also assume that $\gamma_\eta$ is a successor ordinal.
    Let $\gamma^* := \sup\{\gamma_\eta \mid \eta < \xi\}$,
    and let $q^* := \bigcup_{\eta < \xi} p_\eta$. Note that $q^*$ is not a condition
    in $\mathbb{P}$, since its domain is not the square of a successor ordinal. We will extend it
    to a condition $q: [\gamma^* + 1]^2 \rightarrow \theta$, which will then be a lower
    bound for $\vec{p}$. To do so, it suffices to
    specify $q(\alpha, \gamma^*)$ for all $\alpha < \gamma^*$.
    There are two cases to consider:

    $\br$ Assume that $\xi < \theta$. We can then fix an $i^* < \theta$ such that
    $q^*(\gamma_\eta, \gamma_{\eta'}) \leq i^*$ for all $\eta < \eta' < \xi$.
    Now, given $\alpha < \gamma^*$, let
    $\eta_\alpha < \xi$ be the least $\eta$ for which $\alpha < \gamma_\eta$ and then set
    $q(\alpha, \gamma^*) := \max\{i^*, q^*(\alpha, \gamma_{\eta_\alpha})\}$. It is
    straightforward to prove, using our choice of $i^*$ and the fact that each
    $p_\eta$ is subadditive, that the coloring $q$ thus defined is also
    subadditive.

    Since each $p_\eta$ is $E^{\gamma_\eta + 1}_{\geq \theta}$-closed, in order to
    show that $q$ is $E^{\gamma^*+1}_{\geq \theta}$-closed, it suffices to prove
    that for all $A \subseteq \gamma^*$ and all $i < \theta$ such that
    $A \subseteq D^{q}_{\leq i}(\gamma^*)$ and $\sup(A) \in E^{\gamma^*}_{\geq \theta}$,
    we have $\sup(A) \in D^{q}_{\leq i}(\gamma^*)$. To this end, fix such an $A$
    and $i$. Let $\beta := \sup(A)$. By our choice of $i^*$, we know that $i \geq i^*$ and
    $q(\alpha, \gamma^*) = \max\{i^*, p_{\eta_\beta}(\alpha, \gamma_{\eta_\beta})\}$
    for all $\alpha \in A \cup \{\beta\}$. By the fact that $p_{\eta_\beta}$ is
    $E^{\gamma_{\eta_\beta+1}}_{\geq \theta}$-closed, we know that
    $p_{\eta_\beta}(\beta) \leq i$, so $\beta \in D^{q}_{\leq i}(\gamma^*)$.

    To show that $q$ is a condition, it remains only to verify that
    $|D^{q}_{\leq i}(\gamma^*)| < \lambda_i$ for all $i < \theta$. To this
    end, fix $i < \theta$. By our construction, we have
    $D^{q}_{\leq i}(\gamma^*) \subseteq \bigcup_{\eta < \xi}
    D^{p_\eta}_{\leq i}(\gamma_\eta)$. Since $\xi < \lambda_i$ and $\lambda_i$
    is regular, the fact that each $p_\eta$ is a condition in $\mathbb{P}$ then implies
    that $|D^{q}_{\leq i}(\gamma^*)| < \lambda_i$.

    $\br$ Assume that $\xi = \theta$. Fix a strictly increasing sequence
    $\langle i_\eta \mid \eta < \theta \rangle$ of ordinals below $\theta$
    such that, for all $\eta < \eta' < \theta$, we have $q^*(\gamma_\eta,
    \gamma_{\eta'}) \leq i_{\eta'}$. Now, given $\alpha < \gamma^*$,
    let $\eta_\alpha < \theta$ be the least $\eta$ for which
    $\alpha < \gamma_\eta$ and then set $q(\alpha, \gamma^*) :=
    \max\{i_{\eta_\alpha}, q^*(\alpha, \gamma_{\eta_\alpha})\}$. It is again straightforward to prove
    that the coloring $q$ thus defined is subadditive. The verification involves
    a case analysis; to illustrate the type of argument involved, we go through
    the proof of one of the required inequalities in one of the cases, leaving the
    other similar arguments to the reader.

    Suppose that $\alpha < \beta < \gamma^*$ and we have $\gamma_{\eta_\alpha}
    < \beta$. We will prove that $q(\alpha, \gamma^*) \leq \max\{q(\alpha, \beta),
    q(\beta, \gamma^*)\}$. If $q(\alpha, \gamma^*) = i_{\eta_\alpha}$, then
    this is trivial, so assume that $q(\alpha, \gamma^*) = q^*(\alpha, \gamma_{\eta_\alpha})$.
    Now, since each $p_\eta$ is subadditive (and hence $q^*$ is subadditive), we have
    \begin{align*}
      q^*(\alpha, \gamma_{\eta_\alpha}) &\leq \max\{q^*(\alpha, \beta), ~ q^*(\gamma_{\eta_\alpha},
      \beta)\} \\
      &\leq \max\{q^*(\alpha, \beta), ~ q^*(\gamma_{\eta_\alpha}, \gamma_{\eta_\beta}),
      ~ q^*(\beta, \gamma_{\eta_\beta})\} \\
      &\leq \max\{q^*(\alpha, \beta), ~ i_{\eta_\beta}, ~ q^*(\beta, \gamma_{\eta_\beta})\} \\
      &= \max\{q(\alpha, \beta), ~ q(\beta, \gamma^*)\}.
    \end{align*}
    Note that the final equality above holds because $q(\alpha, \beta) = q^*(\alpha, \beta)$
    and $q(\beta, \gamma^*) = \max\{i_{\eta_\beta}, q^*(\beta, \gamma_{\eta_\beta})\}$.

    We now verify that $q$ is $E^{\gamma^*+1}_{\geq \theta}$-closed. As in the
    previous case, we fix an $A \subseteq \gamma^*$ and an $i < \theta$ such that
    $A \subseteq D^{q}_{\leq i}(\gamma^*)$ and $\beta := \sup(A)$ is in $E^{\gamma^*}_{\geq
    \theta}$. It will suffice to show that $\beta \in D^{q}_{\leq i}(\gamma^*)$.
    To avoid triviality, assume that $\beta \notin A$. Since $\beta$ is a limit ordinal
    we know that $\beta$ is not equal to $\gamma_\eta$ for any $\eta < \theta$.
    It follows that, by passing to a tail of $A$ if necessary, we may assume
    that $\eta_\alpha = \eta_\beta$ for all $\alpha \in A$. Then, for
    all $\alpha \in A \cup \{\beta\}$, we have $q(\alpha, \gamma^*) = \max\{i_{\eta_\beta},
    p_{\gamma_\beta}(\alpha, \gamma_{\eta_\beta})$. It follows that $i \geq i_{\eta_\beta}$
    and, since $p_{\gamma_\beta}$ is $E^{\gamma_\beta + 1}_{\geq \theta}$-closed,
    that $p_{\eta_\beta}(\beta) \leq i$, so $\beta \in D^{q}_{\leq i}(\gamma^*)$.

    The fact that $|D^q_{\leq i}(\gamma^*)| < \lambda_i$ for all $i < \theta$
    follows by exactly the same reasoning as in the previous case.
  \end{cproof}

  Let $\dot{c}$ be the canonical $\mathbb{P}$-name for the union of the generic filter.
  Then $\dot{c}$ is forced to be a subadditive function from an initial segment of
  $[\kappa]^2$ to $\theta$ (we will see shortly that its domain is forced to be all of
  $[\kappa]^2$).

  Note that in the $\xi = \theta$ case of the above claim, we actually proved something
  stronger that will be useful later: if $\vec{p} = \langle p_\eta \mid \eta < \theta \rangle$
  is a strictly decreasing sequence of conditions in $\mathbb{P}$ and
  $\gamma:= \sup\{\gamma_{p_\eta} \mid \eta < \theta\}$, then there is a lower bound
  $q$ for $\vec{p}$ such that $q \Vdash_{\mathbb{P}} ``\gamma \in \partial(\dot{c})"$.

  The next claim will show that $\mathbb{P}$ is $({<}\kappa)$-distributive.

  \begin{claim}
    In Case~(1), $\mathbb{P}$ is $\kappa$-strategically closed.
    In Case~(2), $\mathbb{P}$ is $(\lambda + 1)$-strategically closed.
  \end{claim}

  \begin{cproof}
    In Case~(1), denote $\chi:=\kappa$.
    In Case~(2), denote $\chi:=\lambda$.
    We describe a winning strategy for Player II in $\Game_\chi(\mathbb{P})$.
    (Note that, if $\chi = \lambda$, then it appears that we are just showing
    that $\mathbb{P}$ is $\lambda$-strategically closed, but the fact that $\mathbb{P}$ is
    $\theta^+$-closed will then show that $\mathbb{P}$ is in fact
    $\lambda+1$-strategically closed.)
    We will arrange so that, if $\langle p_\eta \mid \eta < \chi \rangle$ is a
    play of the game in which Player II plays according to their prescribed strategy,
    then, letting $\gamma_\eta := \gamma_{p_\eta}$ for all $\eta < \chi$,
    \begin{enumerate}
      \item[(a)] $\langle \gamma_\eta \mid \eta < \chi \text{ is a nonzero even ordinal}
      \rangle$ is a continuous, strictly increasing sequence;
      \item[(b)] for all even ordinals $\eta < \xi < \chi$, we have
      $p_\xi(\gamma_\eta, \gamma_\xi) = \min\{i < \theta \mid \xi < \lambda_i\}$.
    \end{enumerate}
    Now suppose that $\xi < \chi$ is an even ordinal and $\langle p_\eta \mid
    \eta < \xi \rangle$ is a partial run of $\Game_\chi(\mathbb{P})$ that thus
    far satisfies requirements (a) and (b) above. We will describe a strategy for
    Player II to choose the next play, $p_\xi$, while maintaining (a) and (b).

	$\br$ If $\xi = 0$, then we are required to set $p_\xi = 1_{\mathbb{P}} = \emptyset$.

	$\br$ If $\xi = \xi' + 1$ is a successor ordinal, then, since
    $\xi$ is even, there is another even ordinal $\xi''$ such that
    $\xi = \xi'' + 2$. Let $\gamma^* := \gamma_{\xi'} + 1$. We will define $p_\xi$
    so that $\gamma_\xi = \gamma^*$. To do so, we must define $p_\xi(\alpha, \gamma^*)$
    for all $\alpha < \gamma^*$. We assume that $\gamma_{\xi''} < \gamma_\xi$
    (if they are equal, the construction is similar but easier). Let $i^* :=
    \min\{i < \theta \mid \xi < \lambda_i\}$. First, to satisfy (b), we must let
    $p_\xi(\gamma_{\xi''}, \gamma^*) := i^*$. Next, for all $\alpha < \gamma_{\xi''}$,
    let $p_\xi(\alpha, \gamma^*) := \max\{p_{\xi''}(\alpha, \gamma_{\xi''}), i^*\}$.
    Let $i^{**} := \max\{i^*, p_{\xi'}(\gamma_{\xi''}, \gamma_{\xi'})\}$, and set
    $p(\gamma_{\xi'}, \gamma^*) := i^{**}$. Finally, for all $\alpha \in (\gamma_{\xi''},
    \gamma_{\xi'})$, let $p_\xi(\alpha, \gamma^*) := \max\{i^{**}, p_{\xi'}(\alpha, \gamma_{\xi'})\}$.
    It is easily verified that $p_\xi$ thus defined is a condition in $\mathbb{P}$
    and that we have continued to satisfy requirements (a) and (b).

    $\br$ If $\xi$ is a nonzero limit ordinal, then let $p^* := \bigcup_{\eta < \xi}
    p_\eta$, and let $\gamma^* := \sup\{\gamma_\eta \mid \eta < \xi\}$.
    We will define a lower bound $p_\xi$ for the run of the game so far with
    $\gamma_\xi = \gamma^*$. To do so, it suffices to define $p_\xi(\alpha, \gamma^*)$
    for all $\alpha < \gamma^*$. Let $i^* := \min\{i < \theta \mid \xi < \lambda_i\}$.
    For all $\alpha < \gamma^*$, let $\eta_\alpha < \xi$
    be the least even ordinal $\eta$ such that $\alpha < \gamma_\eta$, and then
    set $p_\xi(\alpha, \gamma^*) := \max\{i^*, p^*(\alpha, \gamma_{\eta_\alpha})\}$.
    By the fact that the play of the game thus far satisfied (b), we know that
    $p^*(\gamma_{\eta}, \gamma_{\eta'}) \leq i^*$ for all even ordinals
    $\eta < \eta' < \xi$, so this definition does in fact ensure that
    $p_\xi(\gamma_\eta, \gamma_\xi) = i^*$, so we have satisfied (b).
    The fact that the play of the game thus far satisfied (a) and (b) also implies
    that $p_\xi$ is subadditive and $E^{\gamma_\xi+1}_{\geq \theta}$-closed and that
    we have continued to satisfy requirement (a). Finally, to show that
    $|D^{p_\xi}_{\leq i}(\gamma_\xi)| < \lambda_i$ for all $i < \theta$, note firstly
    that $D^{p_\xi}_{\leq i}(\gamma_\xi) = \emptyset$ for all $i < i^*$ and, secondly, that
    for all $i \in [i^*, \theta)$, we have $D^{p_\xi}_{\leq i}(\gamma_\xi)
    \subseteq \bigcup_{\eta < \xi} D^{p_\eta}_{\leq i}(\gamma_\eta)$. For each
    $i \in [i^*, \theta)$, we know that $\xi < \lambda_i$ and $\lambda_i$ is
    regular, so the fact that $p_\eta$ is a condition for all $\eta < \xi$ implies
    that $|D^{p_\xi}_{\leq i}(\gamma_\xi)| < \lambda_i$.

    This completes the description of Player II's winning
    strategy and hence the proof of the claim.
  \end{cproof}

  By the argument of the proof of the above claim, it follows that, for every
  $\alpha < \kappa$, the set $E_\alpha$ of $p \in \mathbb{P}$ for which $\gamma_p \geq \alpha$
  is dense in $\mathbb{P}$. Therefore, the domain of $\dot{c}$ is forced to be $[\kappa]^2$.
  The definition of $\mathbb{P}$ also immediately implies that $\dot{c}$ is forced to be
  $E^\kappa_{\geq \theta}$-closed and, in Case~(2),
  $\dot{c}$ is also forced to be locally small.
  We now finish the proof of the theorem by showing that $\partial(\dot{c})$ is
  forced to be stationary. (By Lemma~\ref{lemma39}, this will imply that $\dot{c}$
  witnesses $\US(\kappa, \kappa, \theta, \theta)$.) To this
  end, fix a condition $p$ and a $\mathbb{P}$-name $\dot{D}$ forced
  to be a club in $\kappa$. We will find $q \leq p$ and $\gamma < \kappa$ such that
  $q \Vdash_{\mathbb{P}}``\gamma\in \dot{D} \cap \partial(\dot{c})"$.

  Using the fact that $\mathbb{P}$ is $\theta^+$-closed
  and the fact that each $E_\alpha$ as defined above is dense,
  fix a decreasing sequence $\langle p_\eta \mid \eta < \theta \rangle$ of
  conditions in $\mathbb{P}$ together with an increasing sequence
  $\langle \alpha_\eta \mid \eta < \theta \rangle$ of ordinals below $\kappa$ such that
  \begin{enumerate}
    \item $p_0 = p$;
    \item for all $\eta < \theta$, $p_{\eta + 1} \Vdash_{\mathbb{P}}``\alpha_\eta
    \in \dot{D}"$;
    \item for all $\eta < \theta$, $\gamma_\eta < \alpha_\eta < \gamma_{\eta + 1}$.
  \end{enumerate}
  Now let $\gamma := \sup\{\gamma_\eta \mid \eta < \theta\}$, so that $\gamma = \sup\{\alpha_\eta
  \mid \eta < \theta\}$. By the proof of Claim~\ref{directed_closed_claim}, we can
  find a lower bound $q$ for $\langle p_\eta \mid \eta < \theta \rangle$
  such that $q \Vdash_{\mathbb{P}}``\gamma \in \partial(\dot{c})"$.
  For all $\eta < \theta$, since $q$ extends $p_{\eta + 1}$, $q$ also forces
  $\alpha_\eta$ to be in $\dot{D}$. Since $\gamma = \sup\{\alpha_\eta \mid \eta < \theta\}$
  and $\dot{D}$ is forced to be a club, it follows that $q$ forces $\gamma$ to be in $\dot{D}$.
  We then have $q \Vdash_{\mathbb{P}}``\gamma \in \dot{D} \cap \partial(\dot{c})"$, as desired.
\end{proof}

\subsection{Large cardinals and consistency results} \label{consistency_sec}
In this subsection, we investigate the effect on $\US(\ldots)$ of
certain compactness principles, including stationary reflection, the existence
of highly complete or indecomposable ultrafilters, and the P-ideal dichotomy

Recall that, for stationary subsets $S,T$ of $\kappa$, $\refl(\theta,S,T)$ asserts that,
for every $\theta$-sized collection $\mathcal S$ of stationary subsets of $S$,
there exists $\beta\in T\cap E^\kappa_{>\omega}$ such that $S\cap\beta$ is stationary in $\beta$ for every $S\in\mathcal S$.
We write $\refl(\theta,S)$ for $\refl(\theta,S,\kappa)$.
\begin{thm} Suppose that $\Sigma$ is some stationary subset of $\kappa$,
$c:[\kappa]^2\rightarrow\theta$ is a $\Sigma$-closed witness to $\U(\kappa,2,\theta,2)$,
and $\refl(\theta,\Sigma,\kappa\setminus \partial(c))$  holds. Then:
\begin{enumerate}
\item If $\kappa$ is $\theta$-inaccessible,  then $c$ is not subadditive of the first kind.
\item If $\partial(c)$ is stationary,  then $c$ is not subadditive.
\item If $c$ is locally small, then $c$ is not subadditive.
\end{enumerate}
\end{thm}
\begin{proof}
  Suppose that $c$ is subadditive of the first kind.
  For each $\alpha<\kappa$, pick $i_\alpha<\theta$, for which the following set is stationary:
  $$
    S^{i_\alpha}_\alpha:=\{\sigma\in\Sigma\mid \alpha<\sigma\ \&\ c(\alpha,\sigma)\le i_\alpha\}.
  $$
  Next, using the pigeonhole principle, fix $H\in[\kappa]^\kappa$
  and $i<\theta$ such that $i_\alpha=i$ for all $\alpha\in H$.
  \begin{claim} For every $A\in[H]^\theta$, there is $\beta_A<\kappa$ above $\sup(A)$ such that
  $\sup_{\alpha\in A}c(\alpha,\beta_A)<\theta$.
  \end{claim}
  \begin{cproof} Given $A\in[H]^\theta$, as $\refl(\theta,\Sigma,\kappa\setminus \partial(c))$ holds,
  we may pick $\beta\in \kappa\setminus \partial(c)$ such that $S_\alpha^i\cap\beta$ is stationary in $\beta$ for all $\alpha\in A$.
  Now, as $\beta\notin\partial(c)$, we may find some $j<\theta$ and
  $\gamma \in [\beta, \kappa)$ such that $\sup(D^c_{\le j}(\gamma)) \cap \beta=\beta$.
  As $c$ is $\Sigma$-closed, it follows that there exists a club $E$ in $\beta$ such that $E\cap\Sigma\s D^c_{\le j}(\gamma)$.
  For every $\alpha\in A$, fix $\sigma_\alpha\in S^{i}_\alpha\cap E$.
  Since $c$ is subadditive of the first kind,
  we have $c(\alpha,\gamma)\le\max\{c(\alpha,\sigma_\alpha),c(\sigma_\alpha,\gamma)\}\le\max\{i,j\}$.
  So, setting $\beta_A = \gamma$, we have $\sup_{\alpha\in A}c(\alpha,\beta_A)<\theta$.
  \end{cproof}

  (1) Using Lemma~\ref{lemma24}(1), fix $\epsilon<\kappa$ such that, for cofinally many $\beta<\kappa$,
  $\{c(\alpha, \beta) \mid  \alpha \in H \cap \epsilon\}$ is unbounded in $\theta$.
  Assuming that $\kappa$ is $\theta$-inaccessible,
  we may find some $A\in[H\cap\epsilon]^\theta$ such that, for cofinally many $\beta<\kappa$,
  $\{c(\alpha, \beta) \mid  \alpha \in A\}$ is unbounded in $\theta$.
  In particular, we may find such a $\beta<\kappa$ above $\beta_A$.
  Set $j:=\max\{\sup_{\alpha\in A}c(\alpha,\beta_A),c(\beta_A,\beta)\}$. Pick $\alpha\in A$ such that $c(\alpha,\beta)>j$.
  As $c$ is subadditive of the first kind,
  $c(\alpha,\beta)\le\max\{c(\alpha,\beta_A),c(\beta_A,\beta)\}\le j$.
  This  is a contradiction.

  (2) Suppose that $\partial(c)$ is stationary, and pick $\beta\in\acc^+(H)\cap\partial(c)$.
  Fix a cofinal subset $A$ of $H\cap\beta$ of size $\theta$.
  As $\beta\in\partial(c)$, $\{c(\alpha, \beta) \mid \alpha \in A\}$ is unbounded in $\theta$.
  Note that $\beta_A>\sup(A)=\beta$.
  Set $i:=\sup_{\alpha\in A}c(\alpha,\beta_A)$.
  If $c$ were weakly subadditive of the second kind,
  then we could find $j<\theta$ such that $D^c_{\le i}(\beta_A)\cap\beta\s D^c_{\le j}(\beta)$,
  contradicting the fact that for every $j<\theta$,
  there exists $\alpha\in A\s D^c_{\le i}(\beta_A)\cap\beta$ with $c(\alpha,\beta)>j$.

  (3) Suppose that $\kappa=\lambda^+$, $\theta=\cf(\lambda)$ and $c$
  is locally small. Fix the least $\beta<\kappa$ such that $\otp(H\cap\beta)=\lambda$.
  For all $\epsilon<\beta$ and $i<\cf(\lambda)$, $|H\cap[\epsilon,\beta)|=\lambda>|D^c_{\le i}(\beta)|$,
  and hence there exists a cofinal subset $A$ of $H\cap\beta$ of size $\theta$
  such that $\{c(\alpha, \beta) \mid \alpha \in A\}$ is unbounded in $\theta$.
  Now, as in the proof of Clause~(3), $c$ cannot be weakly subadditive of the second kind.
\end{proof}

\begin{cor} For every $\theta\in\reg(\kappa)$ and stationary $\Sigma\s E^\kappa_{\ge\theta}$, if $\refl(\theta,\Sigma)$ holds, then there exists no $\Sigma$-closed witness to $\US(\kappa,2,\theta,2)$.\qed
\end{cor}

We next show that the existence of subadditive witnesses to $\U(\ldots)$ is ruled
out by the existence of certain ultrafilters.

\begin{defn}
  An ultrafilter $U$ over $\kappa$ is \emph{$\theta$-indecomposable} if it is  uniform and,
  for every sequence of sets $\langle A_i \mid i < \theta \rangle$ satisfying
  $\bigcup_{i < \theta} A_i \in U$, there is $B \in [\theta]^{<\theta}$ such that
  $\bigcup_{i \in B} A_i \in U$.
\end{defn}

\begin{lemma}\label{lemma64}
  Suppose that $c:[\kappa]^2\rightarrow\theta$ is a witness to $\U(\kappa,2,\theta,2)$.
  \begin{enumerate}
  \item   If there exists a $\theta^+$-complete uniform ultrafilter over $\kappa$,
  then $c$ is not weakly subadditive;
  \item   If there exists a $\theta^+$-complete uniform ultrafilter over $\kappa$
  and $\kappa$ is $\theta$-inaccessible,
  then $c$ is not weakly subadditive of the first kind;
  \item   If there exists a $\theta$-indecomposable ultrafilter over $\kappa$,
  then $c$ is not subadditive of the second kind.
  \end{enumerate}
\end{lemma}
\begin{proof}
(1)
  Suppose that $U$ is a $\theta^+$-complete ultrafilter over $\kappa$.
  For all $\alpha<\kappa$ and $i<\theta$, let
  $$
    A^i_\alpha:=\{\beta<\kappa\mid \alpha<\beta\ \&\ c(\alpha,\beta)\le i\},
  $$
  so that $\langle A^i_\alpha\mid i<\theta\rangle$ is a $\s$-increasing sequence,
  converging to $\kappa\setminus(\alpha+1)$.
  Since $U$ is, in particular, a $\theta$-indecomposable ultrafilter over $\kappa$,
  we may find some $i_\alpha<\theta$ such that $A^{i_\alpha}_\alpha\in U$.

  Next, using the pigeonhole principle, let us fix $H\in[\kappa]^\kappa$
  and $i<\theta$ such that $i_\alpha=i$ for all $\alpha\in H$.
  As $U$ is closed under intersections of length $\theta$,
  for every $A\in[H]^\theta$, we may let $\beta_A:=\min(\bigcap_{\alpha\in A}A^i_\alpha)$, so
  that $\sup_{\alpha\in A}c(\alpha,\beta_A)\le i$.

  By Lemma~\ref{lemma24}(2), fix $\beta\in H$ such that
  $\{c(\alpha, \beta) \mid \alpha \in H \cap \beta\}$ is unbounded in $\theta$.
  Now, pick $A\in[H\cap\beta]^\theta$ such that   $\{c(\alpha, \beta) \mid \alpha \in A\}$ is unbounded in $\theta$.
  As $c$ is weakly subadditive, we may now pick $j<\theta$ such that
  $D^c_{\le i+1}(\beta_A)\cap\beta\s D^c_{\le j}(\beta)$.
  Then $A\s D^c_{\le j}(\beta)$, contradicting the choice of $A$.

(2)
  Let $H$, $i$ and the notation $\beta_A$ be as in the proof of Clause~(1).
  Assuming that $\kappa$ is $\theta$-inaccessible,
  and using Lemma~\ref{lemma24}(1),
  we may find some $A\in[H]^\theta$ such that, for cofinally many $\beta<\kappa$,
  $\{c(\alpha, \beta) \mid  \alpha \in A\}$ is unbounded in $\theta$.
  In particular, we may find such a $\beta<\kappa$ above $\beta_A$.
  If $c$ were weakly subadditive of the first kind, then we could find $j<\theta$ such that
  $D_{\le i}^c(\beta_A)\s D_{\le j}^c(\beta)$.
  However, for every $j<\theta$, there exists $\alpha\in A$ such that $c(\alpha,\beta)>j$,
  so that $\alpha\in D_{\le i}^c(\beta_A)\setminus D_{\le j}^c(\beta)$.
  This is a contradiction.

(3)
  Suppose that $U$ is a $\theta$-indecomposable ultrafilter over $\kappa$.
  As in the proof of Clause~(1), we may fix $H\in[\kappa]^\kappa$
  and $i<\theta$ such that, for all $\alpha\in H$,
  $A^i_\alpha:=\{\beta<\kappa\mid \alpha<\beta\ \&\ c(\alpha,\beta)\le i\}$
  is in $U$.
  Towards a contradiction, suppose that $c$ is subadditive of the second kind.
  Then, for all $(\alpha,\beta)\in[H]^2$, we may pick some $\gamma\in A^{i}_\alpha\cap A^i_\beta$,
  and infer that  $c(\alpha,\beta)\le\max\{c(\alpha,\gamma), c(\beta,\gamma)\}\le i$.
  So, $\sup(c``[H]^2)\le i$, contradicting the fact that $c$ witnesses
  $\U(\kappa,2,\theta,2)$.
\end{proof}

\begin{cor} Under a suitable large cardinal hypothesis, each of the following propositions is consistent:
\begin{enumerate}
\item  For every $n<\omega$ and
every coloring $c:[\aleph_{\omega+1}]^2\rightarrow\aleph_n$ that is weakly subadditive,
there exists $A\in[\aleph_{\omega+1}]^{\aleph_{\omega+1}}$ such that $c``[A]^2$ is finite;
\item  For every $n<\omega$ and
every coloring $c:[\aleph_{\omega+1}]^2\rightarrow\aleph_n$ that is subadditive of the second kind,
there exists $A\in[\aleph_{\omega+1}]^{\aleph_{\omega+1}}$ such that $c``[A]^2$ is countable.
\end{enumerate}
\end{cor}
\begin{proof} (1) By the proof approach of Corollary~\ref{cor35},
it suffices to get a model in which, for every $n<\omega$,
any witness to $\U(\aleph_{\omega+1},2,\allowbreak\aleph_n,2)$ is not weakly subadditive.
By \cite[Corollary~5.13]{lh_trees_squares_reflection},
relative to the consistency of the existence of infinitely many supercompact cardinals,
it is consistent that, for every $n<\omega$, any witness to $\U(\aleph_{\omega+1},2,\allowbreak\aleph_n,2)$  is not subadditive.
An inspection of \cite[Lemma~5.11]{lh_trees_squares_reflection} on which \cite[Corollary~5.13]{lh_trees_squares_reflection} relies
makes it clear that, furthermore,
for every $n<\omega$, any witness to $\U(\aleph_{\omega+1},2,\allowbreak\aleph_n,2)$  is not weakly subadditive.

(2) Starting from the  consistency of the existence
of a cardinal $\kappa$ that is $\kappa^+$-supercompact,
Ben-David and Magidor produced in \cite{ben-david_magidor} a model of $\zfc$ in which there exists
an ultrafilter over $\aleph_{\omega + 1}$ that is $\aleph_{n}$-indecomposable for
all positive $n < \omega$. Now appeal to Lemma~\ref{lemma64}
and the proof approach of Corollary~\ref{cor35}.
\end{proof}

\begin{lemma}\label{lemma338} Suppose that $\theta\le\mu<\kappa$,
and $c:[\kappa]^2\rightarrow\theta$ witnesses $\U(\kappa,\mu\circledast 1,\theta,2)$.
If there exists a $\mu^+$-complete uniform ultrafilter over $\kappa$,
  then $c$ is not $\mu$-coherent.
\end{lemma}
\begin{proof} As in the proof of Lemma~\ref{lemma64}(1),
a $\mu^+$-complete uniform ultrafilter over $\kappa$
gives rise to $H\in[\kappa]^{\kappa}$ and $i<\theta$
such that, for every $A\in[H]^\mu$, there exists $\beta_A<\kappa$ above $\sup(A)$
such that $\sup_{\alpha\in A}c(\alpha,\beta_A)\le i$.
Now, as $c$ witnesses $\U(\kappa,\mu\circledast 1,\theta,2)$, we may find $A\in[H]^\mu$ and $\beta\in H$ above $\sup(A)$
such that $c(\alpha,\beta)>i$ for all $\alpha\in A$.
So
$$\{\alpha<\min\{\beta,\beta_A\}\mid c(\alpha,\beta)\neq c(\alpha,\beta_A)\}$$
covers $A$, which is a set of size $\mu$. Thereby, $c$ is not $\mu$-coherent.
\end{proof}
\begin{cor}[{Todorcevic, \cite[Remark~6.2.3]{todorcevic_book}}]\label{cor339}
Suppose that $\lambda$ is a singular limit of strongly compact cardinals.
Then there exists no locally small coloring $c:[\lambda^+]^2\rightarrow\cf(\lambda)$ that is $\lambda$-coherent.
\end{cor}
\begin{proof} Fix a strictly increasing sequence of strongly compact cardinals $\langle \lambda_i\mid i<\cf(\lambda)\rangle$
converging to $\lambda$.
Let $c:[\lambda^+]^2\rightarrow\cf(\lambda)$ be a locally small coloring.
Towards a contradiction, suppose that $c$ is $\lambda$-coherent, and then define $d:[\lambda^+]^2\rightarrow\cf(\lambda)$ via
$$d(\gamma,\delta):=\min\{ i<\cf(\lambda)\mid |\{\alpha<\gamma\mid c(\alpha,\gamma)\neq c(\alpha,\delta)\}|\le\lambda_i\}.$$
It is not hard to see that $d$ is subadditive,
and so, by Lemma~\ref{lemma64}(3), $d$ fails to witness $\U(\lambda^+,2,\cf(\lambda),2)$.
This means that there exist $H\in[\lambda^+]^{\lambda^+}$ and $i<\cf(\lambda)$ such that $d``[H]^2\s i$.
Define $e:[\lambda^+]^2\rightarrow\cf(\lambda)$ via $e(\alpha,\beta):=c(\alpha,\min(H\setminus\beta))$.
Then $e$ is a locally small coloring which is moreover $\lambda_i$-coherent.
By Lemma~\ref{locally}(1), $e$ in particular witnesses $\U(\lambda^+,\lambda_{i}\circledast 1,\cf(\lambda),2)$,
and then Lemma~\ref{lemma338}  implies that there exists no $(\lambda_{i})^+$-complete uniform ultrafilter over $\lambda^+$,
contradicting the facts that $\lambda_i<\lambda_{i+1}<\lambda^+$ and $\lambda_{i+1}$ is strongly compact.
\end{proof}
Complementary to Lemma~\ref{locallysmall}, we obtain the following.

\begin{cor}\label{cor64} Suppose that $c:[\kappa]^2\rightarrow\theta$ is a witness to $\U(\kappa,2,\theta,2)$
and there exists a strongly compact cardinal in the interval $(\theta,\kappa]$.
  \begin{enumerate}
  \item $c$ is not weakly subadditive;
  \item If $\kappa$ is not the successor a singular cardinal of cofinality $\theta$,
  then $c$ is not weakly subadditive of the first kind;
  \item $c$ is not subadditive of the second kind.
  \end{enumerate}
\end{cor}
\begin{proof} The hypothesis entails the existence of a $\theta^+$-complete uniform ultrafilter over $\kappa$,
and in particular, the existence of a $\theta$-indecomposable uniform ultrafilter over $\kappa$.
In addition, if $\kappa$ is not the successor a singular cardinal of cofinality $\theta$,
then by Solovay's theorem that $\sch$ holds above a strongly compact cardinal \cite{MR0379200},
$\kappa$ is $\theta$-inaccessible.
Now appeal to Lemma~\ref{lemma64}.
\end{proof}
\begin{remark} A similar statement can consistently hold at small cardinals.
By \cite[Lemma~3.2(v)]{MR3667758} and Lemma~\ref{cor35},
if $\kappa$ is weakly compact and $\lambda<\kappa$ is any regular uncountable cardinal,
then in the forcing extension by $\mathrm{Coll}(\lambda,{<}\kappa)$,
for every $\theta<\lambda$, $\US(\lambda^+, 2, \theta, 2)$ fails.
\end{remark}

Lemma~\ref{fact219}(1) implies that the restriction ``$\theta<\lambda$'' in the preceding
remark cannot be waived.

\begin{cor} It is consistent that all of the following hold simultaneously:
\begin{itemize}
\item $\gch$;
\item $\US(\omega_2,\omega_2,\omega_1,\omega_1)$ holds;
\item $\US(\omega_2,2,\omega,2)$ fails.
\end{itemize}
\end{cor}
\begin{proof} By Lemma~\ref{fact219}(1) and Fact~\ref{fact321}.
\end{proof}

In \cite[Definition~8]{MR2385636}, Viale defined the covering property $\cp(\kappa,\theta)$.
In \cite[Lemma~5.11]{lh_trees_squares_reflection}, the first author proved that for infinite regular cardinals $\theta<\kappa$,
$\cp(\kappa,\theta)$ implies that $\US(\kappa,2,\theta,2)$ fails.
By \cite[\S6]{MR2385636},  the P-ideal dichotomy ($\pid$), which is a
consequence of the proper forcing axiom ($\PFA$), implies that
$\cp(\kappa,\omega)$ holds for every regular $\kappa\geq \aleph_2$
(the relevant result in \cite{MR2385636} is only stated for $\kappa > \mathfrak{c}$,
but its proof works without any modifications for any regular $\kappa \geq \aleph_2$).
Putting this all together shows that the conclusion of Corollary~\ref{cor36} already follows from $\pid$.

Here, by combining the arguments of \cite[\S7]{todorcevic_book},\cite[\S6]{MR2385636} and \cite[\S5]{lh_trees_squares_reflection},
we provide a self-contained proof of a slightly more informative result:

\begin{thm}[Todorcevic, Viale]\label{pid_cor} Suppose that $\pid$ holds
and $c:[\kappa]^2\rightarrow\omega$ witnesses $\U(\kappa, 2, \omega, 2)$.
\begin{enumerate}
\item  If $\kappa\ge\aleph_2$, then $c$ is not weakly subadditive;
\item If $\kappa>2^{\aleph_0}$ is not the successor of a singular cardinal of countable cofinality, then $c$ is not weakly subadditive of the first kind.
\end{enumerate}
\end{thm}
\begin{proof} Suppose not.
For all $X\in[\kappa]^{\le\aleph_0}$ and $\beta<\kappa$, define a function $f^\beta_X:X\cap\beta\rightarrow\omega$
by letting $f^\beta_X(\alpha):=c(\alpha,\beta)$.
Note that if $\kappa>2^{\aleph_0}$, then for every $X\in[\kappa]^{\le\aleph_0}$, there exists some $\Gamma\in[\kappa]^\kappa$ such that $f^\gamma_X=f^\delta_X$ for all $(\gamma,\delta)\in[\Gamma]^2$.
In this case, we shall denote such a set $\Gamma$ by $\Gamma_X$ and $\min(\Gamma_X)$ by $\gamma_X$.

Now, let $\mathcal I$ be the collection of all $X\in[\kappa]^{\le\aleph_0}$ such that,
for every $\beta<\kappa$, $f^\beta_X$ is finite-to-one.
It is clear that $\mathcal I$ is an ideal.

\begin{claim} Let $Z\in[\kappa]^{<\kappa}$.
Then there exists an ordinal $\epsilon_Z\in[\ssup(Z),\kappa)$ such that, for every $X\in[Z]^{\le\aleph_0}$,
$X\in\mathcal I$ iff there exists $\gamma\in[\epsilon_Z,\kappa)$ for which $f^\gamma_X$ is finite-to-one.
\end{claim}
\begin{cproof} $\br$ If $c$ is weakly subadditive, then set $\epsilon_Z:=\ssup(Z)$.
Towards a contradiction, suppose that there exist $X\in[Z]^{\le\aleph_0}$ and $\gamma\in[\epsilon_Z,\kappa)$ such that
$f^\gamma_X$ is finite-to-one, yet $X\not\in\mathcal I$. Fix $\beta<\kappa$ such that $f^\beta_X$ is not finite-to-one,
and then fix $i<\omega$ for which  $Y:=X\cap D^c_{\le i}(\beta)$ is infinite.
Find $j<\theta$ such that $D^c_{\le i}(\beta)\cap\gamma\s D^c_{\le j}(\gamma)$.
Then $X\cap D^c_{\le j}(\gamma)$ covers the infinite set $Y$, contradicting the fact that $f_X^\gamma$ is finite-to-one.

$\br$ If $\kappa>2^{\aleph_0}$ is not the successor of a singular cardinal of countable cofinality,
then by Viale's theorem \cite{MR2385636} that $\pid$ implies $\sch$,
the fact that $|Z|<\kappa$ implies that $|Z|^{\aleph_0}<\kappa$.
It follows that $\epsilon_Z:=\sup\{\gamma_X\mid X\in[Z]^{\le\aleph_0}\}+1$ is $<\kappa$.
Suppose that $c$ is weakly subadditive of the first kind, yet,
there exist $X\in[Z]^{\le\aleph_0}\setminus\mathcal I$ and $\gamma\in[\epsilon_Z,\kappa)$ such that
$f^\gamma_X$ is finite-to-one. Fix $\beta<\kappa$ and $i<\omega$ such that $Y:=X\cap D^c_{\le i}(\beta)$ is infinite.
Pick $\delta\in\Gamma_X$ above $\beta$, and use weak subadditivity of the first kind to find $j<\theta$
such that $D^c_{\le i}(\beta)\s D^c_{\le j}(\delta)$.
As $\delta\in\Gamma_X$, $X\cap D^c_{\le j}(\delta)=X\cap D^c_{\le j}(\gamma_X)$,
so that $Y\s D^c_{\le j}(\gamma_X)$.
As $\gamma_X<\epsilon_Z\le\gamma$, we use weak subadditivity of the first kind to find $k<\theta$
such that $D^c_{\le j}(\gamma_X)\s D^c_{\le k}(\gamma)$.
Altogether, $Y\s D^c_{\le k}(\gamma)$,
so that $f_X^\gamma[Y]\s k+1$, contradicting the fact that $f^\gamma_X$ is finite-to-one.
\end{cproof}

To see that $\mathcal I$ is a $P$-ideal,
suppose that $\vec X=\langle X_n\mid n<\omega\rangle$ is a sequence of sets in $\mathcal I$.
Set $\gamma:=\epsilon_Z$, for $Z:=\bigcup_{n<\omega}X_n$.
Evidently $Y:=\bigcup_{n<\omega}(X_n\setminus D^c_{\le n}(\gamma))$ is a pseudo-union for $\vec X$.
In addition, for every $i<\omega$, $Y\cap D^c_{\le i}(\gamma)$ is covered by the finite set $\bigcup_{n<i}(X_n\cap D^c_{\le i}(\gamma))$,
so, by the preceding claim, $Y\in\mathcal I$.

Finally, by $\pid$, one of the following alternatives must hold:
\begin{enumerate}
\item There exists $A\in[\kappa]^\kappa$ such that $[A]^{\aleph_0}\cap\mathcal I=\emptyset$, or
\item There exists $B\in[\kappa]^{\aleph_1}$ such that $[B]^{\aleph_0}\s\mathcal I$.
\end{enumerate}

In Case~(1), given $A\in[\kappa]^\kappa$, pick some strictly increasing function $g:\kappa\rightarrow A$
such that $g(\alpha)>\epsilon_{A\cap\ssup(g[\alpha])}$ for all $\alpha<\kappa$.
In effect, $A':=\im(g)$ is a cofinal subset of $A$ such that $\epsilon_{A'\cap\gamma}<\gamma$ for every $\gamma\in A'$.
Next, by Lemma~\ref{lemma24}(2),
we may fix $\gamma \in A'$ such that $\{c(\alpha, \gamma) \mid \alpha \in A' \cap \gamma\}$ is infinite.
So, we may find $X\in[A'\cap\gamma]^{\aleph_0}$ such that $f^\gamma_X$ is one-to-one.
As $\epsilon_{A'\cap\gamma}<\gamma$, the preceding claim implies that $X\in\mathcal I$.
In particular, $[A]^{\aleph_0}\cap\mathcal I\neq \emptyset$.

In Case~(2), given $B\in[\kappa]^{\aleph_1}$, let $\beta:=\ssup(B)$.
As $B\cap\bigcup_{i<\omega}D^c_{\le i}(\beta)$ is uncountable,
we may find some $i<\omega$ such that $B\cap D^c_{\le i}(\beta)$ is uncountable.
In particular, there exist $X\in[B]^{\aleph_0}$ such that $X\cap D^c_{\le i}(\beta)$ is infinite.
So, $[B]^{\aleph_0}\nsubseteq\mathcal I$.
\end{proof}

On the other hand, since $\MM$ is
preserved by $\omega_2$-directed closed set forcings,
Theorem~\ref{thm310} implies that $\MM$ does not refute
$\US(\kappa, \kappa, \theta, \theta)$ for regular uncountable cardinals $\theta<\kappa$.

\begin{cor}\label{cor328} In the model of \cite{MR1809418}, for every regular uncountable cardinal $\kappa$, the following are equivalent:
\begin{itemize}
\item There is a witness to $\U(\kappa,\kappa,\omega,\omega)$ that is subadditive of the first kind;
\item There is a witness to $\U(\kappa,2,\omega,2)$ that is weakly subadditive of the first kind;
\item $\kappa$ is the successor of a cardinal of countable cofinality.
\end{itemize}
\end{cor}
\begin{proof}  By Lemma~\ref{locallysmall}, if $\kappa=\lambda^+$ for an infinite cardinal $\lambda$ of countable cofinality,
then there exists a witness to $\U(\kappa,\kappa,\omega,\omega)$ that is subadditive of the first kind.
For any other $\kappa$,
since in the model of \cite{MR1809418}, $\ch$ and $\pid$ both hold,
Theorem~\ref{pid_cor}(2) implies that no witness to $\U(\kappa,2,\omega,2)$
is weakly subadditive of the first kind.
\end{proof}

We conclude this section by pointing out a corollary to Theorem~\ref{pid_cor} and the arguments from the proofs of Lemma~\ref{lemma331} and Corollary~\ref{cor339}.

\begin{cor} Assuming $\pid$, for every singular cardinal $\lambda$ of countable cofinality,
every locally small coloring $c:[\lambda^+]^2\rightarrow\cf(\lambda)$ that is subadditive of the first kind
is not $\lambda$-coherent.\qed
\end{cor}

\section{Indexed square sequences}\label{indexedsquaresection}
In this section, we identify a weakening of $\square^{\ind}(\kappa,\theta)$, denoted by
  $\inds(\kappa, \theta)$, that is tightly related
  to the existence of closed, subadditive witnesses to $\U(\kappa,\ldots)$.

\begin{defn}\label{inds}
  $\inds(\kappa, \theta)$ asserts the existence of a matrix
  \[
    \langle C_{\alpha,i}\mid \alpha\in\Gamma, ~ i(\alpha) \leq i<\theta\rangle
  \]
  satisfying the following statements.
  \begin{enumerate}
    \item $(E^\kappa_{\neq \theta} \cap \acc(\kappa)) \s \Gamma \s \acc(\kappa)$;
    \item For all $\alpha\in\Gamma$, we have $i(\alpha) < \theta$, and $\langle C_{\alpha,i}\mid
      i(\alpha) \leq i<\theta\rangle$ is a $\s$-increasing sequence of clubs in $\alpha$, with
      $\Gamma \cap \alpha = \bigcup_{i<\theta}\acc(C_{\alpha,i})$;
    \item For all $\alpha\in\Gamma$, $i(\alpha) \leq i<\theta$, and $\bar\alpha \in
      \acc(C_{\alpha,i})$, we have $i(\bar\alpha) \leq i$ and $C_{\bar\alpha,i}=C_{\alpha,i} \cap\bar\alpha$;
    \item For every club $D$ in $\kappa$, there exists $\alpha\in\acc(D) \cap \Gamma$ such that,
      for all $i<\theta$, $D\cap\alpha\neq C_{\alpha,i}$.
  \end{enumerate}
\end{defn}

\begin{remark}\label{remark42} The principle $\square^{\ind}(\kappa, \theta)$ from \cite{narrow_systems} is the strengthening of
  $\inds(\kappa, \theta)$ obtained by requiring that $\Gamma = \acc(\kappa)$.
\end{remark}

We now turn to prove Theorem~A.

\begin{thm}\label{thm67}
  Suppose that $\theta \in \reg(\kappa)$. Then the following are equivalent.
  \begin{enumerate}
    \item $\inds(\kappa,\theta)$ holds;
    \item There exists a closed, subadditive witness to $\U(\kappa,2,\theta,2)$;
    \item There exists a closed, subadditive witness to $\U(\kappa,\kappa,\theta,\sup(\reg(\kappa))$;
    \item For every stationary $S\s\kappa$, there exists a  $\inds(\kappa, \theta)$-sequence $\langle C_{\alpha, i} \mid
  \alpha \in \Gamma,\allowbreak  i(\alpha) \leq i < \theta \rangle$ such that,
  either $S\cap\Gamma$ is nonstationary, or, for all
  $i < \theta$, $\{\alpha \in S\cap\Gamma\mid i(\alpha) > i\}$ is stationary.
  \end{enumerate}
\end{thm}

\begin{proof} $(2)\iff(3)$ By Lemma~\ref{uPLUSsub}(3).

  $(4)\implies(1)$ This is trivial.

  $(1)\implies(2)$:
  Fix a $\inds(\kappa,\theta)$-sequence, $\vec{C} =
  \langle C_{\alpha,i}\mid \alpha \in \Gamma, ~ i(\alpha) \leq i<\theta\rangle$. For each
  $\alpha \in \kappa$, let $\tilde{\alpha} := \min(\Gamma \setminus \alpha)$. Define
  $c:[\kappa]^2 \rightarrow \theta$ by setting, for all $(\alpha, \beta) \in [\kappa]^2$,
  $$
    c(\alpha, \beta) := \begin{cases}
      \min\{i \in [i(\tilde{\beta}), \theta) \mid \tilde{\alpha} \in \acc(C_{\tilde{\beta}, i})\} & \text{if }
      \tilde{\alpha} < \tilde{\beta}; \\
      0 & \text{otherwise}.
    \end{cases}
  $$

  \begin{claim}
    $c$ is closed.
  \end{claim}

  \begin{cproof}
    Fix $\beta < \kappa$, $i < \theta$, and $A \s D^c_{\leq i}(\beta)$ with
    $\gamma : = \sup(A)$ in $\beta \setminus A$. To show that $\gamma \in
    D^c_{\leq i}(\beta)$, there are two cases to consider.

    $\br$ If $\tilde{\alpha} < \gamma$ for all $\alpha \in A$, then, by our
    definition of $c$, it follows that $\gamma \in \acc(C_{\tilde{\beta}, i})$, and
    hence $\gamma \in \Gamma$, $\gamma = \tilde{\gamma}$, and $c(\gamma, \beta)
    \leq i$.

    $\br$ Otherwise, there is $\alpha \in A$ such that $\tilde{\alpha} \geq \gamma$.
    But then $\tilde{\gamma} = \tilde{\alpha}$, and hence $c(\gamma, \beta) =
    c(\alpha, \beta) \leq i$.
  \end{cproof}

  \begin{claim}
    $c$ is subadditive.
  \end{claim}

  \begin{cproof}
    Suppose that $\alpha < \beta < \gamma < \kappa$. To prove subadditivity, there
    are three cases to consider.

    $\br$ If $\tilde\alpha=\tilde\beta$, then $c(\alpha,\beta)=0$ and $c(\alpha,\gamma)=c(\beta,\gamma)$.

    $\br$ If $\tilde\beta=\tilde\gamma$, then $c(\alpha,\beta)=c(\alpha,\gamma)$.

    $\br$ Otherwise, we have $|\{\tilde\alpha,\tilde\beta,\tilde\gamma\}|=3$ in which
    case it is easy to verify that $c(\alpha,\beta)\le\max\{c(\alpha,\gamma),c(\beta,\gamma)\}$ and
    $c(\alpha,\gamma)\le\max\{c(\alpha,\beta),c(\beta,\gamma)\}$.
  \end{cproof}

  To finish the proof, suppose towards a contradiction that $c$ fails to witness
  $\U(\kappa, 2, \theta, 2)$. Then there exist $A\in[\kappa]^\kappa$ and $i<\theta$
  such that $\sup(c``[A]^2)\le i$. Set $S:=\acc^+(A)\cap \Gamma$,
  note that $S$ is stationary, and let $D:=\bigcup\{ C_{\alpha,i}\mid \alpha\in S\}$.
  Then $D$ is a club and $D \cap \alpha = C_{\alpha, i}$ for all $\alpha \in \acc(D)$,
  contradicting the hypothesis that $\vec{C}$ is a $\inds(\kappa,\theta)$-sequence.

  $(2)\implies(4)$: Fix a closed, subadditive coloring $c$ witnessing $\U(\kappa, 2, \theta, 2).$
  Set $\Gamma:=\acc(\kappa)\setminus\partial(c)$,
  so that $\acc(\kappa) \cap E^\kappa_{\neq \theta} \s \Gamma \s \acc(\kappa)$.
  By Lemma~\ref{lemma38},
  $$
    \Gamma = \{\alpha \in \acc(\kappa) \mid \text{for some } i < \theta, ~ \sup(D^c_{\leq i}(\alpha)) = \alpha\}.
  $$

  Now, let $S$ be a given stationary subset of $\kappa$.
  If $S\cap\Gamma$ is stationary, then set $S':=S\cap\Gamma$; otherwise, set $S':=\kappa$.
  Using Lemma~\ref{lemma24}(3), let us fix $\epsilon<\kappa$ such that, for every $i<\theta$, $\{\beta\in S'\mid \epsilon<\beta, ~ c(\epsilon,\beta)>i\}$ is stationary.

  Let $\alpha \in \Gamma$ be arbitrary.
  Let $i'(\alpha)$ be the least $i < \theta$ for which
  $\sup(D^c_{\leq i}(\alpha)) = \alpha$.
  Next, if $\alpha\le\epsilon$, then let $i(\alpha):=i'(\alpha)$,
  and otherwise, let $i(\alpha):=\max\{i'(\alpha),c(\epsilon,\alpha)\}$.
  For all $i \in [i(\alpha), \theta)$,
  let $C_{\alpha, i} := D^c_{\leq i}(\alpha)$.
  Clearly, $\langle C_{\alpha, i} \mid i(\alpha) \leq
  i < \theta \rangle$ is a $\s$-increasing sequence of clubs in $\alpha$.

  We claim that $\langle C_{\alpha, i}
  \mid \alpha \in \Gamma, ~ i(\alpha) \leq i < \theta \rangle$ is a
  $\inds(\kappa, \theta)$-sequence.

  \begin{claim}\label{claim593}
    Let $\beta \in \Gamma$. Then $\Gamma \cap \beta = \bigcup_{i(\beta) \leq i < \theta}
    \acc(C_{\beta, i})$, 
    and, for all $i\in[i(\beta),\theta)$ and $\alpha \in \acc(C_{\beta, i})$,
    we have $i(\alpha) \leq i$ and $C_{\alpha, i} = C_{\beta, i} \cap \alpha$.
  \end{claim}

  \begin{cproof}
    To show the forward inclusion, let $\alpha \in \Gamma \cap \beta$
    be arbitrary. Put $i:=\max\{i(\alpha), i(\beta),c(\alpha, \beta) \}$.
    For all $\gamma \in C_{\alpha, i}$, we have $c(\gamma, \alpha) \leq i$ and
    $c(\alpha, \beta) \leq i$, so, by subadditivity, $c(\gamma, \beta) \leq i$.
    It follows that $\alpha \in \acc(C_{\beta, i})$.

    To show the reverse inclusion and the second statement, let $i\in[i(\beta),\theta)$
    and $\alpha \in \acc(C_{\beta, i})$ be arbitrary. For all $\gamma \in C_{\beta, i} \cap \alpha$, we have
    $c(\gamma, \beta) \leq i$ and $c(\alpha, \beta) \leq i$, so, by subadditivity,
    $c(\gamma, \alpha) \leq i$. It follows that $C_{\beta, i} \cap \alpha
    \s D^c_{\leq i}(\alpha)$, so $\sup(D^c_{\leq i}(\alpha)) = \alpha$,
    and hence $\alpha \in \Gamma$, $i'(\alpha) \leq i$, and $C_{\alpha, i} \supseteq
    C_{\beta, i} \cap \alpha$.
    To see that $i(\alpha)\le i$, suppose that $\alpha>\epsilon$,
    and we shall show that $c(\epsilon,\alpha)\le i$.
    We have already observed that $c(\alpha,\beta)\le i$.
    From $i\ge i(\beta)$, we infer $c(\epsilon,\beta)\le i$.
    So, by subadditivity, $c(\epsilon,\alpha)\le i$.

    It remains to show that $C_{\alpha, i} \s C_{\beta, i}
    \cap \alpha$. But, if $\gamma \in C_{\alpha, i}$, then we have $c(\gamma, \alpha)
    \leq i$ and $c(\alpha, \beta) \leq i$, so, again by subadditivity, we have
    $c(\gamma, \beta) \leq i$, and we are done.

  \end{cproof}

  \begin{claim}
    Suppose that $D$ is a club in $\kappa$. Then there is $\alpha\in
    \acc(D)\cap \Gamma$ such that, for all $i \in [i(\alpha), \theta)$, $D\cap\alpha\neq C_{\alpha,i}$.
  \end{claim}

  \begin{cproof}
    Suppose not. For all $\alpha \in \acc(D)$, fix $j_\alpha \in [i(\alpha), \theta)$
    such that $D \cap \alpha = C_{\alpha, j_\alpha}$, and find $j < \theta$ and
    $A \in [\acc(D)]^\kappa$ such that $j_\alpha = j$ for all $\alpha \in A$.
    As $c$ witnesses $\U(\kappa, 2, \theta, 2)$, we can find $(\alpha, \beta)
    \in [A]^2$ such that $c(\alpha, \beta) > j$. But this contradicts the fact
    that $\alpha \in D \cap \beta = C_{\beta, j}$, and hence $c(\alpha, \beta) \leq j$.
\end{cproof}
Finally, by the choice of $\epsilon$, it is clear that
if $S\cap\Gamma$ is stationary, then for all
  $i < \theta$, $\{\alpha\in S\cap\Gamma\mid i(\alpha) > i\}$ is stationary.
\end{proof}

The proof of the preceding theorem together with Remark~\ref{remark42} makes it clear that the following holds as well.
\begin{thm} \label{thm_44}
  For every $\theta \in \reg(\kappa)$, the following are equivalent:
  \begin{enumerate}
    \item $\square^{\ind}(\kappa,\theta)$ holds;
    \item There exists a closed witness $c$ to $\US(\kappa,2,\theta,2)$ for which $\partial(c)$ is nonstationary.\qed
    \end{enumerate}
\end{thm}

\begin{thm}\label{thm45}
  If $\inds(\kappa,\omega)$ holds, then so does $\square^{\ind}(\kappa,\omega)$.
\end{thm}
\begin{proof}
 Suppose that $\inds(\kappa,\omega)$ holds.
 By Theorem~\ref{thm67}, we may fix a closed, subadditive coloring $c:[\kappa]^2\rightarrow\omega$ witnessing $\U(\kappa,2,\omega,2)$.
Let
  $
    \Gamma := \{\alpha \in \acc(\kappa) \mid \exists i < \omega[\sup(D^c_{\leq i}(\alpha)) = \alpha]\},
  $
  so that $\Gamma\supseteq \acc(\kappa) \cap E^\kappa_{\neq \omega}$.
For all $\delta\in E^\kappa_\omega\setminus\Gamma$, let
$a_\delta:=\{\max(D^c_{\le i}(\delta))\mid i<\omega\ \&\ D^c_{\le i}(\delta)\neq\emptyset\}$.
Clearly, $a_\delta$ is a cofinal subset of $\delta$ of order-type $\omega$.

For all $\alpha\in E^\kappa_\omega\setminus\Gamma$, let $i(\alpha):=0$.
For all $\alpha \in \Gamma$, let $i(\alpha)$ be the least $i < \omega$ for which
  $\sup(D^c_{\leq i}(\alpha)) = \alpha$.
  Then, for all $\alpha\in\acc(\kappa)$ and $i\in [i(\alpha), \omega)$,
  let $$C_{\alpha, i} := D^c_{\leq i}(\alpha)\cup\bigcup\{ a_\delta\mid \delta\in (D^c_{\leq i}(\alpha)\cup\{\alpha\})\setminus\Gamma\}.$$

Note that  $C_{\alpha,i}\s C_{\alpha,j}$ for all $j\ge i$.

  \begin{claim}
    Let $\beta \in \acc(\kappa)$. Then:
    \begin{enumerate}
    \item  $\acc(\beta) = \bigcup_{i(\beta) \leq i < \omega}    \acc(C_{\beta, i})$;
    \item  for all $i\in[i(\beta),\omega)$, $C_{\beta, i}$ is a club in $\beta$, and $\acc(C_{\beta,i})\s D^c_{\le i}(\beta)$;
    \item for all $i\in[i(\beta),\omega)$  and  $\alpha \in \acc(C_{\beta, i})$,
    we have $i(\alpha) \leq i$ and $C_{\alpha, i} = C_{\beta, i} \cap \alpha$.
    \end{enumerate}
  \end{claim}

  \begin{cproof}
    (1) Let $\alpha\in\acc(\beta)$ be arbitrary.
    Fix a large enough $i \geq \max\{i(\alpha), i(\beta)\}$ such that $c(\alpha, \beta) \leq i$.
    If $\alpha\not\in\Gamma$, then $a_\alpha\s C_{\beta,i}$.
    If $\alpha\in \Gamma$,
    then for all $\gamma \in C_{\alpha, i}$, we have $c(\gamma, \alpha) \leq i$ and
    $c(\alpha, \beta) \leq i$, so, by subadditivity, $c(\gamma, \beta) \leq i$.
    Thus, in both cases, $\alpha \in \acc(C_{\beta, i})$.

    (2)  Let $i\in[i(\beta),\omega)$ be arbitrary.
    Suppose first that $\beta\in\Gamma$.
  Since $D^c_{\le i}(\beta)$ is a club in $\beta$, to show that $C_{\beta,i}$ is a club
  and $\acc(C_{\beta, i}) \subseteq D^c_{\leq i}(\beta)$, it suffices to show that for any pair $\gamma<\delta$ of successive elements of $D^c_{\le i}(\beta)$,
  if $(\gamma, \delta) \cap C_{\beta, i} \neq \emptyset$, then $\delta \notin \Gamma$
  and $(\gamma, \delta) \cap C_{\beta, i} \subseteq a_\delta$.

  Fix $\gamma,\delta$ as above along with $\alpha\in(\gamma,\delta)\cap C_{\beta,i}$.
  Using the definition of $C_{\beta,i}$, let us fix some $\delta'\in D^c_{\leq i}(\beta)\setminus\Gamma$ such that $\alpha\in a_{\delta'}$.
  We have $c(\delta',\beta)\le i$ and $c(\alpha,\beta)>i$. So, by subadditivity, $j:=c(\alpha,\delta')$ is greater than $i$.
  As $\alpha\in a_{\delta'}$, it follows that $\alpha=\max(D^c_{\le j}(\delta'))$.

    Notice that if $\delta<\delta'$, then from  $\delta,\delta'\in D^c_{\le i}$ and subadditivity, we have $c(\delta,\delta')\le i$,
    and so $\delta\le \max(D^c_{\le i}(\delta'))\le \max(D^c_{\le j}(\delta'))=\alpha$, which is a contradiction.
    So $\delta=\delta'$ and  $\alpha\in a_\delta$, as desired.

    Next, suppose that $\beta\in \acc(\kappa)\setminus\Gamma$. Then the very same argument as before shows that for any pair $\gamma<\delta$ of successive elements of $D^c_{\le i}(\beta)$,
    if the interval $(\gamma,\delta)\cap C_{\beta,i}$ is nonempty, then it is covered by $a_\delta$.
    Moreover, by the definition of $C_{\beta, i}$, we have $C_{\beta, i} \setminus
    \max(D^c_{\leq i}) \subseteq a_\beta$. It follows that $C_{\beta, i}$ is a club in $\beta$
    and $\acc(C_{\beta, i}) \subseteq D^c_{\leq i}(\beta)$.

    (3) Fix $i\in[i(\beta),\omega)$ and $\alpha\in\acc(C_{\beta, i})$. In particular, $\alpha\in D^c_{\le i}(\beta)$,
    so, by subadditivity, $D^c_{\le i}(\alpha)=D^c_{\le i}(\beta)\cap\alpha$.

        $\br$ If $\alpha\in\Gamma$, then $\alpha\in\acc(D^c_{\le i}(\beta))$,
        so that  $i(\alpha)\le i$, and it is clear from the definition that $C_{\alpha,i}=C_{\beta,i}\cap\alpha$.

        $\br$ If $\alpha\notin\Gamma$, then $i(\alpha)=0\le i$, and it is clear from the definition that $C_{\alpha,i}=C_{\beta,i}\cap\alpha$.
  \end{cproof}

  The following claim will now finish our proof.

  \begin{claim}
    Suppose that $D$ is a club in $\kappa$. Then there is $\alpha\in
    \acc(D)$ such that, for all $i \in [i(\alpha), \omega)$, $D\cap\alpha\neq C_{\alpha,i}$.
  \end{claim}

  \begin{proof} \renewcommand{\qedsymbol}{\ensuremath{\boxtimes \ \square}}
    Suppose not. Then, for all $\alpha\in\acc(D)$,
    for some $j_\alpha\in[i(\alpha),\omega)$, $D\cap\alpha=C_{\alpha,j_\alpha}$.
    Find $j < \omega$ and
    $A \in [\acc(D)]^\kappa$ such that $j_\alpha = j$ for all $\alpha \in A$.
    As $c$ witnesses $\U(\kappa, 2, \omega, 2)$, we may pick $(\alpha, \beta)\in [A]^2$ such that $c(\alpha, \beta) > j$.
    However $\alpha\in\acc(D\cap\beta)=\acc(C_{\beta,j})\s D^c_{\le j}(\beta)$, meaning that $c(\alpha,\beta)\le j$. This is a contradiction.
\end{proof}
\let\qed\relax
\end{proof}

An analogue of the preceding result holds for uncountable $\theta$ under the
additional assumption of stationary reflection.

\begin{cor}    Suppose that $\theta\in\reg(\kappa)$.
  If every stationary subset of $E^\kappa_\theta$ reflects,
  then $\inds(\kappa,\theta)$ is equivalent to $\square^{\ind}(\kappa,\theta)$.
\end{cor}
\begin{proof} Suppose that $\inds(\kappa,\theta)$ holds.
By Theorem~\ref{thm67}, we may then fix a closed, subadditive witness $c$ to $\U(\kappa,2,\theta,2)$.
As $\partial(c)$ is a subset of $E^\kappa_\theta$ and every stationary subset of $E^\kappa_\theta$ reflects,
it follows from Lemma~\ref{partial} that $\partial(c)$ is nonstationary. So, by
Theorem~\ref{thm_44}, $\square^{\ind}(\kappa,\theta)$ holds.
\end{proof}

We shall now turn to prove Clause~(3) of Theorem~C, in particular,
establishing that, in general, for uncountable $\theta$, $\inds(\kappa, \theta)$ is not
equivalent to $\square^{\ind}(\kappa, \theta)$.
This will follow from the following two
theorems; these are fairly straightforward modifications of results of Cummings
and Schimmerling \cite{cummings_schimmerling} and Levine and Sinapova
\cite{MR4205976}, respectively, but we provide some details for completeness.

\begin{thm} \label{outer_model_thm}
  Suppose that $V \subseteq W$ are models of $\zfc$ and $\theta < \lambda$ are
  regular cardinals in $V$ such that
  \begin{enumerate}
    \item $\lambda$ is inaccessible in $V$;
    \item $\lambda$ is singular and $\cf(\lambda) = \theta$ in $W$;
    \item $(\lambda^+)^V = (\lambda^+)^W$.
  \end{enumerate}
  Then $\inds(\lambda^+, \theta)$ holds in $W$.
  Moreover, in $W$ there is a
  closed, locally small witness $c$ to $\US(\lambda^+, 2, \theta, 2)$ such that
  $\partial(c) \supseteq (E^{\lambda^+}_\lambda)^V$.
\end{thm}

\begin{proof}
  Work first in $V$, and let $\Gamma := E^{\lambda^+}_{<\lambda}$.
  Recall that a subset $X \subseteq \lambda$ is \emph{$({>}\omega)$-club}
  if there is a club $C \subseteq \lambda$ such that $C \cap E^\lambda_{>\omega}
  \subseteq X$. By \cite[Lemma 4.4]{cummings_schimmerling}, we can fix a matrix
  $\vec{D} = \langle D_{\alpha, \eta} \mid \alpha \in \Gamma, ~ \eta \in X_\alpha \rangle$
  such that, for all $\alpha \in \Gamma$,
  \begin{enumerate}
    \item $X_\alpha$ is a $({>}\omega)$-club subset of $\lambda$;
    \item for all $\eta \in X_\alpha$, $D_{\alpha, \eta}$ is a club in $\alpha$ and $|D_{\alpha, \eta}| < \lambda$;
    \item $\langle D_{\alpha, \eta} \mid \eta \in X_\alpha \rangle$ is $\subseteq$-increasing
    and $\Gamma \cap \alpha = \bigcup_{\eta \in X_\alpha} \acc(D_{\alpha, \eta})$;
    \item for all $\eta \in X_\alpha$ and all $\bar{\alpha} \in \acc(D_{\alpha, \eta})$,
    we have $\eta \in X_{\bar{\alpha}}$ and $D_{\bar{\alpha}, \eta} = D_{\alpha, \eta}
    \cap \bar{\alpha}$.
  \end{enumerate}

  Now move to $W$ and note that, since $\cf(\lambda) = \theta$, we have
  $(E^{\lambda^+}_{\neq \theta} \cap \acc(\lambda^+)) \subseteq \Gamma
  \subseteq \acc(\lambda^+)$. By assumptions (1)--(3) in the statement of the
  theorem, \cite[Corollary 4.2]{pseudo_prikry} implies that we can find an
  increasing sequence of ordinals $\langle \eta_i \mid i < \theta \rangle$ that
  is cofinal in $\lambda$ and such that, for all $\alpha \in \Gamma$, for all
  sufficiently large $i < \theta$, we have $\delta_i \in X_\alpha$. (This fact
  is also implicit in the earlier \cite[Theorem 2.0]{dzamonja_shelah})

  For each $\alpha \in \Gamma$, let $i(\alpha)$ be the least $j < \theta$ such that
  $\delta_i \in X_\alpha$ for all $i \in [j, \theta)$. Now define a matrix
  $\vec{C} = \langle C_{\alpha, i} \mid \alpha \in \Gamma, i(\alpha) \leq i < \theta \rangle$
  by setting $C_{\alpha, i} := D_{\alpha, \delta_i}$ for all $\alpha \in \Gamma$ and
  all $i \in [i(\alpha), \theta)$. Notice that, for all $\alpha \in \Gamma$,
  $i \in [i(\alpha), \theta)$, and $\bar{\alpha} \in \acc(C_{\alpha, i})$,
  the properties of $\vec{D}$ imply that $X_\alpha \setminus i \subseteq X_{\bar{\alpha}}$,
  and therefore, since $i \geq i(\alpha)$ we also have $i \geq i(\bar{\alpha})$
  and $C_{\bar{\alpha}, i} = C_{\alpha, i} \cap \bar{\alpha}$. It is then
  straightforward to verify that $\vec{C}$ satisfies Clauses (1)--(3) of
  Definition~\ref{inds}. To verify that $\vec{C}$ satisfies Clause~(4) of
  Definition~\ref{inds}, fix a club $D \subseteq \lambda^+$, and let
  $\alpha \in \acc(D) \cap \Gamma$ be such that $|D \cap \alpha| = \lambda$.
  Then, for all $i \in [i(\alpha), \theta)$, the fact that $|C_{\alpha, i}| < \lambda$
  implies that $D \cap \alpha \neq C_{\alpha, i}$. It follows that $\vec{C}$ is an
  $\inds(\lambda^+, \theta)$-sequence in $W$.

  To prove the ``moreover" statement, let $c:[\lambda^+]^2 \rightarrow \theta$
  be the closed witness to
  $\US(\lambda^+, 2, \theta, 2)$ derived from $\vec{C}$ as in the proof of
  $(1)\implies(2)$ of Theorem~\ref{thm67}. The fact that $\partial(c) \supseteq
  (E^{\lambda^+}_\lambda)^V$ follows immediately from our definition of $\Gamma$,
  and the fact that $c$ is locally small follows immediately from the fact
  that $|C_{\alpha, i}| < \lambda$ for all $\alpha \in \Gamma$ and
  $i \in [i(\alpha), \theta)$.
\end{proof}

\begin{thm}\label{thm48} Suppose that $\theta<\lambda<\kappa$ are regular uncountable cardinals such that $\lambda$ is supercompact and $\kappa$ is weakly compact.
Then there is a forcing extension in which $\lambda$ is a singular strong limit of cofinality $\theta$,
$\kappa=\lambda^+$, $\inds(\kappa,\theta)$ holds, and $\square(\kappa,\tau)$ fails for all $\tau<\lambda$.
\end{thm}
\begin{proof}
  By forcing with the Laver preparation forcing if necessary, we may assume that the
  supercompactness of $\lambda$ is indestructible under $\lambda$-directed closed
  forcing. Following \cite[\S 4]{MR4205976}, let $\mathbb{C} := \mathrm{Coll}(\lambda,
  {<}\kappa)$, and let $\dot{\mathbb{M}}$ be a $\mathbb{C}$-name for the Magidor forcing
  that turns $\lambda$ into a singular cardinal of cofinality $\theta$.

  For every inaccessible $\delta < \kappa$ above $\lambda$, let $\mathbb{C}_\delta
  := \mathrm{Coll}(\lambda, {<}\delta)$. By \cite[Proposition~4.3]{MR4205976},
  there is a club $C \subseteq \kappa$ such that, for every inaccessible
  $\delta \in C$, $\dot{\mathbb{M}}_\delta := \dot{\mathbb{M}} \cap V_\delta$ is
  a $\mathbb{C}_\delta$-name for a Magidor forcing to turn $\lambda$ into a singular
  cardinal of cofinality $\theta$ such that there is a complete embedding of
  $\mathbb{C}_\delta \ast \dot{\mathbb{M}}_\delta$ into $\mathbb{C} \ast \dot{\mathbb{M}}$.

  In $V^{\mathbb{C}}$, $\mathbb{M}$ has the $\lambda^+$-cc and therefore
  preserves $\lambda^+$. Therefore, applying Theorem~\ref{outer_model_thm} to the
  models $V^{\mathbb{C}}$ and $V^{\mathbb{C} * \dot{\mathbb{M}}}$ shows that
  $\inds(\kappa, \theta)$ holds in $V^{\mathbb{C} * \dot{\mathbb{M}}}$. It
  remains to show that $\square(\kappa, \theta)$ fails in
  $V^{\mathbb{C} * \dot{\mathbb{M}}}$. In fact, we will show that
  $\square(\kappa, \tau)$ fails for every $\tau < \lambda$.

  Suppose for sake of contradiction that $\tau < \lambda$ and $\dot{\mathcal{D}} =
  \langle \dot{\mathcal{D}}_\alpha \mid \alpha \in \acc(\kappa) \rangle$
  is a $\mathbb{C} * \dot{\mathbb{M}}$-name for a $\square(\kappa, \tau)$-sequence.
  This is a $\Pi^1_1$ statement about the structure $(V_\kappa, \in,
  \mathbb{C} * \dot{\mathbb{M}}, \dot{\mathcal{D}})$ (the sole universal
  quantification over subsets of $V_\kappa$ is the assertion that there exists
  no $\mathbb{C} * \dot{\mathbb{M}}$-name for a thread through
  $\dot{\mathcal{D}}$). Therefore, by the weak compactness of $\kappa$, we can
  find an inaccessible $\delta \in C$ such that $\dot{\mathcal{D}}^* :=
  \dot{\mathcal{D}} \cap V_\delta$ is a
  $\mathbb{C}_\delta * \dot{\mathbb{M}}_\delta$-name for a $\square(\delta, \tau)$-sequence.

  Note that $\delta = \lambda^+$ in $V^{\mathbb{C}_\delta * \dot{\mathbb{M}}_\delta}$.
  In $V^{\mathbb{C} * \dot{\mathbb{M}}}$, we have $\cf(\delta) = \theta > \omega$.
  Therefore, in $V^{\mathbb{C} * \dot{\mathbb{M}}}$, any element of
  $\mathcal{D}_\delta$ is a thread through $\mathcal{D}^*$. However, the proof of
  \cite[Lemma~4.7]{MR4205976} shows that forcing over $V^{\mathbb{C}_\delta *
  \dot{\mathbb{M}}_\delta}$ with $(\mathbb{C} * \dot{\mathbb{M}})/(\mathbb{C}_\delta * \dot{\mathbb{M}}_\delta)$ cannot add a thread to a $\square(\delta, \tau)$-sequence.
  (\cite[Lemma~4.7]{MR4205976} is about $\square_{\lambda, \tau}$-sequences,
  but the exact same proof still works for $\square(\delta, \tau)$-sequences.)
  This is a contradiction, thus completing the proof.
\end{proof}

\begin{remark}Suppose that $\theta\in\reg(\kappa)$
is such that there exists a closed subadditive witness to $\U(\kappa,2,\theta,2)$.
By the proof Lemma~\ref{lemma35}, if
$\kappa$ is  $({<}\theta)$-inaccessible,
then there exists a $\kappa$-Aronszajn tree with a $\theta$-ascent path.
In particular, such a tree exists in the model of Theorem~\ref{thm48}.
We note that by a combination of Theorem~\ref{thm67}, \cite[Theorem~6.11]{paper23}
and a minor variation of \cite[Theorem~4.44]{paper23},
if $\kappa=\kappa^{<\kappa}$ is a successor cardinal which is $\theta$-inaccessible,
then there moreover exists a $\kappa$-Souslin tree with a $\theta$-ascent path.
\end{remark}

The \emph{Mapping Reflection Principle} ($\mrp$), introduced by Moore in \cite{moore_mrp}, is a useful consequence of
$\PFA$.

\begin{cor}\label{cor47} $\mrp$ implies that for every regular cardinal $\kappa\ge\aleph_2$, there exists no closed, subadditive witness to $\U(\kappa,2,\omega,2)$.
\end{cor}
\begin{proof} By Theorems \ref{thm67} and \ref{thm45}, if there exists a closed subadditive witness to $\U(\kappa,2,\omega,2)$, then
$\square^{\ind}(\kappa,\omega)$ holds. In particular, $\square(\kappa,\omega)$ holds.
However, by \cite[Theorem~1.8]{MR2895405}, for every regular cardinal $\kappa\ge\aleph_2$, $\square(\kappa,\omega)$ is refuted by $\mrp$.
This is also implicit in \cite[\S7]{MR2385636}.
\end{proof}

In \cite{MR2895405}, Strullu proves that $\mrp + \MA$ refutes $\square(\kappa, \omega_1)$ for all regular $\kappa \geq \aleph_2$.
In light of this fact and Theorem~\ref{pid_cor}, it is natural to raise the following question.

\begin{q} Does $\mrp+\MA$ refute $\inds(\kappa,\omega_1)$? How about $\PFA$ or $\MM$?\footnote{Recall that, by Theorem~\ref{thm310} and the discussion
after Theorem~\ref{pid_cor}, $\MM$ does not refute $\US(\kappa, \kappa, \omega_1, \omega_1)$.}
\end{q}

  By \cite[Theorem~3.4]{lh_lucke}, for every $\theta\in\reg(\kappa)$,
  $\square(\kappa)$ implies $\square^{\ind}(\kappa, \theta)$.
  Essentially the same proof of that theorem establishes
  that for every $\theta \in \reg(\kappa)$, $\square(\kappa, \sq_{\theta})$
  implies $\inds(\kappa, \theta)$.\footnote{$\square(\kappa, \sq_{\theta})$ is a weakening of $\square(\kappa)$; its definition may be found as Definition~1.4 of \cite{paper28}.}
  Here, we shall prove a generalization that also yields the second part of Theorem~A.
\begin{thm}\label{revised_lh_lucke}
  Suppose that $\kappa\ge\aleph_2$, $\theta \in \reg(\kappa)$, $T \subseteq E^\kappa_\theta$ and
  there is a $\square(\kappa, \sq_{\theta})$-sequence that avoids $T$.
  Then:
  \begin{enumerate}
  \item There is a $\inds(\kappa, \theta)$-sequence $\langle C_{\alpha, i} \mid
  \alpha \in \Gamma, ~ i(\alpha) \leq i < \theta \rangle$ such that
  $\Gamma \cap T = \emptyset$;
  \item There is a closed witness $c$ to $\US(\kappa,2,\theta,2)$ such that
  $T \subseteq \partial(c)$.
\end{enumerate}
\end{thm}

\begin{proof} The proof of Theorem~\ref{thm67} makes it clear that $(1)\implies(2)$,
  so we focus on proving Clause~(1).
  Fix a $\square(\kappa, \sq_{\theta})$-sequence $\vec C=\langle C_\beta \mid
  \beta < \kappa \rangle$ that avoids $T$, i.e., for all $\beta < \kappa$,
  we have $\acc(C_\beta) \cap T = \emptyset$.
  Set $\Delta:=\{\delta\in\acc(\kappa)\mid \forall
  \bar\delta\in\acc(C_\delta)[C_{\bar\delta}=C_\delta\cap\bar\delta]\}$.
  By the $\sq_\theta$-coherence of $\vec{C}$, we  have $E^\kappa_{\geq \theta} \subseteq \Delta$.
  Set $\Omega := E^\kappa_{>\theta} \cup \bigcup\{ \acc(C_\delta)\mid \delta\in
  \Delta\}$, and note that $\Omega\subseteq \Delta\setminus T$.
  As $\kappa\ge\aleph_2$, $\Omega$ is stationary.
  So, by \cite[Lemma~1.23 and Lemma~1.7]{paper29},
  we may fix $\zeta<\kappa$ such that
  $\{\beta \in \Omega \mid \otp(C_\beta) > \zeta \text{ and } C_\beta(\zeta) \geq \tau\}$
  is stationary for all $\tau < \kappa$. By successive applications of Fodor's
  Lemma, we may then recursively construct a strictly increasing sequence of ordinals
  $\langle \tau_i\mid i<\theta\rangle$ such that $\tau_0>\min(\Omega)$
  and such that, for every $i<\theta$,
  $\{\beta \in \Omega \mid \otp(C_\beta) > \zeta \text{ and } C_\beta(\zeta) = \tau_i\}$
  is stationary.

  Next, define $\vec{D} = \langle D_\beta \mid \beta < \kappa \rangle$, as follows:
  \begin{itemize}
  \item Let $D_0:=\emptyset$;
  \item For every $\beta<\kappa$, let $D_{\beta+1}:=\{\beta\}$;
  \item For every $\beta\in\Delta$ such that $\otp(C_\beta)\le\zeta$, let $D_\beta:=C_\beta$;
  \item For every $\beta\in\Delta$ such that $\otp(C_\beta)>\zeta$, let $D_\beta:=C_\beta\setminus C_\beta(\zeta)$;
  \item For every $\beta\in\acc(\kappa)\setminus\Delta$, pick a club $D_\beta$ in $\beta$ of order-type $\cf(\beta)$ such that $D_\beta\cap T=\emptyset$.
  \end{itemize}

  For every ordinal $\beta<\kappa$, let
  $$j(\beta):=\begin{cases}
  i,&\text{if }\min(D_\beta)=\delta_i;\\
  0,&\text{otherwise}.
  \end{cases}$$

  For every $i<\theta$, set $\Gamma_i:=\{\beta\in \Omega\mid j(\beta)=i\}$.
  Then set $\Gamma := \Omega \cup (\acc(\kappa)\cap E^\kappa_{<\theta})$.

  \begin{claim}\label{cf_claim}
   \begin{enumerate}
    \item $(E^\kappa_{\neq \theta} \cap \acc(\kappa)) \s \Gamma \s \acc(\kappa)$ and $\Gamma\cap T=\emptyset$.
  \item For all $\beta\in \Delta$ and $\alpha \in \acc(D_\beta)$,
    we have:
    \begin{itemize}
    \item $\alpha\in\Omega$,
    \item $D_\alpha = D_\beta \cap \alpha$, and
    \item $j(\alpha)=j(\beta)$.
    \end{itemize}
  \item    For every $\beta \in \Gamma\setminus\acc^+(\Gamma)$, we have $\otp(D_\beta\setminus\sup(\Gamma\cap\beta))=\omega$.
  \item For every $i<\theta$, $\Gamma_i$ is stationary.
  \item $\min(\Omega)\in\Gamma_0$.
    \end{enumerate}
  \end{claim}

  \begin{cproof} (1) This follows directly from the definition of $\Gamma$.

  (2)   Let $\beta\in \Delta$ and $\alpha \in \acc(D_\beta)$ be arbitrary. There two cases to consider:
  \begin{itemize}
  \item[$\br$] If $\otp(C_\beta)\le\zeta$, then $\alpha\in\acc(D_\beta)=\acc(C_\beta)$,
  so $C_\alpha=C_\beta\cap\alpha$ and $\otp(C_\alpha)<\zeta$.
  Since $\beta \in \Delta$ and $\alpha \in \acc(C_\beta)$, the definition
  of $\Omega$ implies that $\alpha \in \Omega \subseteq \Delta$.
  We therefore have $D_\alpha = C_\alpha = C_\beta \cap \alpha =
  D_\beta \cap \alpha$. In particular, $\min(D_\alpha) = \min(D_\beta)$, so
  $j(\alpha)=j(\beta)$.
	\item[$\br$] If $\otp(C_\beta)>\zeta$, then $\alpha\in\acc(D_\beta)=\acc(C_\beta\setminus C_\beta(\zeta))$,
  so $C_\alpha=C_\beta\cap\alpha$ and $\otp(C_\alpha)>\zeta$. As
  in the previous case, we have $\alpha \in \Omega \subseteq \Delta$, and
  therefore $D_\alpha = C_\alpha\setminus C_\alpha(\zeta)=(C_\beta\setminus
  C_\beta(\zeta))\cap\alpha=D_\beta\cap\alpha$. In particular,
  $j(\alpha)=j(\beta)$.
  \end{itemize}

	(3) Let $\beta \in \Gamma$. Note the following:
	\begin{itemize}
	\item[$\br$]If $\omega < \cf(\beta)\le\theta$, then  $E^\beta_\omega\s \acc(\kappa)\cap E^\kappa_{<\theta}\s\Gamma$,
	and hence $\beta\in\acc^+(\Gamma)$.

	\item[$\br$] If $\cf(\beta)>\theta$, then $\acc(C_\beta)\s\Omega\s \Gamma$,
	so again $\beta\in\acc^+(\Gamma)$.
	\end{itemize}

  Therefore, if $\beta\in \Gamma \setminus\acc^+(\Gamma)$, then $\cf(\beta)=\omega$.
	So, if $\otp(D_\beta)>\omega$,
	then $\beta\in\Delta$, and then $\acc(C_\beta)\s\Omega\s\Gamma$.
	As we assume that $\beta\notin\acc^+(\Gamma)$, it in particular follows that $\sup(\acc(C_\beta))<\beta$,
	and $\otp(C_\beta\setminus\sup(\Gamma\cap\beta))=\omega$.
	As $D_\beta$ is a final segment of $C_\beta$, it also follows that
	$\otp(D_\beta\setminus\sup(\Gamma\cap\beta))=\omega$.

	(4) This follows directly from the choice of $\langle \tau_i\mid i<\theta\rangle$.

	(5) For every $i<\theta$, we have $\tau_i\ge\tau_0>\min(\Omega)$,
  so it is impossible for $j(\min(\Omega))$ to be greater than $0$.
  \end{cproof}

  We now construct a $\inds(\kappa, \theta)$-sequence
  $\langle C_{\alpha, i} \mid \alpha \in \Gamma, ~ i(\alpha) \leq i <
  \theta \rangle$.
  We will maintain the requirement that, for all $\alpha \in \Omega$,
  we have $i(\alpha) = j(\alpha)$ and $\acc(D_\alpha) \subseteq \acc(C_{\alpha, i(\alpha)})$.

  As a base case, if $\beta = \min(\Gamma)$, then by Claim~\ref{cf_claim}(3),
  we have $\otp(D_\beta) =  \omega$. Set $i(\beta):=0$,
  and let $C_{\beta, i} := D_\beta$ for all $i<\theta$.
  Note that if $\beta\in\Omega$, then
  by Claim~\ref{cf_claim}(5), indeed $i(\beta)=j(\beta)$.

  Suppose now that $\beta \in \Gamma \setminus \{\min(\Gamma)\}$
  and we have defined $\langle C_{\alpha, i} \mid \alpha \in \Gamma
  \cap \beta, ~ i(\alpha) \leq i < \theta \rangle$ satisfying all
  relevant instances of Clauses (2) and (3) of Definition~\ref{inds}
  as well as our recursive requirement.
  The construction breaks into a number of different cases based on the
  identity of $\beta$. In all cases, the verification that our sequence
  satisfies our recursive requirement and Clauses (2) and (3) of
  Definition~\ref{inds} at $\beta$ is routine and therefore largely left
  to the reader.

   \begin{description}
  \item[Case 1] $\beta \in \Gamma \setminus \Omega$.
  In particular,
  we have $\cf(\beta) < \theta$. We now split into subcases depending on
  the behavior of $\Gamma \cap \beta$.

  \item[Case 1a] \textbf{$\sup(\Gamma \cap \beta) < \beta$ and $\max(\Gamma
  \cap \beta)$ exists.}
  Let $\alpha := \max(\Gamma \cap \beta)$. By Claim~\ref{cf_claim}(3), we know
  that $D^-_\beta := D_\beta \setminus \alpha$ has order type $\omega$.
  Let $i(\beta) := i(\alpha)$ and, for all $i \in [i(\beta), \theta)$, let
  $C_{\beta, i} = C_{\alpha, i} \cup \{\alpha\} \cup D^-_\beta$.

  \item[Case 1b] \textbf{$\sup(\Gamma \cap \beta) < \beta$ but $\max(\Gamma
  \cap \beta)$ does not exist.}
  Let $\alpha := \sup(\Gamma \cap \beta)$. As in Case (1a), we know that
  $D^-_\beta := D_\beta \setminus \alpha$ has order type
  $\omega$.
  Since $\alpha\notin\Gamma$, it follows that $\alpha\in E^\kappa_\theta$.
  Since $\cf(\beta) < \theta$, we know that $\theta > \omega$, and hence
  $\acc(D_\alpha)$ is unbounded in $\alpha$. Let
  $\langle \alpha_i \mid i < \theta
  \rangle$ be a strictly increasing sequence of elements of
  $\acc(D_\alpha)$ converging to $\alpha$.
  As $\alpha\in E^\kappa_\theta\s\Delta$,
  it follows from Claim~\ref{cf_claim}(2) that, for all $i<\theta$,
  $\alpha_i\in \Omega\s\Gamma$ and $j(\alpha_i)=j(\alpha)$.
  So, by the fact that our
  construction thus far satisfies all of our requirements, we know that,
  for all $i < \theta$, we have $i(\alpha_i) = j(\alpha)$ and
$\acc(D_\alpha) \cap \alpha_i = \acc(D_{\alpha_i})\s \acc(C_{\alpha_i, j(\alpha)})$.
  Set $i(\beta):= j(\alpha)$ and $C_{\beta, i} := C_{\alpha_i, i} \cup
  \{\alpha_i\} \cup D^-_\beta$ for all $i\in[i(\beta),\theta)$.

  \item[Case 1c] $\sup(\Gamma \cap \beta) = \beta$.
  In this case, let
  $\langle \alpha_\eta \mid \eta < \cf(\beta) \rangle$ be an increasing
  sequence of elements of $\Gamma$ converging to $\beta$. Let $i(\beta)$ be
  the least ordinal $i < \theta$ such that, for all $\eta < \xi < \cf(\beta)$,
  we have $i(\alpha_\eta), i(\alpha_\xi) \leq i$ and $\alpha_\eta \in
  \acc(C_{\alpha_\xi, i})$. Such an ordinal exists because $\cf(\beta)
  < \theta$ and our sequence so far satisfies Clauses (2) and (3) of
  Definition~\ref{inds}. Then, for all $i \in [i(\beta), \theta)$, let
  $C_{\beta, i} = \bigcup_{\eta < \cf(\beta)} C_{\alpha_\eta, i}$.

  \item[Case 2] $\beta \in \Omega$. In this case, we are required to
  set $i(\beta) := j(\beta)$. We again split into subcases depending on the
  behavior of $\Gamma \cap \beta$ and $\acc(D_\beta)$.

  \item[Case 2a] \textbf{$\sup(\Gamma \cap \beta) < \beta$ and $\max(\Gamma
  \cap \beta)$ exists.} Let $\alpha := \max(\Gamma \cap \beta)$. As in
  Case (1a), we know that $D^-_\beta := D_\beta \setminus \alpha$ has order
  type $\omega$.

  $\br$ If $\alpha\in\acc(D_\beta)$, then $i(\alpha)=j(\alpha)=j(\beta)=i(\beta)$,
  so we let $C_{\beta, i} := C_{\alpha, i} \cup \{\alpha\}  \cup D^-_\beta$
  for all $i\in[i(\beta),\theta)$.
  Note that
  \[
    \acc(D_\beta) = \acc(D_{\alpha}) \cup \{\alpha\} \subseteq
    \acc(C_{\alpha, i(\alpha)}) \cup \{\alpha\} \subseteq \acc(C_{\beta, i(\beta)}),
  \]
  so we have satisfied our recursive hypothesis.

  $\br$   If $\acc(D_\beta) \neq \emptyset$ but $\alpha>\max(\acc(D_\beta))$,
  then let $\alpha^* := \max(\acc(D_\beta))$
  and let $i^*\in[i(\alpha),\theta)$ be least such that
  $\alpha^* \in \acc(C_{\alpha, i^*})$.
  Note that $i(\beta)=i(\alpha^*)\le i^*$.
  Let $C_{\beta, i} := C_{\alpha^*, i} \cup \{\alpha^*\}  \cup D^-_\beta$
  for all $i\in[i(\beta),i^*)$,
  and let $C_{\beta, i} := C_{\alpha, i} \cup \{\alpha\} \cup D^-_\beta$
  for all $i\in[i^*,\theta)$.

  Note that we have satisfied our recursive hypothesis.

  $\br$ If $\acc(D_\beta)=\emptyset$, then let $C_{\beta,i}:=D_\beta^-$ for all $i\in[i(\beta),i(\alpha))$,
  and let $C_{\beta, i} := C_{\alpha, i} \cup \{\alpha\} \cup D^-_\beta$
  for all $i\in[\max\{i(\beta),i(\alpha)\},\theta)$.

  \item[Case 2b] \textbf{$\sup(\Gamma \cap \beta) < \beta$ but $\max(\Gamma
  \cap \beta)$ does not exist.} Let $\alpha := \sup(\Gamma \cap \beta)$.
  Once again, it follows that $D^-_\beta := D_\beta \setminus \alpha$ has order
  type $\omega$. As in Case (1b), we have $\alpha \in E^\kappa_\theta$,
  and, for all $\bar\alpha\in\acc(D_\alpha)$, we have
  $\bar \alpha\in \Omega$ and $j(\bar \alpha)=j(\alpha)$.
  Let $\alpha^*:=\max(\acc(D_\beta))$ if $\acc(D_\beta) \neq \emptyset$, and
  $\alpha^* := \min(\Gamma)$ otherwise. In either case, note
  that $\alpha^* < \alpha$ and $i(\alpha^*) \le i(\beta)$.

  Our construction now depends on whether or not $\theta = \omega$.

  $\br$ If $\theta = \omega$, then let $\langle \alpha_n \mid n < \omega \rangle$
  be a strictly increasing sequence of elements of $\Gamma \cap \alpha$
  converging to $\alpha$ such that $\alpha_0 = \alpha^*$.
  Let $\langle i_n \mid n < \omega \rangle$
  be a strictly increasing sequence of natural numbers such that,
  for all $n < \omega$, $\max\{i(\alpha_m)\mid m\le n\}\le i_n$
   and $\{\alpha_m \mid m<n\}\s \acc(C_{\alpha_n,  i_n})$.

  Finally, for all $i \in [i(\beta), \omega)$,
  let $n < \omega$ be such that $i_n \leq i < i_{n+1}$, and set
  $C_{\beta, i} := C_{\alpha_n, i} \cup \{\alpha_n\} \cup D^-_\beta$.

  $\br$ If $\theta > \omega$, then let $\langle \alpha_i \mid i < \theta
  \rangle$ be a strictly increasing sequence of elements of $\acc(D_\alpha)$
  converging to $\alpha$ such that $\alpha_0 > \alpha^*$.
  As noted earlier, we have $i(\alpha_i) = j(\alpha)$.
  Let $i^*\in[j(\alpha),\theta)$ be least such that $\alpha^* \in \acc(C_{\alpha_0, i^*})$.
  In particular, for all
  $i \in [i^*, \theta)$, we have $\alpha^* \in \acc(C_{\alpha_i, i})$,
  so that $\acc(D_\beta)\s \acc(C_{\alpha_i, i})$.

  Now, let
  $C_{\beta, i} := C_{\alpha^*, i} \cup \{\alpha^*\} \cup D^-_\beta$
  for all $i\in[i(\beta),i^*)$,
  and let $C_{\beta, i} := C_{\alpha_i, i} \cup \{\alpha_i\} \cup D^-_\beta$
  for all $i\in[i^*,\theta)$.

  \item[Case 2c] \textbf{$\sup(\Gamma \cap \beta) = \beta$ but
  $\sup(\acc(D_\beta)) < \beta$.} In this case,
  $\otp(D_\beta\setminus\sup(\acc(D_\beta)))=\omega$.
  Let $\alpha^*:=\max(\acc(D_\beta))$ if $\acc(D_\beta) \neq \emptyset$, and
  $\alpha^* := \min(\Gamma)$ otherwise. In either case,
  $\alpha^* < \alpha$ and $i(\alpha^*) \le i(\beta)$.
  Let $\langle \alpha_n \mid n < \omega \rangle$ be an increasing sequence of elements of $\Gamma$
  converging to $\beta$ with $\alpha_0 := \alpha^*$.
  Fix an increasing sequence $\langle i_n \mid n < \omega \rangle$ of ordinals below
  $\theta$ such that
  for all $n < \omega$, $\max\{i(\alpha_m)\mid m\le n\}\le i_n$
  and $\{\alpha_m \mid m<n\}\s \acc(C_{\alpha_n,  i_n})$.

  Finally, fix $i\in[i(\beta),\theta)$. If there is
  $k < \omega$ such that $i_k \leq i < i_{k+1}$, then set
  $C_{\beta, i} := C_{\alpha_k, i} \cup \{\alpha_n \mid k \leq n < \omega\}$
  for this unique $k$. Otherwise, set
  $C_{\beta, i} := \bigcup_{n < \omega} C_{\alpha_n, i}$.

  \item[Case 2d] \textbf{$\sup(\acc(D_\beta)) = \beta$.} Note that, for
  all $\alpha \in \acc(D_\beta)$, we have $\alpha \in \Omega$,
  $i(\alpha) = i(\beta)$ and $\acc(D_\beta) \cap \alpha =
  \acc(D_\alpha) \subseteq C_{\alpha, i(\alpha)}$. Therefore, for all $i\in[i(\beta),\theta)$,
  we can simply set $C_{\beta, i} := \bigcup_{\alpha \in \acc(D_\beta)}
  C_{\alpha, i}$.
  \end{description}

  Our construction has yielded a matrix satisfying clauses
  (1)--(3) of Definition~\ref{inds} and such that $\Gamma \cap T =
  \emptyset$. It remains to verify clause (4) of Definition~\ref{inds}.
    Towards a contradiction, suppose that $D$ is a club in $\kappa$
    satisfying that, for every $\alpha\in\acc(D) \cap \Gamma$,
      there exists $i<\theta$ such that $D\cap\alpha\neq C_{\alpha,i}$.
      Fix $i<\theta$ for which $G:=\{\gamma\in \Gamma\mid D\cap\gamma=C_{\gamma,i}\}$
      is stationary.
      Recalling Clause~\ref{cf_claim}(4),
    let us now fix $\beta\in\Gamma_{i+1}\cap\acc(D)$. Pick $\gamma\in G$ above $\beta$.
    Then $\beta\in\acc(D\cap\gamma)=\acc(C_{\gamma,i})$, so $i(\beta)\le i$,
    contradicting the fact that $i(\beta)=j(\beta)=i+1$.
\end{proof}

\begin{cor} $(\aleph_{\omega+1},\aleph_\omega)\twoheadrightarrow(\aleph_1,\aleph_0)$ is compatible with $\US(\aleph_{\omega+1},\aleph_{\omega+1},\allowbreak\theta,\aleph_\omega)$ holding for every infinite cardinal $\theta<\aleph_\omega$.
\end{cor}
\begin{proof} Starting with a ground model in which $(\aleph_{\omega+1},\aleph_\omega)\twoheadrightarrow(\aleph_1,\aleph_0)$ holds,
one can add a $\square(\aleph_{\omega+1})$-sequence via an $\aleph_{\omega+1}$-strategically closed forcing,
hence preserving the principle $(\aleph_{\omega+1},\aleph_\omega)\twoheadrightarrow(\aleph_1,\aleph_0)$.
By Theorems \ref{thm67} and \ref{revised_lh_lucke}, in the extension,
$\US(\aleph_{\omega+1},\aleph_{\omega+1},\allowbreak\theta,\aleph_\omega)$ holds for every infinite cardinal $\theta<\aleph_\omega$.
\end{proof}
By Theorem~\ref{thm310}, for any $\theta\in\reg(\kappa)$,
there exists a $\theta^+$-directed closed, $\kappa$-strategically closed forcing notion
that introduces a somewhere-closed witness to $\US(\kappa, \kappa,\theta, \theta)$.
In order to force a fully closed witness, it seems that we must decrease the degree of
closure of the poset by one cardinal. The next theorem forms Clause~(2) of Theorem~C.
\begin{thm}\label{add_inds}   Suppose that $\theta \in\reg(\kappa)$. Then there exists a
  $\theta$-directed closed, $\kappa$-strategically closed forcing notion that introduces a closed witness $c:[\kappa]^2\rightarrow\theta$ to $\US(\kappa, \kappa,\theta, \theta)$
  for which $\partial(c)$ is stationary.
\end{thm}
\begin{proof}
  By \cite[\S 7]{narrow_systems}, there is a $\theta$-directed closed,
  $\kappa$-strategically closed forcing notion that adds an
  $\square^{\ind}(\kappa, \theta)$-sequence and therefore, by Theorem~\ref{thm67},
  it adds a closed witness $c$ to $\US(\kappa, \kappa,\theta, \theta)$. However, unless $\theta=\omega$,
  it is unclear whether $\partial(c)$ can be made to be stationary in this case.
  Thus, instead, we appeal to the forcing $\mathbb P(\kappa,\theta)$
  from \cite[\S3.3]{paper28}. This is a $\theta$-directed closed,
  $\kappa$-strategically closed forcing notion
  that introduces a witness $\vec C=\langle C_\alpha\mid\alpha<\kappa\rangle$ to $\p^-(\kappa,2,\allowbreak{\sq_\theta},1,\{\kappa\},2,\sigma)$
  for every $\sigma<\kappa$.
  Now, utilizing the instance $\sigma:=\theta$,
  the proof of \cite[Theorem~4.1]{lh_trees_squares_reflection}
  makes it clear that the
  set $\{\alpha\in E^\kappa_\theta\mid \exists\eta<\alpha[\otp(C_\alpha\setminus\eta)=\theta]\}$ is stationary.
  Fix some $\eta<\kappa$ for which $T:=\{\alpha\in E^\kappa_\theta\mid \otp(C_\alpha\setminus\eta)=\theta\}$ is stationary.
  Note that, by the $\sq_\theta$-coherence of $\vec C$,
  the set $\Gamma:=\{\alpha\in\acc(\kappa)\mid \forall\bar\alpha\in\acc(C_\alpha)[C_{\bar\alpha}\sq C_\alpha]\}$
  covers $E^\kappa_{\ge\theta}$.
  Now, define a $C$-sequence $\vec{D} = \langle D_\alpha \mid \alpha < \kappa \rangle$ as follows:
  \begin{itemize}
  \item Let $D_0:=\emptyset$.
  \item For every $\alpha<\kappa$, let $D_{\alpha+1}:=\{\alpha\}$.
  \item For every $\alpha\in\acc(\kappa)\setminus\Gamma$, let $D_\alpha$ be some club in $\alpha$ of order-type $\cf(\alpha)$ such that $\nacc(D_\alpha)\s\nacc(\alpha)$.
  \item For every $\alpha\in \Gamma$ such that $\otp(C_\alpha\setminus\eta)\le\theta$, let $D_\alpha:=C_\alpha$.
  \item For any other $\alpha$, let $D_\alpha:=\{\beta\in C_\alpha\mid \otp((C_\alpha\setminus\eta)\cap\beta)>\theta\}$.
  \end{itemize}

  Note that, for every $\alpha\in\kappa\setminus\Gamma$, $\otp(D_\alpha)=\cf(\alpha)<\theta$, so $\acc(D_\alpha)\cap T=\emptyset$.

  \begin{claim}   $\vec D$ is a $\square(\kappa,\sq_\theta)$-sequence that avoids $T$.
  \end{claim}
  \begin{cproof} As any $\sq_\theta$-coherent $C$-sequence that avoids a stationary set is a $\square(\kappa,\sq_\theta)$-sequence,
  we shall focus on verifying that $\vec D$ is $\sq_\theta$-coherent and avoids $T$.
  By the definition of $\vec D$, it  suffices to verify that for all $\alpha\in\Gamma$ and $\bar\alpha\in\acc(D_\alpha)$, $D_{\bar\alpha}\sq D_\alpha$
  and $\bar\alpha\notin T$.
  Now, given such a pair $\bar\alpha<\alpha$, since $\alpha\in\Gamma$, we infer that
  $C_{\bar\alpha}\sq C_\alpha$ and $\bar\alpha\in\Gamma$. We shall verify that $D_{\bar\alpha}\sq D_\alpha$
  and that $\bar\alpha\notin T$.

  $\br$ If $\otp(C_\alpha\setminus\eta)\le\theta$, then $\otp(C_{\bar\alpha}\setminus\eta)<\theta$, so
  $D_{\bar\alpha}=C_{\bar\alpha}\sq C_\alpha=D_\alpha$
  and $\bar\alpha\notin T$.

  $\br$ If $\otp(C_\alpha\setminus\eta)>\theta$, then $D_\alpha=\{\beta\in C_\alpha\mid \otp((C_\alpha\setminus\eta)\cap\beta)>\theta\}$,
  and so, from $\bar\alpha\in\acc(D_\alpha)$ and $C_{\bar\alpha}=C_\alpha\cap\bar\alpha$, it follows that $\otp(C_{\bar\alpha}\setminus\eta)>\theta$,
  so again $D_{\bar\alpha}\sq D_\alpha$ and $\bar\alpha\notin T$.
  \end{cproof}
  Now, by Theorem~\ref{revised_lh_lucke}(2) and Lemma~\ref{uPLUSsub}(3),
  there is a closed witness $c$ to $\US(\kappa, \kappa,\theta, \theta)$ such that $T\s\partial(c)$.
\end{proof}

It follows that there can be dramatic failures of monotonicity
in the third coordinate of $\US(\ldots)$.

\begin{cor} Suppose that $\theta$ is an indestructible supercompact cardinal below $\kappa$.
Then there is a forcing extension in which
all cofinalities $\leq \kappa$ are preserved and all of the following hold:
\begin{itemize}
\item $\US(\kappa,\kappa,\theta,\theta)$ holds and is witnessed by a closed coloring;
\item for all $\theta'<\theta$, no witness to $\U(\kappa,2,\theta',2)$ is subadditive of the second kind;
\item for all $\theta'<\theta$ such that $\kappa$ is not the successor of a singular cardinal of cofinality $\theta$', no witness to $\U(\kappa,2,\theta',2)$ is weakly subadditive of the first kind.
\end{itemize}
\end{cor}
\begin{proof}
  By Theorem~\ref{add_inds}, there is a $\theta$-directed closed,
  $\kappa$-strategically closed forcing notion $\mathbb{P}$ that adds
  a closed witness to $\US(\kappa, \kappa, \theta, \theta)$.
  Since $\mathbb{P}$ is $\theta$-directed closed, $\theta$ remains supercompact,
  so, we may appeal to Corollary~\ref{cor64}.
\end{proof}

\section{Successors of singular cardinals} \label{singular_sec}

\begin{defn}
  Suppose that $\vec\lambda=\langle \lambda_j\mid j<\nu\rangle$ is a sequence of infinite
  cardinals, each greater than $\nu$. The principle $\sd(\vec\lambda)$
  asserts the existence of a sequence $\langle X_\alpha\mid\alpha<\sup(\vec\lambda)\rangle$,
  and, for all $j<\nu$, a sequence $\vec {C}^j=\langle C_\alpha^j\mid \alpha<\lambda_j^+\rangle$
  such that
  \begin{itemize}
    \item for all $j<\nu$ and $\alpha\in\acc(\lambda_j^+)$, $C_\alpha^j$ is a
      club in $\alpha$, and for all $\bar\alpha\in\acc(C^j_\alpha)$,
      \begin{itemize}
        \item $C^j_{\bar\alpha}=C^j_\alpha\cap\bar\alpha$;
        \item $X_{\bar\alpha}=X_\alpha\cap\bar\alpha$, provided that $\alpha>\lambda_j$.
      \end{itemize}
    \item for all  $X\s\sup(\vec\lambda)$ and $p\in H_{\Upsilon}$, there exists $N\prec H_\Upsilon$
      such that
      \begin{itemize}
        \item $p \in N$;
        \item $|N| < \sup(\vec{\lambda})$;
        \item $N$ is internally approachable of length $\nu^+$;
        \item for all $j < \nu$, we have $X \cap \sup(N \cap \lambda_j^+) =
          X_{\sup(N \cap \lambda_j^+)}$.
      \end{itemize}
  \end{itemize}
\end{defn}

\begin{remark}\label{remark7}
  Note that, if $\sd(\vec{\lambda})$ holds for a sequence $\vec{\lambda} = \langle \lambda_j \mid j < \nu \rangle$,
  then $2^\mu\le\lambda$ for all $\mu<\lambda$,
  and $\square(\lambda_j^+)$ holds for each $j < \nu$, as witnessed by the sequence
  $\vec{C^j}$ from the above definition.
\end{remark}

We will soon show that $\sd(\vec{\lambda})$ entails certain instances of $\US(\lambda^+,
\ldots)$. We first need the following lemma.

\begin{lemma}\label{lem32}
  Suppose that $\langle C_\alpha \mid \alpha < \kappa \rangle$ is a $\square(\kappa)$-sequence,
  $\theta \in \reg(\kappa)$, and $h:\kappa \rightarrow \theta$ is any function such that
  \begin{enumerate}
    \item for all $i < \theta$, $h^{-1}\{i\}$ is stationary;
    \item for all $\alpha < \kappa$ and all $\bar\alpha \in \acc(C_\alpha)$, we have
      $h(\bar\alpha) = h(\alpha)$.
  \end{enumerate}
  Then there is a closed, subadditive witness $c$ to $\U(\kappa, \kappa, \theta,
  \sup(\reg(\kappa)))$ such that, for all $(\alpha, \beta) \in [\kappa]^2$
  with $\alpha \in \acc(\kappa)$, we have $c(\alpha, \beta) \geq h(\alpha)$.
\end{lemma}

\begin{proof}
  By the proof of \cite[Theorem~3.4]{lh_lucke}, the hypotheses of the lemma entail
  the existence of a $\square^{\ind}(\kappa, \theta)$-sequence
  $\langle C_{\alpha, i} \mid \alpha \in \acc(\kappa), ~ i(\alpha) \leq i < \theta \rangle$
  such that $i(\alpha) = h(\alpha)$ for all $\alpha \in \acc(\kappa)$. The
  construction in the proof of the implication $(1) \implies (2)$ of Theorem~\ref{thm67}
  then yields a closed subadditive function $c:[\kappa]^2 \rightarrow \theta$
  witnessing $\U(\kappa, 2, \theta, 2)$ such that, for all $(\alpha, \beta) \in [\kappa]^2$,
  if $\alpha \in \acc(\kappa)$, then $c(\alpha, \beta) \geq i(\alpha) = h(\alpha)$.
  By Fact~\ref{uPLUSsub}, $c$ is moreover a witness to $\U(\kappa, \kappa, \theta,
  \sup(\reg(\kappa)))$.
\end{proof}

Our next result is Theorem~B. It is very much in the spirit of \cite[Corollary~3.10]{MR2078366};
both results take instances of incompactness below a singular cardinal $\lambda$
and use them together with a scale of length $\lambda^+$ to produce an instance
of incompactness at $\lambda^+$.

\begin{thm}\label{sdlambda}
  Suppose that $\lambda$ is a singular cardinal, $\theta \in \reg(\lambda) \setminus (\cf(\lambda) + 1)$,
  and $\vec \lambda=\langle \lambda_j \mid j < \cf(\lambda) \rangle$ is an increasing sequence
  of cardinals, converging to $\lambda$, such that
  \begin{itemize}
    \item $\sd(\vec{\lambda})$ holds;
    \item $\prod_{j < \cf(\lambda)} \lambda_j^+$ carries a scale $\vec{f}$ of length
      $\lambda^+$.
  \end{itemize}
  Then there is a $\Sigma$-closed, subadditive witness to $\U(\lambda^+, \lambda^+,
  \theta, \lambda)$, where $\Sigma \s E^{\lambda^+}_{\neq \cf(\lambda)}$ denotes the set of good points for $\vec{f}$.
\end{thm}

\begin{proof}
  Without loss of generality, assume that $\lambda_0>\theta$.
  Fix sequences $\langle X_\alpha \mid \alpha < \lambda \rangle$ and $\vec{C}^j
  = \langle C^j_\alpha \mid \alpha < \lambda_j^+ \rangle$ for $j < \cf(\lambda)$
  witnessing $\sd(\vec{\lambda})$.

  Let $j<\cf(\lambda)$ be arbitrary.
  It is clear from the definition of $\sd(\vec{\lambda})$ that we may assume
  that  $\min(C^j_\alpha) \geq \lambda_j$ whenever $\alpha \in(\lambda_j, \lambda_j^+)$.
  Now, define a function $h_j:\lambda_j^+ \rightarrow \theta$ by letting, for all $\alpha<\lambda_j^+$,
  $$h_j(\alpha):=\begin{cases}X_\alpha &\text{if }\alpha>\lambda_j\ \&\ X_\alpha\in\theta;\\
  0 &\text{otherwise}.\end{cases}$$
  Recalling the definition of $\sd(\vec{\lambda})$ and Remark~\ref{remark7}, it is evident that
  $\vec{C}^j$ and $h_j$ satisfy the hypotheses of Lemma~\ref{lem32},
  so we may fix a closed, subadditive witness $c_j:[\lambda_j^+]^2 \rightarrow \theta$ to
  $\U(\lambda_j^+, \lambda_j^+, \theta, \lambda_j)$ such that $c_j(\alpha, \beta)
  \geq h_j(\alpha)$ for all $(\alpha, \beta) \in [\lambda_j^+]^2$ with
  $\alpha \in \acc(\lambda_j^+)$.

  Next, let $\vec{f} = \langle f_\beta \mid \beta < \lambda^+ \rangle$ be a continuous scale in
  $\prod_{j < \cf(\lambda)} \lambda_j^+$, and let $\Sigma \s E^{\lambda^+}_{\neq \cf(\lambda)}$
  denote the set of good points for $\vec{f}$. Define a coloring $c:[\lambda^+]^2
  \rightarrow \theta$ by setting, for all $\beta<\gamma<\lambda^+$,
  $$
    c(\beta, \gamma) := \limsup_{j \rightarrow \cf(\lambda)} c_j(f_\beta(j), f_\gamma(j)).
  $$
  Note that, for all $(\beta, \gamma) \in [\lambda^+]^2$ and all sufficiently large
  $j < \cf(\lambda)$, we have $f_\beta(j) < f_\gamma(j)$, so the above expression
  is well-defined. We claim that $c$ is a $\Sigma$-closed,
  subadditive witness to
  $\U(\lambda^+, \lambda^+, \theta, \lambda)$. Let us verify each of these
  requirements in turn.

  \begin{claim}
    $c$ is $\Sigma$-closed.
  \end{claim}

  \begin{cproof}
    Suppose that $\gamma < \lambda^+$, $i < \theta$, and $B \s D^c_{\leq i}(\gamma)$,
    with $\beta := \sup(B)$ in $(\gamma \cap \Sigma) \setminus B$.
    We will show that $\beta \in D^c_{\leq i}(\gamma)$.

    Since $\beta \in \Sigma$, we can assume, by thinning out $B$ if necessary,
    that there is $j_0 < \cf(\lambda)$ such that, for all $(\alpha, \alpha') \in
    [B]^2$, we have $f_\alpha <_{j_0} f_{\alpha'}$. Since $\vec{f}$ is continuous,
    there is $j_1 \in [j_0, \cf(\lambda))$ such that, for all
    $j \in [j_1, \cf(\lambda))$, we have $f_\beta(j) = \sup\{f_\alpha(j) \mid
    \alpha \in B\}$. Finally, since $B \s D^c_{\leq i}(\gamma)$, we can assume,
    by thinning out $B$ again if necessary, that there is $j_2$ with
    $j_1 \leq j_2 < \cf(\lambda)$ such that, for all $j \in [j_2, \cf(\lambda))$,
    \begin{itemize}
      \item $f_\beta(j) < f_\gamma(j)$;
      \item for all $\alpha \in B$, $c_j(f_\alpha(j), f_\gamma(j)) \leq i$.
    \end{itemize}
    Since each $c_j$ is closed, it follows that, for all $j \in [j_2,\cf(\lambda))$,
    we have $\sup\{f_\alpha(j) \mid \alpha \in B\} = f_\beta(j)$ and
    $c_j(f_\beta(j), f_\gamma(j)) \leq i$, and hence $c(\beta, \gamma) \leq i$.
  \end{cproof}

  \begin{claim}
    $c$ is subadditive.
  \end{claim}

  \begin{cproof}
    Let $\alpha < \beta < \gamma < \lambda^+$ be arbitrary. For all sufficiently
    large $j < \cf(\lambda)$, we have $f_\alpha(j) < f_\beta(j) < f_\gamma(j)$
    and hence, since each $c_j$ is subadditive, we have
    \begin{itemize}
      \item $c_j(f_\alpha(j), f_\gamma(j)) \leq \max\{c_j(f_\alpha(j), f_\beta(j)),
        c_j(f_\beta(j), f_\gamma(j))\}$; and
      \item $c_j(f_\alpha(j), f_\beta(j)) \leq \max\{c_j(f_\alpha(j), f_\gamma(j)),
        c_j(f_\beta(j), f_\gamma(j))\}$.
    \end{itemize}
    It follows immediately from the definition of $c$ that $c(\alpha, \gamma)
    \leq \max\{c(\alpha, \beta), c(\beta, \gamma)\}$ and $c(\alpha, \beta)
    \leq \max\{c(\alpha, \gamma), c(\beta, \gamma)\}$.
  \end{cproof}

  \begin{claim}
    $c$ witnesses $\U(\lambda^+, \lambda^+, \theta, \lambda)$.
  \end{claim}

  \begin{proof} \renewcommand{\qedsymbol}{\ensuremath{\boxtimes \ \square}}
    By a theorem of Shelah, $\Sigma \cap E^{\lambda^+}_{\ge\chi}$  is stationary for all $\chi<\lambda$.
    So, since $c$ is $\Sigma$-closed and subadditive,
    Lemma~\ref{uPLUSsub}(3) implies that it suffices to verify that $c$ witnesses
    $\U(\lambda^+, 2, \theta, 2)$. To this end, fix $A \in [\lambda^+]^{\lambda^+}$
    and a color $i < \theta$. We will find $(\alpha, \gamma) \in [A]^2$ such that
    $c(\alpha, \gamma) > i$.

    Apply $\sd(\vec\lambda)$ with $X := i+1$ and $p:=\{\vec\lambda,\vec f,A\}$ to find $N \prec H_\Upsilon$
    such that
    \begin{itemize}
      \item $p\in N$;
      \item $|N| < \lambda$;
      \item $N$ is internally approachable of length $\cf(\lambda)^+$;
      \item for all $j < \cf(\lambda)$, we have $X_{\sup(N \cap \lambda_j^+)} = i+1$.
    \end{itemize}

    As $|N|<\lambda$ and $\vec\lambda\in N$, there exists a function $\chi_N\in\prod_{j<\cf(\lambda)}\acc(\lambda_j^+)$
    such that $\chi_N(j)=\sup(N\cap\lambda_j^+)$ for all sufficiently large $j<\cf(\lambda)$.
    Let $\beta:=\sup(N\cap\lambda^+)$.
    As $N$ is internally approachable of length $\cf(\lambda)^+$, we know that $\beta\in \Sigma \cap E^{\lambda^+}_{\cf(\lambda)^+}$.
    Since $\vec{f} \in N$, $\vec{f}$ is continuous,
    and $N$ is internally approachable, we know that $f_\beta(j) = \chi_N(j)$ for all sufficiently large $j<\cf(\lambda)$.
    Thus, for all sufficiently large $j<\cf(\lambda)$ and all $\eta \in (f_\beta(j),     \lambda_j^+)$,
    \[
      c_j(f_\beta(j), \eta) \geq h_j(f_\beta(j))= h_j(\chi_N(j))=X_{\sup(N\cap\lambda_j^+)}=i+1.
    \]
    Consequently, $c(\beta, \gamma)>i$ for all $\gamma>\beta$.

    Fix $\gamma \in A \setminus (\beta + 1)$.
    Since $\beta \in \Sigma$
    and $c$ is $\Sigma$-closed, there is $\epsilon < \beta$ such that
    $c(\alpha, \gamma)>i$ for all $\alpha \in (\epsilon, \beta)$. As $A\in N$, the
    elementarity of $N$ entails that $\sup(A \cap \beta) = \beta$, so we may
    fix $\alpha \in A \cap (\epsilon, \beta)$. Then $c(\alpha, \gamma)>i$,
    as desired.
  \end{proof}
  \let\qed\relax
\end{proof}

We now connect the above finding with the
notion of the \emph{$C$-sequence spectrum}, introduced in Part~II of this series.

\begin{defn}[\cite{paper35}]\label{cspecdef}
  \begin{enumerate}
    \item For every $C$-sequence $\vec{C} = \langle C_\beta \mid \beta < \kappa \rangle$,
      $\chi(\vec{C})$ is the least cardinal $\chi \leq \kappa$ such that there exist
      $\Delta \in [\kappa]^\kappa$ and $b:\kappa\rightarrow[\kappa]^\chi$ with
      $\Delta\cap\alpha\s\bigcup_{\beta\in b(\alpha)}C_\beta$
      for every $\alpha<\kappa$.
    \item $\cspec(\kappa) := \{\chi(\vec{C}) \mid \vec{C}$ is a $C$-sequence over $\kappa\} \setminus \omega$.
  \end{enumerate}

\end{defn}

The next result yields the ``In particular'' part of Theorem~B.

\begin{cor}
  Suppose that $\lambda$ is a singular cardinal, $\vec \lambda$ is an increasing $\cf(\lambda)$-sequence
  of cardinals, converging to $\lambda$, such that $\sd(\vec{\lambda})$ holds, and  $\tcf(\prod\vec\lambda,<^*)=\lambda^+$.
  Then $\reg(\lambda)\s \cspec(\lambda^+)$.
\end{cor}
\begin{proof} $\br$ By Remark~\ref{remark7}, $2^{\cf(\lambda)}<\lambda$.
So, by \cite[Theorem~5.29(1)]{paper35}, $\reg(\cf(\lambda))\s\cspec(\lambda^+)$.

$\br$ By \cite[Lemma~4.11]{paper35}, $\cf(\lambda)\in\cspec(\lambda^+)$.

$\br$ By Theorem~\ref{sdlambda} and \cite[Corollary~5.21]{paper35}, $\reg(\lambda)\setminus(\cf(\lambda)+1)\s \cspec(\lambda^+)$.
\end{proof}

Our next goal is to improve the following fact from Part~II,
and present a weaker sufficient condition for $\cspec(\lambda^+)$ to cover $\reg(\cf(\lambda))$.

\begin{fact}[{\cite[Theorem~5.29(2)]{paper35}}]\label{fact57}
Suppose that $\lambda$ is a singular cardinal of successor cofinality $\mu^+$.
Then:
\begin{itemize}
\item $\reg(\mu)\s\cspec(\lambda^+)$;
\item If $2^{\mu}\le\lambda$, then $\reg(\cf(\lambda))\s\cspec(\lambda^+)$.
\end{itemize}
\end{fact}

\begin{thm} Suppose that $\lambda$ is a singular cardinal whose cofinality $\nu$ is not greatly Mahlo.
Then, for every infinite regular $\theta\le\cf(\nu)$, there exists a closed witness to $\U(\lambda^+,\lambda^+,\theta,\cf(\lambda))$.
\end{thm}
\begin{proof} By \cite[Corollary~4.17]{paper34}, there exists a closed witness to $\U(\lambda^+,\lambda^+,\cf(\lambda),\allowbreak\cf(\lambda))$,
so assume that $\lambda$ has uncountable cofinality.
Recalling Claim~4.21.3 from the proof of \cite[Theorem~4.21]{paper34},
it suffices to prove that the ideal $\mathcal I$ defined there in ``Case 1: Uncountable cofinality''
is not weakly $\nu$-saturated.

Let us first remind the reader that the definition of the ideal $\mathcal I$ goes through first fixing a stationary subset $\Delta\s E^{\lambda^+}_{\cf(\lambda)}$
  and a sequence $\vec e=\langle e_\delta\mid \delta\in \Delta\rangle$ such that
  \begin{itemize}
    \item for every $\delta\in \Delta$, $e_\delta$ is a club in $\delta$ of order type $\cf(\lambda)$;
    \item for every $\delta\in \Delta$, $\langle \cf(\gamma)\mid \gamma\in \nacc(e_\delta)\rangle$
      is strictly increasing and converging to $\lambda$;
    \item for every club $D$ in $\lambda^+$, there exists $\delta\in \Delta$ such that $e_\delta\s D$.
  \end{itemize}
Then, the ideal $\mathcal I$ consists of all subsets $\Gamma\s\lambda^+$ for which
  there exists a club $D\s\lambda^+$ such that
  $\sup(\nacc(e_\delta)\cap D\cap \Gamma)<\delta$   for every $\delta\in \Delta\cap D$.

Now, since $\nu$ is not a greatly Mahlo cardinal,
by a theorem from \cite{paper47},
we may fix a coloring $c:[\nu]^2\rightarrow\nu$ satisfying that,
for every cofinal $B\s\nu$, there exist $(\eta,T)\in(\nu,[\nu]^\nu)$ such that,
for all $\tau\in T$, $\sup\{\beta\in B\mid c(\eta,\beta)=\tau\}=\nu$.

Fix a club $\Lambda$ in $\lambda$ of order-type $\nu$.
For every $\eta<\nu$ and $\tau<\nu$, let
$$\Gamma^\tau_\eta:=\{ \gamma<\lambda^+\mid c(\eta,\otp(\Lambda\cap\cf(\gamma))=\tau\}.$$

\begin{claim} There exists $\eta<\nu$ such that $|\{ \tau<\nu\mid \Gamma^\tau_\eta\in\mathcal I^+\}|=\nu$.
\end{claim}
\begin{cproof} Suppose not. Then, for every $\eta<\nu$, the set $T_\eta:=\{ \tau<\nu\mid \Gamma^\tau_\eta\in\mathcal I^+\}$ is bounded in $\nu$.
For each $\eta<\nu$ and $\tau\in\nu\setminus T_\eta$, fix a club $D_\eta^\tau\s\lambda^+$ such that
$\sup(\nacc(e_\delta)\cap D_\eta^\tau\cap \Gamma_\eta^\tau)<\delta$   for every $\delta\in \Delta\cap D^\eta_\tau$.
Let $D:=\bigcap\{ D_\eta^\tau\mid \eta<\nu, \tau\in \nu\setminus T_\eta\}$.
Now, using the choice of $\vec e$,
let us fix $\delta\in \Delta$ such that $e_\delta\s D$. In particular, $\delta\in D$.
As $\langle \cf(\gamma)\mid \gamma\in \nacc(e_\delta)\rangle$
is strictly increasing and converging to $\lambda$,
it follows that $B:=\{ \otp(\Lambda\cap\cf(\gamma))\mid \gamma\in\nacc(e_\delta)\}$ is cofinal in $\nu$.
Thus, by the choice of $c$, we may now find
$(\eta,T)\in(\nu,[\nu]^\nu)$ such that,
for all $\tau\in T$, $\sup\{\beta\in B\mid c(\eta,\beta)=\tau\}=\nu$.
Pick $\tau\in T\setminus T_\eta$.
As $\delta\in \Delta\cap D\s \Delta\cap D_\eta^\tau$,
we infer that $\sup(\nacc(e_\delta)\cap D_\eta^\tau\cap \Gamma_\eta^\tau)<\delta$.
In particular, $A:=\{\otp(\Lambda\cap\cf(\gamma))\mid \gamma\in \nacc(e_\delta)\cap D_\eta^\tau\cap \Gamma_\eta^\tau\}$ is bounded in $\nu$.
Pick $\beta\in B$ above $\sup(A\cup(\eta+1))$ such that $c(\eta,\beta)=\tau$.
Then, find $\gamma\in\nacc(e_\delta)$ such that $\beta=\otp(\Lambda\cap\cf(\gamma))$.
Altogether:
\begin{itemize}
\item $\gamma\in\nacc(e_\delta)$;
\item $\gamma\in e_\delta$, so that $\gamma\in D\s D^\tau_\eta$;
\item $c(\eta,\otp(\Lambda\cap\cf(\gamma))=c(\eta,\beta)=\tau$, so that $\gamma\in \Gamma^\tau_\eta$;
\item $\otp(\Lambda\cap\cf(\gamma))=\beta>\sup(A)$, so that $\gamma\notin \nacc(e_\delta)\cap D_\eta^\tau\cap \Gamma_\eta^\tau$.
\end{itemize}
This is a contradiction.
\end{cproof}

It follows that $\mathcal I$ is indeed not weakly $\nu$-saturated, so we are done.
\end{proof}

As successor cardinals are non-Mahlo, the following indeed improves Fact~\ref{fact57}.

\begin{cor} Suppose that $\lambda$ is a singular cardinal of cofinality $\nu$.
If $\nu$ is not greatly Mahlo, then $\reg(\cf(\lambda))\s\cspec(\lambda^+)$.
\end{cor}
\begin{proof} By \cite[Corollary~5.21]{paper35}, to show that an infinite regular cardinal $\theta$ is  in $\cspec(\lambda^+)$,
it suffices to prove that there exists a closed witness to $\U(\lambda^+,\lambda^+,\theta,\theta)$.
This is the content of the preceding theorem.
\end{proof}

\section{Stationarily layered posets} \label{stat_layered_sec}

In Part~I of this project \cite{paper34}, we motivated the study of
$\U(\kappa, \mu, \theta, \chi)$ by showing that it places limits on
the infinite productivity of the $\kappa$-Knaster condition. Here, we
present an analogous result, indicating that closed witnesses to
$\US(\kappa, 2,\allowbreak \theta, 2)$ place limits on the
infinite productivity of the property of being $\kappa$-stationarily layered,
which is a strengthening of the $\kappa$-Knaster condition.

\begin{defn}[\cite{cox_layered}]
  A partial order $\mathbb{P}$ is \emph{$\kappa$-stationarily layered} if the
  collection of regular suborders of $\mathbb{P}$ of size less than $\kappa$
  is stationary in $\mathcal{P}_\kappa(\mathbb{P})$.
\end{defn}

By \cite[Lemma 1.5]{MR3620068}, any $\kappa$-stationarily layered poset is
also $\kappa$-Knaster. We now recall a useful equivalence.

\begin{fact}[{\cite[Lemma 2.3]{MR3620068}}] \label{stat_layered_fact}
  Given a poset $\mathbb{P}$, the following are equivalent.
  \begin{enumerate}
    \item $\mathbb{P}$ is $\kappa$-stationarily layered.
    \item There is a regular cardinal $\Upsilon$ with $\mathcal{P}_\kappa(\mathbb{P})
      \in H_\Upsilon$ and an elementary substructure $M \prec H_\Upsilon$ such that
      \begin{enumerate}
        \item $\kappa,\mathbb{P}\in M$;
        \item $\kappa \cap M \in \kappa$;
        \item $\mathbb{P} \cap M$ is a regular suborder of $\mathbb{P}$.
      \end{enumerate}
  \end{enumerate}
\end{fact}

The following is Theorem~D:
\begin{thm} \label{stat_layered_thm}
  Suppose that $\theta \leq \chi < \kappa$ are infinite, regular cardinals,
  $\kappa$ is $({<}\chi)$-inaccessible, and there exists a closed witness $c$ to $\U(\kappa,2,\theta,2)$ that is subadditive of the second kind. Then:
  \begin{itemize}
  \item There is a sequence of posets $\langle\mathbb{P}_i \mid  i < \theta\rangle$ such that:
  \begin{enumerate}
    \item for all $i < \theta$, $\mathbb{P}_i$ is well-met and $\chi$-directed
      closed with greatest lower bounds;
    \item for all $j < \theta$, $\prod_{i < j} \mathbb{P}_i$ is $\kappa$-stationarily layered;
    \item $\prod_{i < \theta} \mathbb{P}_i$ is not $\kappa$-cc.
  \end{enumerate}
  \item If $\partial(c)\cap E^\kappa_\chi$ is stationary, then there is a poset $\mathbb P$ such that:
  \begin{enumerate}
    \item $\mathbb{P}$ is well-met and $\chi$-directed closed with greatest lower bounds;
    \item for all $j < \theta$, $\mathbb{P}^j$ is $\kappa$-stationarily layered;
    \item $\mathbb{P}^\theta$ is not $\kappa$-cc.
  \end{enumerate}

  \end{itemize}
\end{thm}

\begin{proof} Using Lemma~\ref{lemma24}(3) (with $S:=E^\kappa_{\chi}$),
let us fix $\epsilon<\kappa$ such that, for every $j<\theta$, $\{\beta\in E^\kappa_{\chi}\mid c(\epsilon,\beta)>j\}$ is stationary.
For every $i<\theta$, let
\begin{itemize}
\item $\Gamma_i:=\{\gamma< \kappa\mid \epsilon < \gamma \text{ and } c(\epsilon,\gamma)\le i\}$,
\item $P_i$ denote the collection of all pairs
$(i,f)$ where $f:\kappa\rightarrow 2$ is a partial function of size less than $\chi$.
\end{itemize}

Set $P:=\{\emptyset\}\cup \biguplus_{i<\theta}P_i$.
Define an ordering $\le$ of $P$ as follows:
\begin{enumerate}
\item $\emptyset$ is the top element of $(P,\le)$.
\item For all $(i,f),(j,g)\in P\setminus\{\emptyset\}$,
we let $(j,g)\le (i,f)$  iff
  \begin{itemize}
    \item $i=j$,
    \item $g \supseteq f$, and
    \item for all $\gamma \in \dom(f) \cap \Gamma_i$ and
    $\alpha \in \gamma \cap \dom(g)\setminus \dom(f)$, if $c(\alpha,\gamma)\le i$,
     then $g(\alpha) = 0$.
  \end{itemize}
\end{enumerate}

We then let $\mathbb P:=(P,\le)$ and $\mathbb P_i:=(\{\emptyset\}\cup P_i,\le)$ for all $i<\theta$.
\begin{claim} $\le$ is transitive.
\end{claim}
  \begin{cproof} Suppose $(i,h)\le (i,g)$ and $(i,g)\le (i,f)$.
  It is clear that $h\supseteq f$. Let
      $\gamma \in \dom(f) \cap \Gamma_i$ and
    $\alpha \in \gamma \cap \dom(h)\setminus \dom(f)$ with $c(\alpha,\gamma)\le i$.
    Clearly, $\gamma\in \dom(g)\cap \Gamma_i$.

$\br$ If $\alpha\notin\dom(g)$, then from $(i,h)\le (i,g)$, it follows that $h(\alpha) = 0$, as sought.

$\br$ If $\alpha\in\dom(g)$, then from $(i,g)\le (i,f)$, it follows that $h(\alpha) = g(\alpha)= 0$, as sought.
  \end{cproof}
 Clearly, $\mathbb P$ and each of the $\mathbb P_i$'s is a well-met poset which is $\chi$-directed closed with greatest
  lower bounds.

  \begin{claim}\label{claim333}
  \begin{enumerate}
    \item $\prod_{i < \theta} \mathbb{P}_i$ does not satisfy the $\kappa$-cc;
	\item $\mathbb{P}^\theta$ does not satisfy the $\kappa$-cc.
	\end{enumerate}
  \end{claim}

  \begin{cproof} (1)     For every $\alpha<\kappa$, define $p_\alpha\in\prod_{i < \theta}\mathbb{P}_i$
    by setting $p_\alpha(i):=(i,\{(\alpha, 1)\})$ for all $i<\theta$.
    To see that $\{p_\alpha\mid \alpha<\kappa\}$ is
    an antichain (of size $\kappa$) in $\prod_{i < \theta}\mathbb{P}_i$,
    fix arbitrary $(\alpha,\gamma)\in[\kappa]^2$.
    Let $i:=\max\{c(\epsilon,\gamma),c(\alpha,\gamma)\}$.
    Then $p_\alpha(i)$ and $p_\gamma(i)$
    are incompatible in $\mathbb{P}_i$, so $p_\alpha$ and $p_\gamma$ are incompatible
    in $\prod_{i < \theta}\mathbb{P}_i$.

    (2) This follows from Clause~(1).
  \end{cproof}

  \begin{claim} Let $j < \theta$.
  \begin{enumerate}
  \item $\prod_{i < j} \mathbb{P}_i$ is $\kappa$-stationarily layered;
  \item If $\partial(c)\cap E^\kappa_\chi$ is stationary, then $\mathbb P^j$  is $\kappa$-stationarily layered.
  \end{enumerate}
  \end{claim}

  \begin{proof} \renewcommand{\qedsymbol}{\ensuremath{\boxtimes \ \square}}
    Let $\mathbb{Q}$ denote $\prod_{i < j} \mathbb{P}_i$ (resp.~$\mathbb P^j$).
    Let $\Upsilon$ be a regular cardinal with $\mathcal{P}_\kappa(\mathbb{Q}) \in H_\Upsilon$.
    Using the fact that $\kappa$ is $({<}\chi)$-inaccessible, construct an $\in$-increasing, continuous
    sequence $\vec M=\langle M_\beta \mid \beta < \kappa \rangle$ such that,
    for all $\beta < \kappa$,
    \begin{itemize}
      \item $M_\beta \prec H_\Upsilon$;
      \item $|M_\beta| < \kappa$;
      \item $\kappa,\mathbb{Q} \in M_\beta$;
      \item ${^{{<}\chi}}M_{\beta} \s M_{\beta + 1}$.
    \end{itemize}

    $\br$ If $\mathbb{Q}=\prod_{i < j} \mathbb{P}_i$, then,
    using our choice of $\epsilon$, we fix $\beta\in E^\kappa_{\chi}$
    with $\kappa\cap M_\beta=\beta>\epsilon$
    such that $c(\epsilon,\beta)>j$.

    $\br$ If $\partial(c)\cap E^\kappa_\chi$ is stationary, then we fix $\beta \in \partial(c)\cap E^\kappa_\chi$
    with $\kappa\cap M_\beta=\beta>\epsilon$.

    Note that by the continuity of the sequence $\vec M$,
     ${^{{<}\chi}}M_\beta \s M_\beta$.
    Now, by Fact~\ref{stat_layered_fact}, it suffices to prove
    that $\mathbb{Q} \cap M_\beta$ is a regular suborder of $\mathbb{Q}$.
    To this end, we will define, for each $p \in \mathbb{Q}$, a \emph{reduction} of $p$ to
    $M_\beta$, i.e., a condition $p \vert M_\beta \in \mathbb{Q} \cap M_\beta$
    such that, for all $q \leq_{\mathbb{Q}} p \vert M_\beta$ with $q \in M_\beta$,
    $q$ is compatible with $p$.

    Fix $p \in \mathbb{Q}$. For every $\eta < j$,
    write $p(\eta)$ as $(i_\eta,f_\eta)$,
    and set $x_\eta := \dom(f_\eta) \cap \Gamma_{i_\eta} \setminus \beta$.
    \begin{subclaim} Let $\eta<j$ and $\gamma\in x_\eta$.
    Then $D^c_{\leq i_\eta}(\gamma)\cap\beta$ is a closed bounded subset of $\beta$.
    \end{subclaim}
    \begin{scproof}
    As $c$ is closed, it suffices to prove that $\sup(D^c_{\leq i_\eta}(\gamma)\cap\beta)<\beta$.
    To avoid trivialities, suppose that $\beta\notin\partial(c)$.
    So, $\mathbb{Q}=\prod_{i < j} \mathbb{P}_i$ and $i_\eta=\eta$.
    Now, as $c$ is subadditive of the second kind,
    $$j<c(\epsilon,\beta)\le\max\{c(\epsilon,\gamma),c(\beta,\gamma)\}.$$
    So, since $c(\epsilon,\gamma)\le i_\eta=\eta<j$,
     it follows that $c(\beta,\gamma)\ge j>i_\eta$.
    As $c$ is closed, it thus follows that
    $\{\alpha < \beta\mid c(\alpha, \gamma) \leq i_\eta\}$ is bounded below $\beta$.
    \end{scproof}

    For every $\eta<j$, by the subclaim and since $|x_\eta|<\chi$,
    $$
      y_\eta := \{\max(D^c_{\leq i_\eta}(\gamma)\cap\beta)\mid \gamma \in x_\eta \ \&\ D^c_{\leq i_\eta}(\gamma)\cap\beta\neq\emptyset\}
    $$
    is a well-defined element of $[\beta]^{<\chi}$.

    For each $\eta < j$, define a partial function $g_\eta:\beta\rightarrow 2$ by letting $$g_\eta:=(f_\eta\restriction \beta)\cup((y_\eta\setminus\dom(f_\eta))\times\{0\}).$$

    Clearly, $g_\eta$ has size less than $\chi$, so as ${^{{<}\chi}}M_\beta \s M_\beta$, we have $(i_\eta,g_\eta)\in M_\beta\cap  P_{i_\eta}$.
    Define a condition $p \vert M_\beta$ in $\mathbb Q$ by letting $(p \vert M_\beta)(\eta) := (i_\eta,g_\eta)$ for all $\eta<j$.

    As ${^{{<}\theta}}M_\beta \s M_\beta$, we have $p \vert M_\beta\in M_\beta\cap\mathbb Q$.
    To see that $p \vert M_\beta$ is a reduction of $p$ to $M_\beta$, fix
    $q \leq_{\mathbb{Q}} p \vert M_\beta$ in $M_\beta$. The only way
    that $q$ can fail to be compatible with $p$ is if, for some $\eta < j$,
    there are $\gamma \in x_\eta$ and $\alpha\in \dom(q(\eta)) \setminus \dom(f_\eta)$ such
    that $c(\alpha,\gamma)\le i_\eta$, but $q(\eta)(\alpha) = 1$.
    So suppose $\eta$, $\alpha$ and $\gamma$ are as described. In particular,
    $\alpha$ is smaller than $\gamma_\eta:=\max(D^c_{\leq i_\eta}(\gamma)\cap\beta)$.
	As $\gamma_\eta\in\dom(g_\eta)$,
	it thus follows that either $\gamma_\eta\notin  \Gamma_{i_\eta}$ or $c(\alpha,\gamma_\eta)> i_\eta$.
	But $c$ is subadditive of the second kind, so $c(\epsilon,\gamma_\eta)\le\max\{c(\epsilon,\gamma),c(\gamma_\eta,\gamma)\}\le i_\eta$
	and  $c(\alpha,\gamma_\eta)\le\max\{c(\alpha,\gamma),c(\gamma_\eta,\gamma)\}\le i_\eta$.
	This is a contradiction.
  \end{proof}

  \let\qed\relax
\end{proof}

\begin{cor}\label{coro64} Suppose that $\kappa$ is $({<}\theta)$-inaccessible and any one of the
  following six statements holds.
\begin{enumerate}
\item $\kappa=\lambda^+$ and $\square_\lambda$ holds.
\item $V = L$ and $\kappa$ is not weakly compact.
\item $\square(\kappa)$ holds after being added generically.
\item There exists a $\square(\kappa)$-sequence that avoids a stationary subset of $E^\kappa_\theta$.
\item $\kappa=\lambda^+$ where $\lambda$ is a former inaccessible that changed its cofinality to $\theta$ via Prikry/Magidor forcing.
\item $\theta=\omega$ and $\square^{\ind}(\kappa,\omega)$ holds after being added generically.
\end{enumerate}
Then there is a poset $\mathbb P$ such that:
  \begin{itemize}
    \item $\mathbb{P}$ is well-met and $\theta$-directed closed with greatest lower bounds;
    \item for all $j < \theta$, $\mathbb{P}^j$ is $\kappa$-stationarily layered;
    \item $\mathbb{P}^\theta$ is not $\kappa$-cc.
  \end{itemize}
\end{cor}
\begin{proof}
  We will show that any one of the assumptions entails the existence of a
  closed witness $c$ to $\US(\kappa, 2, \theta, 2)$ such that $\partial(c)$ is
  stationary. Then the second bullet point of Theorem~\ref{stat_layered_thm}
  will furnish the desired poset $\mathbb{P}$.

  Note first that $(1) \implies (4)$ and $(2)\implies(4)$ (cf.\ \cite[Theorem~6.1]{jensen} or \cite[Theorem VII.1.2']{devlin}).
  In addition, as established in the proof of Theorem~\ref{add_inds},
  $(3)\implies(4)$.
  Next, by Theorem~\ref{revised_lh_lucke}(2), if (4) holds, then there is a
  closed witness $c$ to $\US(\kappa, 2, \theta, 2)$ for which $\partial(c)$ is
  stationary.

  The fact that Clause (5) implies the desired conclusion is a consequence of Theorem~\ref{outer_model_thm}
  and the fact that Prikry forcing and Magidor forcing
  do not kill the stationarity of the ground model's $E^{\lambda^+}_\lambda$.

  It remains to deal with (6). To this end, suppose
  that $\vec{C} = \langle C_{\alpha, i} \mid \alpha \in \acc(\kappa),
  i(\alpha) \leq i < \omega \rangle$ is a generically-added
  $\square^{\mathrm{ind}}(\kappa, \omega)$-sequence, i.e., $V$ is an extension
  of some ground model by the forcing from \cite[\S 7]{narrow_systems} to
  add a $\square^{\mathrm{ind}}(\kappa, \omega)$-sequence, and $\vec{C}$ is the
  sequence generated by the generic filter. Let
  \[
    S:=\{\alpha \in E^\kappa_\omega \mid \forall i \in [i(\alpha), \omega)
    [\sup(\acc(C_\alpha, i)) < \alpha]\}.
  \]
  By \cite[Claim 3.4.1]{paper35}, $S$ is stationary.
  Let $\Gamma := \acc(\kappa) \setminus S$. We now define an
  $\inds(\kappa, \omega)$-sequence $\vec{D} := \langle D_{\alpha, i} \mid
  \alpha \in \Gamma, ~ j(\alpha) \leq i < \omega\rangle$ as follows. For
  each $\alpha \in S$, let $j(\alpha)$ be the least $j \in [i(\alpha), \omega)$
  such that $\sup(\acc(C_\alpha, i)) = \alpha$ and, for all $j \in [j(\alpha), \omega)$,
  let $D_{\alpha, j} := \acc(C_{\alpha, j})$. Using our choice of $S$ and
  the fact that $\vec{C}$ is an $\square^{\ind}(\kappa, \omega)$ sequence, it is
  straightforward to verify that $\vec{D}$ is an $\inds(\kappa, \omega)$-sequence.
  By the proof of Theorem~\ref{thm67} (cf.\ also Theorem~\ref{revised_lh_lucke})
  and the fact that $\Gamma \cap S = \emptyset$, it follows that there is a
  closed witness $c$ to $\US(\kappa, 2, \omega, 2)$ such that $S \subseteq \partial(c)$.
\end{proof}

\section*{Acknowledgments}
The second author was partially supported by the Israel Science Foundation (grant agreement 2066/18)
and by the European Research Council (grant agreement ERC-2018-StG 802756).

The results of this paper were presented by the first author at the \emph{Kobe Set Theory Workshop 2021}, March 2021.
Preliminary results of this paper were presented by the second author at the \emph{15th International Luminy Workshop in Set Theory}, September 2019.
We thank the organizers for the invitations.

\end{document}